\documentclass[10pt,reqno]{amsart}
\usepackage[margin=0.8in]{geometry}
\usepackage{amsthm, amsmath,amsfonts,amssymb,euscript,graphics,color,slashed,bm}
\usepackage{graphicx}
\usepackage{mathrsfs}
\usepackage{comment}
\usepackage{import}
\usepackage{latexsym}
\usepackage[makeroom]{cancel}

\def\norm#1#2{\|#1\|_{#2}}

\def\inte#1{
	\displaystyle\mathop{#1\kern0pt}^\circ }

\let\al=\alpha

\let\e=\varepsilon

\let\lam=\lambda

\let\f=\frac

\let\p=\psi

\let\Lam=\Lambda

\let\wt=\widetilde

\def\cE{{\mathcal E}}
\def\cF{{\mathcal F}}

\def\cL{{\mathcal L}}

\def\cN{{\mathcal N}}

\def\cP{{\mathcal P}}

\def\cS{{\mathcal S}}
\def\cT{{\mathcal T}}

\def\virgp{\raise 2pt\hbox{,}}
\def\cdotpv{\raise 2pt\hbox{;}}

\def\eqdefa{\buildrel\hbox{\footnotesize def}\over =}

\def\C{\mathop{\mathbb C\kern 0pt}\nolimits}
\def\DD{\mathop{\mathbb D\kern 0pt}\nolimits}
\def\EE{\mathop{{\mathbb E \kern 0pt}}\nolimits}
\def\K{\mathop{\mathbb K\kern 0pt}\nolimits}
\def\N{\mathop{\mathbb N\kern 0pt}\nolimits}
\def\Q{\mathop{\mathbb Q\kern 0pt}\nolimits}
\def\R{\mathop{\mathbb R\kern 0pt}\nolimits}
\def\SS{\mathop{\mathbb S\kern 0pt}\nolimits}
\def\ZZ{\mathop{\mathbb Z\kern 0pt}\nolimits}
\def\TT{\mathop{\mathbb T\kern 0pt}\nolimits}
\def\P{\mathop{\mathbb P\kern 0pt}\nolimits}

\newcommand{\Z}{{\ZZ}}

\def\na{\nabla}
\def\p{\partial}

\newcommand{\beq}{\begin{equation}}
	\newcommand{\eeq}{\end{equation}}
\newcommand{\ben}{\begin{eqnarray}}
	\newcommand{\een}{\end{eqnarray}}
\newcommand{\beno}{\begin{eqnarray*}}
	\newcommand{\eeno}{\end{eqnarray*}}

\newtheorem{defi}{Definition}[section]

\newtheorem{lem}{Lemma}[section]

\newtheorem{prop}{Proposition}[section]

\newcommand{\vv}[1]{\boldsymbol{#1}}

\newtheorem*{Main Theorem}{Main Theorem}
\newtheorem{theorem}{Theorem}[section]
\newtheorem{lemma}[theorem]{Lemma}

\newtheorem{remark}[theorem]{Remark}

\numberwithin{equation}{section}

\begin{document}
	\title[Well-posedness]{Long time existence for a Boussinesq-like system with strong topography variations}

	\author[Qi Li]{Qi Li}
\address{School of Mathematical Sciences, Beihang University\\  100191 Beijing, China}
\email{Ricci@buaa.edu.cn}

\author{Jean-Claude Saut}
\address{Laboratoire de Math\' ematiques, UMR 8628\\
Universit\' e Paris-Saclay et CNRS\\ 91405 Orsay, France}
\email{jean-claude.saut@universite-paris-saclay.fr}

\author[Li XU]{Li Xu}
\address{School of Mathematical Sciences, Beihang University\\  100191 Beijing, China}
\email{xuliice@buaa.edu.cn}
	
	\maketitle
	
\textit{Abstract}.
 This paper investigates the long-time existence of solutions for a Boussinesq-like system modeling surface water waves in shallow water  with {\it strong bottom topography variations} (i.e. $b=O(1)$ where $b$ measures the bathymetry).  The system is derived under the {\it long-wavelength, small-amplitude} regime, with $h=1-b$ represents the mean water depth. 
  The strong bottom topography may exhibit {\it slow ($\na h=O(\e)$) or {\it fast} ($\na h=O(1)$) oscillations}, where the small parameter $\e$ measures the comparable long wave and weak nonlinearity effects.  We establish the long time existence results in two typical regimes:

 \begin{itemize}
\item[1).]\textbf{Slow oscillation ($\na h=O(\e)$):} For both $1D$ and $2D$ cases, we prove the local existence on time scale $O(1/\e)$ under some restrictions on the parameters  ensuring the symmetry of the main dispersive terms. To overcome key challenges arising from $h$-dependent coefficients and the incomplete control of $\na\vv V$ in $2D$, we design  adapted energy functionals and exploit the cancellation for lower-order linear terms in $2D$.
\item[2).]\textbf{Fast oscillation ($\na h=O(1)$):} In $1D$, we extend the natural $O(1)$ existence time  to $O(1/\e)$ for a special 
case maintaining the symmetric principal dispersive structure. The proof uses a strategy that first establishes time regularity before recovering spatial derivatives, supported by carefully designed energy estimates and cancellation of lower-order terms.
\end{itemize} 

These results provide the first rigorous long-time existence theory for Boussinesq-like system with strong topography variations in both $1D$ and $2D$,
 extending significantly  previous work limited to fully symmetric cases.

	\vspace{0.5cm}
	
	\setcounter{tocdepth}{1}
	
	\section{Introduction}
	
	\subsection{The Boussinesq-like system with strong bottom topography variations}
	
	In this paper, we study  asymptotic models for gravity surface water waves  in the {\it long wavelength, small amplitude, strong bottom topography variations} regime. More precisely, in this regime, denoting by $\lambda$ a typical wavelength, $a$ a typical amplitude of the surface waves, $h_0$ the typical depth of the fluid, $b_0$ the typical amplitude of the bottom, one assumes that 
	\beno
    \f{a}{h_0}\sim\bigl(\f{h_0}{\lambda}\bigr)^2\ll1,\quad\beta\eqdefa\f{b_0}{h_0}\sim 1
	\eeno

	For simplicity, we will assume from now on that
	\beq\label{regime 2}
\f{a}{h_0}=\bigl(\f{h_0}{\lambda}\bigr)^2=\e\ll1,\quad\beta=1.
	\eeq
	The relation $\beta=O(1)$ corresponds to bottoms with large variations in amplitude so that the asymptotic models for this case have quite different structure from that for the mild case $\beta=O(\e)$. The asymptotic  system for the case $\beta=O(\e)$ is a slight modification  of the classical {\it abcd}-Boussinesq systems, see \cite{BCL, BCS1, BCS2} and it was studied in \cite{FC} where one finds in particular "optimal" error estimates for a 
	class of fully symmetric systems in the spirit of \cite{BCL} for the {\it abcd-} Boussinesq systems.
	
	\vspace{0.3cm}

	We   consider the following Boussinesq-like system,   introduced in \cite{FC}:
\begin{equation}\label{Boussinesq-like system}
    	\left\{   \begin{aligned}
    		&(1-\f\e2 \mathcal{P}_h^1) \p_t \vv V   + \sqrt{h} \nabla \eta + \f\e2 \bigl[ {\vv F}_h(\vv V,\eta) + b_1 \sqrt{h} \nabla \nabla \cdot (h^2 \nabla \eta) + b_2 \sqrt{h} \nabla (  h \nabla h \cdot \nabla \eta ) \\ 
    		&\qquad\qquad\qquad+ b_3 \nabla h \nabla \cdot (h\sqrt{h} \nabla\eta) + b_4 \sqrt{h} \nabla h (\nabla h \cdot \nabla \eta) \bigr] =0, \\
    		&(1-\f\e2 \mathcal{P}_h^2) \partial_t \eta + \nabla\cdot  (\sqrt{h} \vv V ) + \f\e2 \bigl\{f_h(\vv V,\eta) + \nabla \cdot \bigl[  c_1 h^2 \nabla \nabla \cdot (\sqrt{h} \vv V) + c_2 h \nabla h \nabla \cdot (\sqrt{h}\vv V)  \\
    		&\qquad\qquad\qquad+ c_3 h \sqrt{h} \nabla (\nabla h \cdot \vv V ) + c_4 \sqrt{h} \nabla h (\nabla h \cdot \vv V)   \bigr] \bigr\} =0,
    	\end{aligned}\right.   
    \end{equation}
    where  $\vv V=\vv V(t,x)\, (x\in\R^n,\,n=1,2)$ is an $O(\e^2)$ approximation of the horizontal velocity, $\eta=\eta(t,x)$ is the deviation of the free surface from the rest state,  $h=h(x)=1-b(x)\in [0,1]$ is the still water depth, and $b=b(x)$ is the amplitude of the bottom.  In \eqref{Boussinesq-like system},  the operators $\mathcal{P}_h^1$ and $\mathcal{P}_h^2$ are defined by
    \begin{equation}\label{def of Ph}
    	\left\{ \begin{aligned}
    		&\mathcal{P}_h^1 =  a_1  \nabla (h^2 \nabla\cdot\quad)  + a_2  \nabla (h \nabla h \cdot\quad) , \\
    		&\mathcal{P}_h^2 = d_1   \nabla \cdot (h^2 \nabla\quad) +  d_2 \nabla\cdot (h\nabla h\times \quad) ,
    	\end{aligned}\right. 
    \end{equation}
   while  the nonlinear terms ${\vv F}_h(\vv V,\eta)$ and $f_h(\vv V,\eta)$ are defined by
    \begin{equation}\label{def of Fh and fh}
    	\left\{ \begin{aligned}
    		&{\vv F}_h(\vv V,\eta) = \frac{1}{\sqrt{h}} \left[ \eta \nabla \eta + \f12 \nabla |\vv V|^2 + \vv V \cdot \nabla \vv V + \vv V \nabla\cdot \vv V + \f1h \left( \f12 (\vv V\cdot\nabla h  ) \vv V - |\vv V|^2 \nabla h \right) \right],\\
    		&f_h(\vv V,\eta) = \frac{1}{\sqrt{h}} \left( \nabla \cdot (\eta \vv V) - \frac{\eta}{2h} \nabla h \cdot \vv V  \right),
    	\end{aligned}  \right.
    \end{equation}
    where the parameters $\{ a_j,d_j \}_{j=1,2},\{ b_j,c_j \}_{j=1,2,3,4}$ have the following expressions\footnote {Based on the derivations in \cite{FC}, the parameter $c_4$ should be $ -\f12 (\theta-2)^2$, rather than  $\f12 (\theta-2)^2$ as stated in \cite{FC}.}:
     \begin{align} \label{expression of a,b,c,d}
    	\begin{cases}
    		 a_1 = (1-\theta^2)(1-\lam_1),\\
    		 a_2 = 2(1-\theta)(1-\lam_2),\\
    		 d_1 =(1-\mu)(\theta^2 -\f13),\\
    		 d_2 = (1-\mu)(\f32 \theta^2 - \f76),
    	\end{cases}\,
    	\begin{cases}
    		b_1 = \lam_1 (1-\theta^2),\\
    		b_2 = 2\lam_2(1-\theta) - \f32 \lam_1 (1-\theta^2),\\
    	b_3 = \f12 \lam_1 (1-\theta^2) , \\
    		b_4 = \lam_2(1-\theta)  -\f12 \lam_1(1-\theta^2) ,
    	\end{cases}\, \begin{cases}
    	 c_1 = \mu(\theta^2 - \f13) , \\
    	 c_2 = \mu (\f32 \theta^2 -\f76) , \\
     	c_3 = -\f12 \theta^2 + 2\theta - \f76, \\
    	c_4 = -\f12 (\theta-2)^2,
    	\end{cases}
    \end{align}
    and the BBM parameters  $(\lam_1,\lam_2,\mu,\theta)\in \R^3 \times [0,1]$.
    
    We remark that the parameters $(a_1,b_1,c_1,d_1)$ satisfy the constraint
    \begin{align*}
    	a_1  + b_1 + c_1 + d_1 = \f23,
    \end{align*} 
    which is similar to the  "$(a,b,c,d)$"-Boussinesq system's constraint \footnote{The dispersive parameters $(a_1,b_1,c_1,d_1)$ of system \eqref{Boussinesq-like system} should be multiplied by $\f12$.}. More precisely, if the bottom is flat (i.e., $b=0$), the linearization of \eqref{Boussinesq-like system} around the null solutions $(\vv 0,0)$ is 
    \begin{align}\label{Boussinesq-like system, flat bottom}
		\begin{cases}
			(1-\f\e2 a_1 \na\na\cdot) \p_t \vv V + \nabla \eta + \frac{\e}{2} b_1 \nabla \Delta \eta =0, \\
			(1-\f\e2 d_1 \Delta) \p_t \eta +  \nabla \cdot \vv V  + \frac{\e }{2} c_1 \nabla \cdot \Delta \vv V =0,
		\end{cases}  
	\end{align}
	where $\f{a_1}{2} +\f{b_1}{2} + \f{c_1}{2} +\f{ d_1}{2} = \f13$.  

	Notice that \eqref{Boussinesq-like system, flat bottom} is similar to the linearization of the {\it abcd} Boussinesq system at $(\vv 0,0)$ and  has the same  eigenvalues :
	\beno
\lambda_{\pm}=\pm i |\xi|\Bigl(\f{(1-\f\e2 b_1|\xi|^2)(1-\f\e2 c_1|\xi|^2)}{(1+\f\e2 a_1|\xi|^2)(1+\f\e2 d_1|\xi|^2)}\Bigr)^{\f12}.
	\eeno
This shows that the linearization \eqref{Boussinesq-like system, flat bottom} is well-posed provided that
	\begin{align}\label{constraint of Boussinesq}
		\begin{aligned}
			 &a_1\ge 0,\quad d_1 \ge 0,\quad b_1 \le 0,\quad c_1 \le 0, \\
			 \text{or}\quad & a_1\ge 0,\quad d_1 \ge 0,\quad b_1 = c_1 >0.
		\end{aligned}
	\end{align}

    In this paper, we assume that the parameters defined in \eqref{expression of a,b,c,d} satisfy \eqref{constraint of Boussinesq}.
    
     Regarding the dispersive terms appearing in \eqref{Boussinesq-like system}  as the perturbation of the hyperbolic system, the standard energy method suggests that the Cauchy problem for \eqref{Boussinesq-like system} might be locally well-posed on a time scale of $O(\f{1}{|\na h|})$, where $h=1-b$. This assertion was stated in Proposition 3.4 of \cite{FC} for a fully symmetric version of \eqref{Boussinesq-like system}. 
     
     Based on this observation,  the analysis can be split into the following two regimes:
\beno
    \na h=O(\e)\quad\text{and}\quad \na h=O(1),
    \eeno
which corresponds to the slow and fast oscillation of the uneven bottom respectively.  

\vspace{0.3cm}
As for other asymptotic models, such as the classical  Boussinesq systems, the local well-posedness of the Cauchy problem is not sufficient to rigorously justify the model. One needs to establish the {\it long time existence} of solutions, that is on  time scales on which the asymptotic model is supposed to provide a good approximation of the original system in the relevant regime. This problem has been solved for the {\it abcd} Boussinesq systems in \cite{CB,SX1, SX2, SX3, SWX}.

\vspace{0.3cm}
In \cite{FC}, F. Chazel proved the long time existence result on a time scale of $O(\f{1}{\e})$ for \eqref{Boussinesq-like system} in the fully symmetric one-dimensional case \footnote{In the two-dimensional version of \eqref{Boussinesq-like system}, the parameters in  \eqref{expression of a,b,c,d} cannot allow to obtain fully symmetric versions of  \eqref{Boussinesq-like system}. } when $\nabla h = O(\e)$, as well as the local existence theory on  a time scale of $O(1)$ for $\nabla h =O(1)$. To the best of our knowledge, besides the mentioned results in \cite{FC}, there are no other long time results for system \eqref{Boussinesq-like system} in both cases where $\nabla h = O(\e)$ and $\nabla h = O(1)$. In particular, there is no long time result available even in one dimension when $\nabla h = O(1)$.

The goal of this paper is to investigate the long time existence theory on a time scale of $O(\f{1}{\e})$ for the Cauchy problem of \eqref{Boussinesq-like system} in the following two cases:
\begin{itemize}
\item {\bf  Case $\na h=O(\e)$: }
\beq\label{case for slow oscillation}
\begin{aligned}
&\text{for}\,n=1:a_1>0,\,d_1>0,\,b_1=c_1,\,(b_j,c_j)_{j=2,3,4}\in\R^6,\,(a_2,d_2)\in\R^2;\\
&\text{for}\,n=2: a_1>0,\,d_1>0,\,a_2=0,\,d_2\in\R,\,b_1=c_1,\,b_3=-c_3,\,(b_j,c_j)_{j=2,4}\in\R^4;
\end{aligned}\eeq
\item {\bf Case $\na h=O(1)$ and $n=1$: }
\beq\label{case for fast oscillation}
a_1>0,\,a_2=0,\,d_1=d_2=0,\,b_1=c_1<0,\,b_2+c_2+b_3+c_3=0,\,(b_4,d_4)\in\R^2.
\eeq
\end{itemize}
We remark that the restrictions \eqref{case for slow oscillation} and \eqref{case for fast oscillation} for the parameters can be achieved by choosing suitable $(\lambda_1,\lambda_2,\mu,\theta)\in\R^3\times [0,1]$ in \eqref{expression of a,b,c,d}. 

\vspace{0.3cm}  

Let us give a short review of the known results for the asymptotic models in the long wave regimes. We recall that for the {\it abcd}  Boussinesq systems, the long time existence results were established for all generic cases  in \cite{BCL, SX1, SWX} by using  symmetrization techniques, in \cite{CB} by energy methods, and for the strongly dispersive cases in \cite{SX2,SX3} by using normal form techniques. 

On the other hand,  global existence of arbitrary large solutions  were obtained for the 1D Amick-Schonbek system (that is a {\it abcd} Boussinesq system with $a=b=c=0, d=1$)   in \cite{Amick, Schonbek} (see \cite{MTZ} for optimal results) and for some Hamiltonian cases in \cite{BCS2, Hu} for small initial data. The asymptotic dynamics of global small solutions in the one-dimensional Hamiltonian {\it abcd} systems is described in \cite{KM, KMPP} and in \cite{MMP} for the {\it abcd} Boussinesq model with an uneven bottom derived  in \cite{Chen}.

We also mention \cite{Ang} where global  existence is proven close  to a solitary wave for the so-called Kaup-Broer-Kuperschmidt system.  We refer to the book \cite{Kl-Sa} for an extensive discussion of results for the Boussinesq or Boussinesq-like systems.

In contrast, there are far fewer results for asymptotic models in long-wave regimes with uneven bottoms, particularly for strong topographic variations. Peregrine \cite{Peregrine} first derived a Boussinesq system for surface water waves in shallow water over variable topography. The Boussinesq-Peregrine system was later rigorously justified in \cite{Lannes, Lannes-Bonneton}. Local existence was proven on short time intervals for the Boussinesq-Peregrine system, and on a long time scale $O(1/\e)$
 for a modified $1D$ version with large topographic variations, in \cite{BMG}. Recently, a global existence result was obtained in \cite{Mol-Tal} for $1D$ Boussinesq-Peregrine system under small bottom variations that can be viewed as the Amick-Schonbek system with variable bottom .

Beyond the Boussinesq-Peregrine system, a new Boussinesq-like system was formally derived in \cite{Madsen, Nwogu} and rigorously in \cite{ FC}. Another asymptotic model  in the case of  slowly varying bottoms is  derived  in \cite{Iguchi}.

 For large bottom variations, no long-time existence results exist beyond those in \cite{FC}.

	\subsection{Main results}
	In this subsection, we will present our main results regarding cases \eqref{case for slow oscillation} and \eqref{case for fast oscillation} for \eqref{Boussinesq-like system} with the following initial data
	\beq\label{initial data}
\vv V|_{t=0}=\vv V_0,\quad\eta|_{t=0}=\eta_0.
	\eeq

	 Before  going further, we introduce the following functional spaces.
	
	\begin{defi} Let $s \ge 0$, $k\in\N$, $\e>0$. 
		\begin{enumerate}
			\item (Scalar Function space) The Banach space $X_{\e^k}^s(\R^n)$ is defined as $ H^{s+k}(\R^n)$ equipped with the norm:
			\begin{equation*}
				\|u\|_{X_{\e^k}^s}^2 \eqdefa \|u\|_{H^s}^2 + \e^k \|u\|_{H^{s+k}}^2,\quad\forall\, u\in H^{s+k}(\R^n) .
			\end{equation*}
			\item (Vector Function space) The Banach space ${\bf X}_{\e^k}^s(\R^n)$ is defined as 
			\begin{equation*}
			{\bf X}_{\e^k}^s(\R^n)\eqdefa\Bigl\{\vv U\in \bigl(H^{s}(\R^n)\bigr)^n\,\Big|\,	\|\vv U\|_{{\bf X}_{\e^k}^s}^2 \eqdefa \|\vv U\|_{H^s}^2 + \e^k \|\nabla \cdot \vv U\|_{H^{s+k-1}}^2<+\infty\Bigr\}.
			\end{equation*}
		\end{enumerate}
	\end{defi}
	\begin{remark}
For $n=1$, we remark that ${\bf X}_{\e^k}^s(\R)=X_{\e^k}^s(\R)$.
	\end{remark}
	
	We are now in a position to state the main results of this paper. The first main result concerns the case of slow bottom oscillations ($\na h=O(\e)$) described in \eqref{case for slow oscillation}:

	\begin{theorem}[Long time existence with $\nabla h= O(\e)$]\label{main results 1}
	Let $ s>\frac{n}{2}+2$, $n=1,2$, $\e\in(0,1)$ and $h(x)=1-b(x)$ satisfy  the (non-cavitation) condition
      \begin{equation}\label{non cavitation}
			 2 \ge h(x) \ge h_0>0, \,(\forall x\in \R^n), \quad\text{for some }\,h_0>0.
		\end{equation}
		Let $\{ a_j,d_j \}_{j=1,2},\{ b_j,c_j \}_{i=1,2,3,4}$ satisfy the condition \eqref{case for slow oscillation} and 
		$(\vv V_0,\eta_0)\in {\bf X}^s_\e(\R^n) \times X^s_\e(\R^n)$. In the two-dimensional case $n=2$ we further impose that $b_3=-c_3.$
		
		Then there exist $\e_0>0$ and  $T_0>0$ (independent of $\e$) such that for any $\e\in (0,\e_0)$ and 
		\beq\label{slow oscillation}
       \|\nabla h\|_{H^{s+2}} \leq C_0\e,\quad\text{for some }\, C_0>0,
       \eeq 
		system \eqref{Boussinesq-like system}-\eqref{initial data} admits an unique solution $(\vv V,\eta)$ with
		\beq\label{solve space}
			(\vv V,\eta)\in C ( [0,T_0/\e]; {\bf X}^s_\e(\R^n) \times X^s_\e(\R^n) )\cap C^1 \bigl( [0,T_0/\e];{\bf X}_\e^{s-1}(\R^n) \times X_\e^{s-1}(\R^n)  \bigr) . 
		\eeq
		Moreover, 
		\beq\label{total energy estimate for case 1}
		\max_{t\in[0,T_0/\e]}\bigl(\|\vv V\|_{{\bf X}^s_\e}+\|\eta\|_{X^s_\e}
		\bigr)
		\leq C\bigl(\|\vv V_0\|_{{\bf X}^s_\e}+\|\eta_0\|_{X^s_\e}\bigr),
		\eeq
		where $C>0$ is an universal constant. 
	\end{theorem}
	\begin{remark}
 (1). Using \eqref{Boussinesq-like system},  we also have
   \beno
   \max_{t\in[0,T_0/\e]}\bigl(\|\vv V_t\|_{{\bf X}^{s-1}_\e}+\|\eta_t\|_{X^{s-1}_\e}\bigr)\leq C\bigl(\|\vv V_0\|_{{\bf X}^s_\e}+\|\eta_0\|_{X^s_\e}\bigr).
   \eeno

(2). For $n=2$, the additional restriction $b_3=-c_3$ is imposed to cancel the effects of the uncontrolled terms involving $b_3$ and $c_3$.  This is necessary because the lack of full information on $\na\vv V$ in the functional energy $\|\vv V\|_{{\bf X}^s_\e}$  prevents these terms from being controlled directly. 
	\end{remark}

	The second main result concerns the case of fast bottom oscillations ($\na h=O(1)$) as defined in \eqref{case for fast oscillation}.

	\begin{theorem}[Long time existence with $\nabla h= O(1)$]\label{main results 2}
		Let $n=1$, $\e\in(0,1)$ and $h(x)=1-b(x)$ satisfy 
		\beq\label{fast oscillation}
        2 \ge h(x) \ge h_0>0, \,(\forall x\in \R),\quad \|\p_x h\|_{H^4}\leq 1,
		\eeq
		for some $h_0>0$.
Let $\{ a_j,d_j \}_{j=1,2},\{ b_j,c_j \}_{i=1,2,3,4}$ satisfy the condition \eqref{case for fast oscillation} and $(V_0,\eta_0)\in X^2_{\e^3}(\R)\times X^2_{\e^2}(\R)$. Then there exist  $\e_0>0 $ and $T_0>0$ (independent of $\e$) such that for any $\e\in (0,\e_0)$, system \eqref{Boussinesq-like system}-\eqref{initial data} has an unique solution $(V,\eta)\in C\bigl([0,T_0/\e];X^2_{\e^3}(\R)\times X^2_{\e^2}(\R)\bigr)$. 
Moreover,
\beq\label{total energy estimate for case 2}\begin{aligned}
&\max_{t\in[0,T_0/\e]}\bigl(\|V\|_{X^2_{\e^3}}^2+\|\eta\|_{X^2_{\e^2}}^2
+\|V_t\|_{X^1_{\e^2}}^2+\|\eta_t\|_{X^1_\e}^2+\|V_{tt}\|_{X^0_{\e}}^2+\|\eta_{tt}\|_{L^2}^2\bigr)\\
&\qquad
\leq C\bigl(\|V_0\|_{X^2_{\e^3}}^2+\|\eta_0\|_{X^2_{\e^2}}^2\bigr),
\end{aligned}\eeq
where $C>0$ is an universal constant. 	
	\end{theorem}
	\begin{remark}
(1). The result also holds when $(V_0,\eta_0)\in X^{2+k}_{\e^3}(\R)\times X^{2+k}_{\e^2}(\R)$ for any integer $k\geq 0$.

(2). The construction of the energy functional on the left-hand side of \eqref{total energy estimate for case 2} is crucial for establishing the long time existence result. This approach is tied to the strategy of improving spatial regularity by first improving the temporal regularity.
	\end{remark}

    \subsection{Difficulties, parameter selection and strategies} 

The analysis of long-time existence for the Boussinesq-like system \eqref{Boussinesq-like system} presents two principal challenges originating from significant bottom topography variations (e.g., $h(x)=1-b(x)=O(1)$):
\begin{itemize}
\item[(i)] {\it Dependence on bottom variations:} 
Large bottom variation introduce explicit $h$-dependence in both the linear and dispersive terms of \eqref{Boussinesq-like system}, fundamentally distinguishing it from flat-bottom Boussinesq systems where these terms have constant coefficients. 
\item[(ii)] {\it Structural implications:} This structural difference requires careful parameter selection to preserve essential mathematical properties-particularly symmetry and cancellation mechanisms-which are crucial for ensuring long-time existence of solutions.
\end{itemize}

\smallskip

\noindent$\bullet$ \textbf{Parameter selection}

\smallskip

The well-posedness analysis  is intrinsically linked to the selection of parameters. Based on the  linearized analysis of the flat-bottom case, the dispersion parameters set $(a_1,b_1,c_1,d_1)$ must satisfy the constraints \eqref{constraint of Boussinesq}. To eliminate the coupling effect of depth function $h$ on dispersive terms, we adopt a symmetrized dispersive principal part by setting $b_1=c_1$. The determination of other parameters depends on both the system structure and key technical challenges:

\textbf{ (1). Energy functionals.} In view of system \eqref{Boussinesq-like system}, we design the associated energy functional via  $h$-dependent operators $\cP_h^1$ and $\cP_h^2$ as follows
\beq\label{design of energy}\begin{aligned}
\bigl((1-\f\e2\cP_h^1)\vv V\,\big|\,\vv V\bigr)_{L^2}
&=\|\vv V\|_{L^2}^2+\f{\e}{2}a_1\|h\na\cdot\vv V\|_{L^2}^2
+\f{\e}{2}a_2\bigl(\na h\cdot\vv V\,\big|\,h\na\cdot\vv V\bigr)_{L^2}\\ 
&
\sim \|\vv V\|_{L^2}^2+\e a_1\|h\na\cdot\vv V\|_{L^2}^2\\
\bigl((1-\f\e2\cP_h^2)\eta\,\big|\,\eta\bigr)_{L^2}
&=\|\eta\|_{L^2}^2+\f{\e}{2}d_1\|h\na\eta\|_{L^2}^2
+\f{\e}{2}d_2\bigl(\na h\eta\,\big|\,h\na\eta\bigr)_{L^2}\\ 
&\sim \|\eta\|_{L^2}^2+\e d_1\|h\na\eta\|_{L^2}^2,
\end{aligned}\eeq
provided that $0<h_0\leq h\leq 2$ and
\beno\begin{aligned}
&a_1\geq 0,\,a_2=0,\,\, \text{or}\,\, a_1>0,\, a_2\neq 0,\,\,\|\na h\|_{L^\infty}\ll 1,\\ 
\text{and }\,\, &d_1\geq 0,\,d_2=0,\,\, \text{or}\,\, d_1>0,\, d_2\neq 0,\,\,\|\na h\|_{L^\infty}\ll 1.
\end{aligned}\eeno
The higher-order energy functionals for $(\vv V,\eta)$ are designed in a similar way. In this paper, we impose $a_1>0$.
	
\textbf{(2). Lack of full information on $\na\vv V$ for $n=2$.} For $n=2$, the energy functional defined in \eqref{design of energy}  contains only $\na\cdot\vv V$ and lacks full information on $\na\vv V$. However, the energy estimate for $\vv V$ leads to
\beq\label{eq 1 for introduction}\begin{aligned}
\f12\f{d}{dt}\bigl((1-\f\e2\cP_h^1)\vv V\,\big|\,\vv V\bigr)_{L^2}
&=\bigl((1-\f\e2\cP_h^1)\vv V_t\,\big|\,\vv V\bigr)_{L^2}\\
&\quad+\f\e2 a_2\bigl[\bigl(\na h\cdot\vv V\,\big|\,h\na\cdot\vv V_t\bigr)_{L^2}
-\bigl(\na h\cdot\vv V_t\,\big|\,h\na\cdot\vv V\bigr)_{L^2}\bigr]
\end{aligned}\eeq
The lack of full information on $\na\cdot\vv V$ means that  the cross term with $a_2$ in \eqref{eq 1 for introduction} cannot be controlled unless $a_2=0$.  Similarly, we have to take $b_3=-c_3$ for $n=2$. Thus, for $n=2$, we impose
\beno
a_2=0,\quad b_3=-c_3.
\eeno

\textbf{(3). Difficulty  in higher-order energy estimates.} The first-order energy estimate for $\vv V$ leads to
\beno\begin{aligned}
\f12\f{d}{dt}\bigl((1-\f\e2\cP_h^1)\p_j\vv V\,\big|\,\p_j\vv V\bigr)_{L^2}
&=\bigl(\p_j[(1-\f\e2\cP_h^1)\vv V_t]\,\big|\,\p_j\vv V\bigr)_{L^2}\\ 
&\quad-\f\e2 a_1\bigl(h\p_jh\na\cdot\vv V_t\,\big|\,\na\cdot\p_j\vv V\bigr)_{L^2}
+\cdots.
\end{aligned}\eeno
When $a_1>0$, the spatial derivative $\p_j$ acts on $\cP_h^1$ (and consequently on $h$), producing the cross term $-\f\e2 a_1\bigl(h\p_jh\na\cdot\vv V_t\,\big|\,\na\cdot\p_j\vv V\bigr)_{L^2}$. Estimates on this term suggest an existence time scale of   $O(\f{1}{|\na h|})$. Accordingly, we analyze \eqref{Boussinesq-like system} in two regimes:
\beno
\textit{slow bottom oscillation:}\, \na h=O(\e)\quad\text{and}\quad
\textit{fast bottom oscillation:}\, \na h=O(1).
\eeno 
The slow oscillations,  the standard energy method on $H^s(\R^n)$ yields the desired $O(\f{1}{\e})$ existence time under appropriate parameter choices. For fast oscillations, the expected existence time scale is $O(1)$. Achieving the long time scale $O(\f{1}{\e})$ requires a refined approach: we first improve time regularity (since $\p_t$ commutes with time-independent $h$) and then recover spatial regularity from the equations.  

The fully symmetric case
	\begin{align}
		 a_1\ge 0,d_1\ge 0,a_2 = d_2 =0;\quad b_1 = c_1,\, b_2 = -c_2,\, b_3 = - c_3,\, b_4 = c_4, \label{full symmetric}
	\end{align}
would naturally ensure long-time existence for solutions of system \eqref{Boussinesq-like system}. However, the constraint \eqref{full symmetric} cannot be satisfied for any $\{ \lambda_1,\lam_2,\mu,\theta \}$. Therefore, we  consider systems where only the principal parts maintain symmetry, while developing careful estimates to control the resulting non-symmetric lower-order terms.

	\smallskip

	Based on this analysis, we select parameters for {\it slow} and {\it fast bottom oscillations} as specified in \eqref{case for slow oscillation} and \eqref{case for fast oscillation}. 

	\smallskip
	
\noindent$\bullet$ \textbf{Strategies of proofs}

\smallskip 
	
	{\bf (1).} For Theorem \ref{main results 1} concerning the slow bottom oscillation, we employ standard energy methods for symmetric quasi-linear systems in $H^s(\R^n)$. For $n=2$, cancellations between lower-order linear terms is crucial. 
	
    {\bf (2).} For  Theorem \ref{main results 2} concerning the fast bottom oscillation ($n=1$), the $h$-dependent coefficients in the linear part of system \eqref{Boussinesq-like system} (with $\p_xh = O(1)$) require a two-stage regularity analysis. Since the spatial derivatives in $x$ for the solution $(V,\eta)$ present substantial analytical difficulties, we first derive {\it a priori} energy estimates for higher-order temporal regularity in $t$ for $(V,\eta)$,(specially, $(\p_t^kV,\,\p_x^k\eta)\,$ for $k=0,1,2$), by carefully utilizing both the structural properties of the principal linear operator and the cancellation mechanisms in lower-order terms.  The established temporal regularity of $(V,\eta)$ then enable us to recover the corresponding spatial regularity through the careful analysis of the coupled equations in system \eqref{Boussinesq-like system}.

	\section{Preliminary}
	
	\subsection{Notations}
	
	The notation $f\lesssim g$ means that there exists an universal constant $C>0$ such that $f\leq Cg$. While the notation $g\gtrsim f$ means $f\lesssim g$ and  $f\sim g$ represents  both $f\lesssim g$ and $g\lesssim f$.  
	The constant $C>0$ denotes a universal constant which may changes from line to line. 
	
	If  $A, B$ are two operators, $[A,B]=AB-BA$ denotes their commutator.

	The Fourier transform of a tempered distribution $f\in\cS'(\R^n)$ is defined by
	\begin{align*}
		\widehat{f}(\xi)= \cF(f) (\xi) \eqdefa \int_{\R^n} e^{-ix\cdot \xi} f(x) dx.
	\end{align*}
	We shall use $\cF^{-1}(f)$ to denote the Fourier inverse transform of $f\in\cS'(\R^n)$.
	And the Fourier multiplier $\Lam^s$ is defined by
	\beno
\Lam^s f(x)\eqdefa\cF^{-1}\bigl[(1+|\xi|^2)^{\frac{s}{2}}\widehat{f}(\xi)\bigr],\quad\forall\, f\in\cS'(\R^n).
	\eeno 
	
    The notation $\|\cdot\|_{L^p}$ stands for  the $L^p(\R^n)$ norm for $1\leq p \leq \infty$. For any $s\in\R$ and $p\in [1,\infty]$, $W^{s,p}(\R^n)$  denote the classical  $L^p$ based  non-homogeneous Sobolev spaces with the norm $\|\cdot\|_{W^{s,p}}$. We also use the notation $H^s(\R^n) = W^{s,2}(\R^n)$ for the sake of convenience. 

    The $L^2(\R^n)$-inner product of $f$ and $g$ is defined by
$(f\,|\,g)_{L^2}\eqdefa\int_{\R^n}f\bar{g}dx.$ 
  And the $H^s(\R^n)$-inner product of $f$ and $g$ is denoted by
	\begin{align*}
		( f\,|\,g )_{H^s} \eqdefa ( \Lam^s f\,|\,\Lam^s g )_{L^2}.
	\end{align*}

	For the function $f(t,x)$ defined on $\R_+\times\R$, we denote its derivatives by subscripts for simplicity: $f_t=\p_tf$, $f_x=\p_xf$, with higher derivatives (e.g.,$f_{tt},\,f_{tx},\,f_{xx}$, etc.) following the same convection.

	\subsection{Technical lemmas}
	In this subsection, we present several crucial technical lemmas. 
	Firstly, we recall the classical tame product estimate: 
if $s\geq0$, 
\beq\label{tame}
	\norm{f\cdot g}{H^s} \lesssim \norm{f}{H^s} \norm{g}{L^\infty}+\norm{f}{L^\infty}\norm{g}{H^s},\quad\forall\, f,g\in H^s(\R^n)\cap L^\infty(\R^n).
\eeq

Wee also restate the lemma concerning the classical commutator estimates, as detailed in \cite{Lannes1}, Theorems 3 and 6:
	\begin{lem}\label{commutator lemma} Let $s\ge 0$, $t_0>\frac{n}{2}$ and $-t_0< r \le t_0+ 1$, then for any $f\in H^{t_0+1} \cap H^{s+r}(\R^n)$ and $u\in H^{s+r-1}(\R^n)$, there holds
		\begin{equation}\label{commutator estimate}
			\norm{ [\Lam^s,f] u  }{H^r} \lesssim \begin{cases}
				  \| \nabla f \|_{H^{t_0}} \|u\|_{H^{s+r-1}}, \quad & s+ r -1 \le t_0 \\
				 \norm{\nabla f}{H^{s+r-1}} \norm{u}{H^{s+r-1}}, &s+ r -1 > t_0
			\end{cases}. 
		\end{equation}
	\end{lem}

    Furthermore, we need  more refined product estimates that are stated in the following lemma.
	\begin{lem}\label{tame estimate}
		Let $s\ge 0$, $t_0>\frac{n}{2}$, and $h-1\in H^{t_0+1}(\R^n) \cap  H^{s}(\R^n)$ be a real value function satisfying $h\in[0,2]$. Then there holds
		\begin{align}\label{multiplier bounds}
				\norm{h u}{H^s}  \lesssim  (\|\nabla h\|_{H^{t_0}}+\|\nabla h\|_{H^{s-1}})\|u\|_{H^{s-1}}+ \norm{h}{L^\infty}  \norm{u}{H^s},\quad\forall u\in H^s(\R^n). 
		\end{align} 
	\end{lem}
	\begin{proof}
		Since 
		\beno
		\Lam^s (h u) = [\Lam^s,h-1]u + h \Lam^s u,
		\eeno
		 we obtain \eqref{multiplier bounds} by using Lemma \ref{commutator lemma} for $r=0$.
\end{proof}

\begin{lem}\label{composition lem}
 Let $s>\frac{n}{2}+1$, $n=1,2$, $[s]$ be the integer part of $s$ and $h$ be a real valued function taking values on $[h_0,2]$ for some $0<h_0<2$.

(1). If $F\in C^1( [h_0,2] )$ and $h-1\in H^{s}(\R^n)$, there holds
\beq\label{composition}
\|F(h)-F(1)\|_{H^s}\leq C_{\|F'\|_{L^\infty},\|h-1\|_{L^\infty}}\|h-1\|_{H^s}.
\eeq

(2).  If $F\in C^{[s]+1}([h_0,2])$ and $h-1\in H^{s+1}(\R^n)$, there holds
\beq\label{bounds of grad composite}
	\|\nabla (F\circ h)\|_{H^{[s]}}\leq C_{\|F\|_{W^{[s]+1,\infty}},\|\nabla h\|_{H^s}}  \|\nabla h\|_{H^s}.
	\eeq

(3). If $r\in[0,2]$,  $F\in C^{ [s+r]+1 }([h_0,2])$ and $h-1\in H^{s+r+1}(\R^n)$, there holds
\begin{align}\label{composite commutator}
			\norm{[\Lam^s,F\circ h]u}{H^{r}} \leq C_{\norm{F}{W^{[s+r]+1,\infty}},\norm{\nabla h}{H^{s+r}}}  \norm{\nabla h}{H^{s+r}}\norm{u}{H^{s+r-1}}  ,\quad \forall u\in H^{s+r-1}(\R^n). 
		\end{align}
		Here and in what follows, $C_{\lambda_1,\cdots,\lambda_j}$ denotes  a constant that depends on $\lambda_1,\cdots,\lambda_j$ and increases accordingly. 
\end{lem}
\begin{proof}
(1). The composition estimate \eqref{composition} is a consequence of Theorem 2.61 in \cite{BCD}. 

(2). For \eqref{bounds of grad composite}, we  recall the  Fa\`{a} di Bruno's formula on $F\circ h$ as follows: for $\alpha\in(\Z_{\geq 0})^n$
	    and $|\al|\geq 1$,
		\beno
		\partial^\alpha ( F\circ h ) = \sum_{1\leq j\leq|\al|}\sum_{\beta_1+\cdots+\beta_j=\al} C_{j,\beta_1,\cdots,\beta_j} F^{(j)}(h)\cdot\prod_{l=1}^j\partial^{\beta_l} h,
		\eeno
		where  $|\beta_l|\geq 1\,(l=1,\cdots,j)$, $F^{(j)}(h)\eqdefa\f{d^jF}{dh^j}(h)$, and the coefficients $C_{j,\beta_1,\cdots,\beta_j}$ are positive integers.
	
		Consequently, for any multi-index $\alpha$ such that $ 1\leq|\alpha|\leq [s]+1$, we obtain by using H\"older inequality and Sobolev's embedding theorem that
		\begin{align*}
			\|\partial^\alpha( F\circ h )\|_{L^2} &\lesssim \norm{F}{W^{{[s]+1},\infty}} \sum_{1\leq j\leq|\al|}\sum_{\beta_1+\cdots+\beta_j=\al} \prod_{l=1}^j\|\partial^{\beta_l } h\|_{L^{\f{2|\alpha|}{|\beta_l|}}}\\
			&\lesssim \norm{F}{W^{{[s]+1},\infty}} \sum_{1\leq j\leq|\al|}\sum_{\beta_1+\cdots+\beta_j=\al} \prod_{l=1}^j \| \nabla h\|_{H^{(\frac{1}{2} - \frac{|\beta_l|}{2|\alpha|})n+ |\beta_l| -1 }}.
		\end{align*}
	  Since
	    \begin{align*}
	    	(\frac{1}{2} - \frac{|\beta_l|}{2|\alpha|})n+ |\beta_l|-1=
	    	\bigl(1-\f{n}{2|\al|}\bigr)|\beta_l|+\f{n}{2}-1\leq [s]\leq s,\quad \forall l=1,\cdots,j,
	    \end{align*}
	   we get
	   \beno
\|\partial^\alpha( F\circ h )\|_{L^2}\lesssim\norm{F}{W^{{[s]+1},\infty}} \sum_{1\leq j\leq|\al|}\|\na h\|_{H^s}^j,
	   \eeno
which implies  
		\beno
			 \norm{\nabla (F\circ h)}{H^{[s]}}\lesssim  \sum_{1\leq |\alpha|\leq [s]+1}
		      \norm{\partial^\alpha(F\circ h)}{L^2} \lesssim \sum_{j=1}^{[s]+1} \norm{\nabla h}{H^s}^j \norm{F}{W^{{[s]+1},\infty}} . 
		\eeno
This gives rise to \eqref{bounds of grad composite}.

(3). Thanks to \eqref{commutator estimate}, we get
		\beno\begin{aligned}
\norm{[\Lam^s,F\circ h]u}{H^{r}} &\lesssim
\bigl(\|\na (F\circ h)\|_{H^{t_0}}+\|\na (F\circ h)\|_{H^{s+r-1}}\bigr)\|u\|_{H^{s+r-1}}\lesssim\|\na (F\circ h)\|_{H^{[s+r]}}\|u\|_{H^{s+r-1}},
		\end{aligned}\eeno
where $t_0=1$ if $n=1$ and $t_0=s-1$ if $n=2$. Then we obtain \eqref{composite commutator} by using \eqref{bounds of grad composite}. This completes the proof of the lemma.
\end{proof}
	\begin{remark}
Let $s>\frac{n}{2}+1,\,n=1,2$. Assume that   $F\in C^{[s]+1}(\R)\cap W^{[s]+1,\infty}(\R)$ and $h-1\in H^{s+1}(\R^n)$. Due to \eqref{multiplier bounds}, we get
\beno\begin{aligned}
\norm{(F\circ h) u}{H^s}&\lesssim\|\na(F\circ h)\|_{H^{s-1}}\|u\|_{H^{s-1}}+\|F\circ h\|_{L^\infty}\|u\|_{H^s}\\
&
\lesssim\|\na(F\circ h)\|_{H^{[s]}}\|u\|_{H^{s-1}}+\|F\|_{L^\infty}\|u\|_{H^s}.
\end{aligned}\eeno
which along with \eqref{bounds of grad composite} implies 
	\begin{align}\label{composite multiplier bounds}
			\norm{(F\circ h) u}{H^s} \leq C_{\norm{F}{W^{{[s]+1},\infty}},\norm{\nabla h}{H^s}} \norm{u}{H^s},\quad \forall u\in H^s(\R^n).
		\end{align}
	\end{remark}

	We end this section with a lemma that involves the operators $\cP_h^1$ and $\cP_h^2$ as defined in \eqref{def of Ph}.
	\begin{lem}\label{technical lemma 1}
		Let $s>\frac{n}{2}+1$. Assume that $h-1\in H^{s+1}(\R^n)$ is the real value function such that  $h\in [h_0,2]$ for some constant $h_0>0$ and $\|\na h\|_{H^s}\leq 1$.
		Then there exists a constant $c_0>0$ such that, if
		$\norm{\nabla h}{H^s}\le c_0< 1$ and $\e\in(0,1)$, there hold

		(1). for $a_1>0$,  
			\begin{align}\label{equivalent functional 1}
		\left(  (1-\f\e2 \mathcal{P}_h^1) \vv f \,|\,  \vv f  \right)_{H^s}\sim \norm{\vv f}{{\bf X}_\e^s}^2, \quad\forall\, \vv f\in {\bf X}_\e^s(\R^n),
			\end{align}
			
		(2). for $d_1>0$,  	
		\beq\label{equivalent functional 2}
		\left( (1-\f\e2 \mathcal{P}_h^2)  g \,|\, g \right)_{H^s}\sim\norm{g}{X_\e^s}^2,\quad\forall\,g\in X_\e^s(\R^n). 
			\end{equation}
	\end{lem}
	\begin{proof}
			(1). By the definition of $\mathcal{P}_h^1$ in \eqref{def of Ph}, one infers that
		\beq\label{Eq 1}\begin{aligned}
			\bigl( (1-\f\e2 \mathcal{P}_h^1) \vv f\,|\, \vv f   \bigr)_{H^s}\!\!=\, & \norm{\vv f}{H^s}^2\! +\f\e2 a_1  \bigl( h^2 \nabla \cdot \vv f \,|\, \nabla \cdot \vv f \bigr)_{H^s} + \f\e2 a_2 \bigl( h\nabla h\cdot \vv f\,|\,\nabla \cdot \vv f  \bigr)_{H^s} \\
			=\,&  \norm{\vv f}{H^s}^2 \! + \f\e2 a_1  \norm{h\Lam^s\nabla \cdot\vv f}{L^2}^2 +\f\e2 I_1,
		\end{aligned}\eeq
		where 
		\beno
I_1\eqdefa a_1  \bigl( [\Lam^s,h^2]\nabla \cdot \vv f \,|\, \Lam^s \nabla \cdot \vv f  \bigr)_{L^2} +  a_2 \bigl( h\nabla h\cdot \vv f\,|\,\nabla \cdot \vv f  \bigr)_{H^s}
		\eeno
		
		Since $0<h_0\leq h(x)\leq 2$, there holds
		\beq\label{Eq 2}
h_0^2\|\na\cdot\vv f\|_{H^s}^2\leq \norm{h\Lam^s\nabla \cdot\vv f}{L^2}^2
\leq 4\|\na\cdot\vv f\|_{H^s}^2.
		\eeq

	For the first term of $I_1$,   by using \eqref{composite commutator} with $F(r)=r^2$, we get 
	\beno
\bigl|\bigl( [\Lam^s,h^2]\nabla \cdot \vv f \,|\, \Lam^s \nabla \cdot \vv f  \bigr)_{L^2}\bigr|
\leq C_{\|\na h\|_{H^s}}\|\na h\|_{H^s}\|\na\cdot\vv f\|_{H^{s-1}}\|\na\cdot\vv f\|_{H^s}\leq C_1\|\na h\|_{H^s}\|\vv f\|_{H^s}\|\na\cdot\vv f\|_{H^s},
	\eeno
	where $C_1>0$ is a constant, and we have used the assumption $\|\na h\|_{H^s}\leq 1$ in the last inequality.

	For the second term of $I_1$, due to the tame estimate,\eqref{composite multiplier bounds} and the condition $\|\na h\|_{H^s}\leq 1$, we have
	\beno
\bigl|\bigl( h\nabla h\cdot \vv f\,|\,\nabla \cdot \vv f  \bigr)_{H^s}\bigr|
\leq C_{\|\na h\|_{H^s}}\|\na h\|_{H^s}\|\vv f\|_{H^s}\|\na\cdot\vv f\|_{H^s}
\leq C_2\|\na h\|_{H^s}\|\vv f\|_{H^s}\|\na\cdot\vv f\|_{H^s},
\eeno
where $C_2>0$ is a constant. Then there exists a constant $C_3>0$ such that
\beq\label{Eq 3}
|I_1|\leq \|\vv f\|_{H^s}^2+ C_3\|\na h\|_{H^s}^2\|\na\cdot\vv f\|_{H^s}^2,
\eeq
which along with \eqref{Eq 1} and \eqref{Eq 2} leads to 
\beno
\bigl( (1-\f\e2 \mathcal{P}_h^1) \vv f\,|\, \vv f \bigr)_{H^s}
\geq(1-\f{\e}{2})\|\vv f\|_{H^s}^2+\f\e2\bigl(a_1h_0^2-C_3\|\na h\|_{H^s}^2\bigr)\|\na\cdot\vv f\|_{H^s}^2.
\eeno

For $a_1>0$, taking  $c_1=\min\{\sqrt{\f{a_1}{2C_3}}h_0,1\}$, there holds for any $\|\na h\|_{H^s}\leq c_1\leq 1$,
\beq\label{Eq 4}
\bigl( (1-\f\e2 \mathcal{P}_h^1) \vv f\,|\, \vv f \bigr)_{H^s}
\geq (1-\f{\e}{2})\|\vv f\|_{H^s}^2+\f\e4a_1h_0^2\|\na\cdot\vv f\|_{H^s}^2.
\eeq

On the other hand, using \eqref{Eq 2}, \eqref{Eq 3} and the condition $\|\na h\|_{H^s}\leq 1$, we deduce from \eqref{Eq 1}that 
\beno
\bigl( (1-\f\e2 \mathcal{P}_h^1) \vv f\,|\, \vv f \bigr)_{H^s}
\leq(1+\f{\e}{2})\|\vv f\|_{H^s}^2+\f\e2(4+C_3)\|\na\cdot\vv f\|_{H^s}^2,
\eeno
which along with \eqref{Eq 4} implies that
\beq\label{Eq 5}
\bigl( (1-\f\e2 \mathcal{P}_h^1) \vv f\,|\, \vv f \bigr)_{H^s}
\sim\|\vv f\|_{H^s}^2+\e\|\na\cdot\vv f\|_{H^s}^2=\|\vv f\|_{{\bf X}_\e^s}^2,
\eeq
provided that $\|\na h\|_{H^s}\leq c_1\leq1$. This is exactly \eqref{equivalent functional 1}.

\smallskip
		
(2). By the definition of $\mathcal{P}_h^2$ in \eqref{def of Ph}, one has
		\begin{align}\label{calculation for g}
			\bigl(  (1-\f\e2 \mathcal{P}_h^2) g\,|\,g \bigr)_{H^s} &= \norm{g}{H^s}^2 + \f\e2 d_1 \bigl( h^2 \nabla g\,|\, \nabla g  \bigr)_{H^s} + \f\e2 d_2 \bigl( h\nabla h\, g\,|\, \nabla g  \bigr)_{H^s}\notag \\
			&= \norm{g}{H^s}^2 + \f\e2 d_1 \|h\Lam^s\nabla g\|_{L^2}^2 + \f\e2 I_2, 
		\end{align}
		where
		\begin{align*}
			I_2\eqdefa d_1  \bigl( [\Lam^s,h^2]\nabla g \,|\, \Lam^s \nabla g \bigr)_{L^2} + d_2 \bigl( h\nabla h\,g\,|\,\nabla g  \bigr)_{H^s}
		\end{align*}

		A similar derivation as for \eqref{Eq 3} leads to
		\beno
|I_2|\leq \|g\|_{H^s}^2+ C_4\|\na h\|_{H^s}^2\|\na g\|_{H^s}^2,
		\eeno
		for some constant $C_4>0$. Then for $d_1>0$, there exists $c_2\in(0,1)$ such that 
		\beno
\bigl(  (1-\f\e2 \mathcal{P}_h^2) g\,|\,g \bigr)_{H^s}\sim \|g\|_{H^s}^2+\e\|\na g\|_{H^s}^2=\|g\|_{X_\e^s}^2,
		\eeno
provided that $\|\na h\|_{H^s}\leq c_2$. This is exact \eqref{equivalent functional 2}. 

	Thus, there exists a constant $c_0=\min\{c_1,c_2\}\in(0,1)$, such that both \eqref{equivalent functional 1} and \eqref{equivalent functional 2} hold. The proof of the lemma is completed.
	\end{proof}

	\section{Proof of Theorem \ref{main results 1}}
	The aim of this section is to establish Theorem \ref{main results 1}, that is to show the long time existence of solutions to the Cauchy problem of \eqref{Boussinesq-like system} under the assumption of  slowly oscillating bottom in the cases  specified in \eqref{case for slow oscillation}. The proof   relies heavily on the continuity arguments and the {\it a priori } energy estimates. 

\subsection{ {\it a priori} energy estimate}

	We only provide the proof for the two-dimensional case here; the proof of the one-dimensional case is significantly simpler.

	For $n=2$, we restate the restrictions on the parameters $\{ a_j,d_j \}_{j=1,2}$ and $\{ b_j,c_j \}_{j=1,2,3,4}$ as given in \eqref{case for slow oscillation} as follows:
 \beno
a_1>0,\,d_1>0,\,a_2=0,\,d_2\in\R,\,b_1=c_1,\,b_3=-c_3,\,(b_j,c_j)_{j=2,4}\in\R^4.
\eeno
Correspondingly, the expressions of $\mathcal{P}_h^j\,(j=1,2)$ are  as follows
	\begin{align}\label{def of Phi, d>1}
		&\quad \left\{ \begin{aligned}
			\mathcal{P}_h^1 &= a_1 \nabla (h^2 \nabla \cdot \quad ) , \\
			\mathcal{P}_h^2 & = d_1 \nabla\cdot (h^2 \nabla \quad) +d_2 \nabla \cdot (h\nabla h\times \quad), 
		\end{aligned} \right.
	\end{align}
	while system \eqref{Boussinesq-like system}  reads as 
	\begin{equation}\label{Boussinesq-like case 1}
	\left\{   \begin{aligned}
		(1-\f\e2 \mathcal{P}_h^1) \p_t \vv V   + \sqrt{h} \nabla \eta + \f\e2 \Bigl[ {\vv F}_h(\vv V,\eta) + b_1 \sqrt{h} \nabla \nabla \cdot (h^2 \nabla \eta) + b_2 \sqrt{h} \nabla (  h \nabla h \cdot \nabla \eta ) \quad \quad \quad\quad \,\,\,\\ 
		+ b_3 \nabla h \nabla \cdot (h\sqrt{h} \nabla\eta) + b_4 \sqrt{h} \nabla h (\nabla h \cdot \nabla \eta) \Bigr] =0, \\
		(1-\f\e2 \mathcal{P}_h^2) \partial_t \eta + \nabla\cdot  (\sqrt{h} \vv V ) + \f\e2 \Bigl[   f_h(\vv V,\eta) + \nabla \cdot \bigl(   b_1 h^2 \nabla \nabla \cdot (\sqrt{h} \vv V) + c_2 h \nabla h \nabla \cdot (\sqrt{h}\vv V)  \quad \quad \\
		- b_3 h \sqrt{h} \nabla (\nabla h \cdot \vv V ) + c_4 \sqrt{h} \nabla h (\nabla h \cdot \vv V)   \bigr) \Bigr] =0,
	\end{aligned}\right.   
	\end{equation}
	where $\vv F_h(\vv V,\eta)$ and $f_h(\vv V,\eta)$ are defined in \eqref{def of Fh and fh}. 

	The {\it a priori} energy estimates for system \eqref{Boussinesq-like system} under restriction \eqref{case for slow oscillation} is stated in the following proposition. For $n=2$, the system \eqref{Boussinesq-like system} reduces to \eqref{Boussinesq-like case 1}.

\begin{prop}\label{priori energy estimate}
		Let ${\color{red} s>\frac{n}{2}+2}$, $n=1,2$, $T_0>0$, $\e\in(0,1)$, 
		$\{ a_j,d_j \}_{j=1,2}$ and $\{ b_j,c_j \}_{i=1,2,3,4}$ satisfy the conditions outlined in \eqref{case for slow oscillation}. Assume that $h$ satisfies
\beq\label{condition for h}
\|\na h\|_{H^{s+2}(\R^n)}\leq C\e\leq c_0\leq1,\quad 2\geq h(x)\geq h_0>0,\quad \text{for some}\,h_0>0,
\eeq
where $C>0$ is a universal constant and $c_0>0$ is the constant stated in Lemma \ref{technical lemma 1}. Then for any sufficiently smooth solutions  $(\vv V,\eta)$ of \eqref{Boussinesq-like system}-\eqref{initial data} over $[0,\f{T_0}{\e}]$, there exist constants $C_1>1$ and $C_2>1$ such that
\beq\label{priori estimate for case 1}
\cE_s(t)
\leq C_1\cE_s(0)+C_2\e t\max_{\tau\in[0,t]}\cE_s(\tau)\bigl(1+\e\cE_s(\tau)^{\f12}\bigr),
\quad\forall\,t\in[0,\f{T_0}{\e}],
\eeq
where 
$
\cE_s(t)\eqdefa\norm{\vv V(t,\cdot)}{{\bf X}_\e^s}^2 + \norm{\eta(t,\cdot)}{X_\e^s}^2.
$
\end{prop}
	\begin{proof}
		We only sketch the proof of case $n=2$. The proof is {divided} into several steps. 

		{\bf Step 1. Energy estimates.} Firstly, we define the energy functional  associated with \eqref{Boussinesq-like case 1} as follows:
\beq\label{energy functional for case 1}
E_s(t)\eqdefa\bigl((1-\f\e2\cP_h^1)\vv V\,|\,\vv V\bigr)_{H^s}+\bigl((1-\f\e2\cP_h^2)\eta\,|\,\eta\bigr)_{H^s}.
		\eeq
Thanks to Lemma \ref{technical lemma 1}, there holds
		\begin{align}\label{equivalent energy function}
    	 E_s(t) \sim \norm{\vv V}{{\bf X}_\e^s}^2 + \norm{\eta}{X_\e^s}^2=\cE_s(t).
    \end{align}

For system \eqref{Boussinesq-like case 1}, a direct calculation yields to
		\begin{equation}\label{calculation of energy}
			\f12\frac{d}{dt} E_s(t) = \mathcal{S} + \mathcal{T} + \mathcal{N},  
		\end{equation}
		where
		\beno\begin{aligned}
			\mathcal{S} &  \eqdefa- \left\{  \bigl( \sqrt{h} \nabla \eta  \,|\, \vv V  \bigr)_{H^s} +\left( \nabla \cdot (\sqrt{h} \vv V) \,|\, \eta \right)_{H^s}\right\} \\
			&\,\, \quad -\f\e2 b_1 \left\{  \bigl( \sqrt{h} \nabla \nabla \cdot (h^2 \nabla \eta)\,|\, \vv V  \bigr)_{H^s} + \bigl(   \nabla \cdot (h^2 \nabla \nabla \cdot (\sqrt{h} \vv V)) \,|\, \eta  \bigr)_{H^s}  \right\}  \\
			&\,\, \quad -\f\e2 b_3 \left\{ \bigl( \nabla h \nabla \cdot  (h\sqrt{h}\nabla \eta  )\,|\, \vv V \bigr)_{H^s} - \bigl( \nabla \cdot  (h\sqrt{h} \nabla (\nabla h\cdot \vv V) ) \,|\, \eta   \bigr)_{H^s}    \right\} \\
			&\eqdefa \mathcal{S}_1 -\f\e2 b_1\mathcal{S}_2 -\f\e2 b_3\mathcal{S}_3,\\
			\mathcal{T}& \eqdefa -\f\e2 \left( {\vv F}_h(\vv V,\eta)\,|\, \vv V  \right)_{H^s} - \f\e2 \left(  f_h(\vv V,\eta) \,|\, \eta \right)_{H^s},\\
			\mathcal{N} & \eqdefa \f\e4 \Bigl\{\left( \mathcal{P}_h^1 \vv V_t \,|\, \vv V  \right)_{H^s} - \left(  \mathcal{P}_h^1\vv V \,|\, \vv V_t \right)_{H^s} + \left( \mathcal{P}_h^2 \eta_t  \,|\, \eta \right)_{H^s} -\left( \mathcal{P}_h^2 \eta  \,|\, \eta_t \right)_{H^s}  \Bigr\}  \\
			&\quad \,\, - \f\e2 \Bigl\{    b_2\bigl(   \sqrt{h} \nabla ( h\nabla h\cdot \nabla \eta) \,|\,\vv V   \bigr)_{H^s}  - c_2 \bigl( h \nabla h \nabla \cdot (\sqrt{ h}\vv V)  \,|\, \nabla \eta \bigr)_{H^s}  \\
			&\quad \,\, + b_4 \bigl(   \sqrt{h} \nabla h(\nabla h\cdot \nabla \eta) \,|\,\vv V   \bigr)_{H^s}- c_4\bigl(     \sqrt{h} \nabla h(\nabla h\cdot \vv V)   \,|\, \nabla\eta  \bigr)_{H^s}\Bigr\}\\
			&\eqdefa \f\e4\mathcal{N}_1 -\f\e2 \mathcal{N}_2.
		\end{aligned}\eeno

	{\bf Step 2. Estimate of $\mathcal{S}$.} We estimate $\mathcal{S}_j\,(j=1,2,3)$ one by one.
	
	\textit{Step 2.1. Estimate of $\mathcal{S}_1$.} Due to the definition of $\mathcal{S}_1$, we have
	\begin{align*}
		\mathcal{S}_1& = - \left( \Lambda^s(\sqrt{h} \nabla \eta ) \,|\, \Lambda^s\vv V  \right)_{L^2} + \left( \Lambda^s(\sqrt{h} \vv V) \,|\, \Lambda^s\na\eta \right)_{L^2} \\
		&
		 = - \left([\Lam^s, \sqrt{h}] \nabla \eta \,|\, \Lam^s \vv V  \right)_{L^2} - \left(  \Lam^s\eta\,|\,\na( [\Lam^s,\sqrt{h}] \vv V) \right)_{L^2}. 
	\end{align*}
	
	Using  the assumption $\|\na h\|_{H^{s+2}(\R^n)}\leq C\e \leq c_0\leq1$, \eqref{commutator estimate} and \eqref{composite commutator} with $F(x) = \sqrt{x}$, we obtain
	\beq\label{estimate of S1}\begin{aligned}
		|\mathcal{S}_1| &\lesssim \|\nabla h\|_{H^s} \|\nabla\eta\|_{H^{s-1}} \|\vv V\|_{H^s} + \|\nabla h\|_{H^{s+1}} \|\vv V\|_{H^s} \|\eta\|_{H^s}\\
		 & \lesssim\|\nabla h\|_{H^{s+1}} \norm{\eta}{H^s} \norm{\vv V}{H^s} \lesssim  \e \norm{\eta}{H^s} \norm{\vv V}{H^s} \lesssim \e \cE_s(t).
	\end{aligned}\eeq

	\textit{Step 2.2. Estimate of $\mathcal{S}_2$. } For $\mathcal{S}_2$, we rewrite it as
	\beno\begin{aligned}
		\mathcal{S}_2 &= \Bigl(\Lam^s\bigl(\sqrt{h} \nabla \nabla \cdot (h^2 \nabla \eta)\bigr)\,\Big|\,\Lam^s\vv V \Bigr)_{L^2} -\Bigl(  \Lam^s\bigl(h^2 \nabla \nabla \cdot (\sqrt{h} \vv V)\bigr) \,\Big|\,\Lam^s\na\eta \Bigr)_{L^2}  \\
&=\bigl([\Lam^s,\sqrt{h}] \nabla \nabla \cdot (h^2 \nabla \eta)\,\big|\,\Lam^s\vv V  \bigr)_{L^2} -\bigl([\Lam^s,h^2]\nabla \nabla \cdot (\sqrt{h} \vv V) \,|\,\Lam^s\na\eta \bigr)_{L^2}\\
		& \quad \,-\bigl(\Lam^s(h^2 \nabla \eta)\,|\, \nabla \nabla \cdot ([\Lam^s,\sqrt{h}] \vv V)\bigr)_{L^2} +\bigl( \Lam^s\nabla \nabla \cdot (\sqrt{h} \vv V)\,|\, [\Lam^s, h^2]\nabla \eta  \bigr)_{L^2}\\
		&\eqdefa\sum_{j=1}^4\cS_{2j}.
	\end{aligned}\eeno

	{\bf For $\cS_{21}$}, using integration by parts, we have
	\begin{align*}
		\mathcal{S}_{21} &= -\bigl( [\Lam^s,\nabla (\sqrt{h})] \nabla \cdot (h^2 \nabla \eta)\,|\, \Lam^s \vv V  \bigr)_{L^2} - \bigl( [\Lam^s,\sqrt{h}] \nabla \cdot (h^2 \nabla \eta)\,|\,\Lam^s \nabla \cdot \vv V  \bigr)_{L^2} 
	\end{align*}
	which along with \eqref{commutator estimate} implies
	\beno\begin{aligned}
|\mathcal{S}_{21}|&\lesssim\|\nabla^2 (\sqrt{h})\|_{H^{s-1}}\|\nabla \cdot (h^2 \nabla \eta)\|_{H^{s-1}} \|\vv V\|_{H^s}+\|\nabla \sqrt{h}\|_{H^{s-1}} \|\na\cdot(h^2 \nabla \eta)\|_{H^{s-1}} \|\nabla \cdot \vv V\|_{H^s}\\
&\lesssim \|\nabla \sqrt{h}\|_{H^s}   \|h^2\nabla \eta\|_{H^s}( \norm{\vv V}{H^s} + \norm{\nabla \cdot \vv V}{H^s} ).
	\end{aligned}\eeno

	By virtue of  \eqref{bounds of grad composite} with $F(r)=\sqrt{r}$ and \eqref{composite multiplier bounds} with $F(r)=r^2$, we get by using the assumption $\|\na h\|_{H^{s+2}}\leq C\e\leq c_0\leq 1$ that
	\beq\label{E 1}
\|\nabla \sqrt{h}\|_{H^{s+1}}\leq\|\nabla \sqrt{h}\|_{H^{[s]+2}}\lesssim\|\na h\|_{H^{s+2}}\lesssim\e,\quad
\|h^2\nabla \eta\|_{H^s}\lesssim\|\na\eta\|_{H^s}.
\eeq
Then we obtain
\beq\label{estimate of S21}
|\mathcal{S}_{21}|\lesssim\e\|\na\eta\|_{H^s}( \norm{\vv V}{H^s} + \norm{\nabla \cdot \vv V}{H^s} )\lesssim\|\eta\|_{X^s_\e}\|\vv V\|_{ \textbf{X}^s_\e}.
\eeq

{\bf For $\mathcal{S}_{23}$}, we rewrite as 
\beno
\mathcal{S}_{23}=-\bigl(\Lam^s(h^2 \nabla \eta)\,|\, \nabla  ([\Lam^s,\sqrt{h}] \nabla \cdot\vv V)\bigr)_{L^2}-\bigl(\Lam^s(h^2 \nabla \eta)\,|\, \nabla ([\Lam^s,\na\sqrt{h}]\cdot\vv V)\bigr)_{L^2},
\eeno
which together with \eqref{commutator estimate}, \eqref{composite commutator} and the assumption $\|\na h\|_{H^{s+2}}\leq C\e\leq c_0\leq 1$ gives rise to
\beno
|\cS_{23}|\lesssim\bigl(\|\na h\|_{H^{s+1}}\|\na\cdot\vv V\|_{H^s}+\|\na^2\sqrt h\|_{H^s}\|\vv V\|_{H^s}\bigr)\|h^2\na\eta\|_{H^s}.
\eeno
Using \eqref{E 1}, we obtain
\beq\label{estimate of S23}
|\mathcal{S}_{23}|\lesssim\e\|\na\eta\|_{H^s}( \norm{\vv V}{H^s} + \norm{\nabla \cdot \vv V}{H^s} )\lesssim\|\eta\|_{X^s_\e}\|\vv V\|_{ \textbf{X}^s_\e}.
\eeq

{\bf For $\mathcal{S}_{22}$ and $\mathcal{S}_{24}$,}  due to \eqref{composite commutator} and the assumption $\|\na h\|_{H^{s+2}}\leq C\e\leq c_0\leq 1$, we get
\beno
|\mathcal{S}_{22}|+|\mathcal{S}_{24}|
\lesssim\|\na h\|_{H^{s+1}}\| \nabla \cdot (\sqrt{h} \vv V)\|_{H^{s}}\|\na\eta\|_{H^s}\lesssim\e\bigl(\|\na\sqrt h\|_{H^s}\|\vv V\|_{H^s}+\|\sqrt h\na\cdot\vv V\|_{H^s}\bigr)\|\na\eta\|_{H^s},
\eeno
which along with \eqref{E 1} and \eqref{composition} yields to
\beq\label{estimate of S22 and S24}
|\mathcal{S}_{22}|+|\mathcal{S}_{24}|\lesssim\e\|\na\eta\|_{H^s}( \norm{\vv V}{H^s} + \norm{\nabla \cdot \vv V}{H^s} )\lesssim\|\eta\|_{X^s_\e}\|\vv V\|_{\vv X^s_\e}.
\eeq

Thanks to \eqref{estimate of S21},\eqref{estimate of S23} and \eqref{estimate of S22 and S24}, we deduce 
\beq\label{estimate on S2}
|\mathcal{S}_2|\lesssim\|\eta\|_{X^s_\e}\|\vv V\|_{\vv X^s_\e}\lesssim \cE_s(t).
\eeq

	\textit{Step 2.3. Estimate of $\mathcal{S}_3$.} Using the expression of $\mathcal{S}_3$, we have
	\begin{align*}
		\mathcal{S}_3 &=  \bigl\{\bigl(  [\Lam^s,\nabla h]  \nabla \cdot (h \sqrt{h} \nabla \eta)\,|\, \Lam^s \vv V  \bigr)_{L^2} 
		 +  \bigl([\Lambda^s, h \sqrt{h}]\na(\na h\cdot\vv V)\,|\, \Lam^s\na\eta) \bigr)_{L^2}\bigr\} \\
		& \qquad-\bigl\{\bigl(  \Lam^s  ( h \sqrt{h} \nabla \eta) \,|\, \na(\na h\cdot\Lambda^s \vv V )\bigr)_{L^2} -\bigl(\Lam^s\na(\na h \cdot \vv V) \,|\,h\sqrt h\Lam^s \nabla \eta  \bigr)_{L^2}\bigr\} \eqdefa \mathcal{S}_{31} + \mathcal{S}_{32}.
	\end{align*}

	For $\cS_{31}$, by virtue of \eqref{commutator estimate} and \eqref{composite commutator},  we have
	\beno
		|\mathcal{S}_{31}| \lesssim \|\nabla ^2 h\|_{H^{s-1}} \|\na(h\sqrt h \nabla \eta)\|_{H^{s-1}} \|\vv V\|_{H^s} + \|\na h\|_{H^s}\|\na(\na h\cdot\vv V)\|_{H^{s-1}} \|\na\eta\|_{H^s}
\eeno
	which along with \eqref{composite multiplier bounds} and the assumption $\|\na h\|_{H^{s+2}}\leq C\e\leq c_0\leq 1$ yields to
	\beq\label{estimate of S31}
|\mathcal{S}_{31}|\lesssim\|\na h\|_{H^s}
\|\vv V\|_{H^s}\|\na\eta\|_{H^s}\lesssim\e\|\vv V\|_{H^s}\|\na\eta\|_{H^s}
\lesssim\|\vv V\|_{H^s}\|\eta\|_{X^s_\e}.
	\eeq

For $\mathcal{S}_{32}$,  we rewrite it as
	\beno\begin{aligned}
	\mathcal{S}_{32} & = -\bigl( [\Lam^s, h \sqrt{h}] \nabla \eta \,|\, \na(\na h\cdot\Lambda^s \vv V )\bigr)_{L^2}+\bigl(\na([\Lam^s,\na h]\cdot\vv V) \,|\,h\sqrt h\Lam^s \nabla \eta  \bigr)_{L^2}\\
	&  = \bigl( \nabla \cdot ([\Lam^s, h \sqrt{h}] \nabla \eta )\,|\,\na h\cdot\Lambda^s \vv V \bigr)_{L^2}+\bigl(\na([\Lam^s,\na h]\cdot\vv V) \,|\,h\sqrt h\Lam^s \nabla \eta  \bigr)_{L^2}.
	\end{aligned}\eeno
	Thanks to \eqref{commutator estimate}, \eqref{composite commutator},\eqref{composite multiplier bounds} and the assumption \eqref{condition for h}, we get
	\beno\begin{aligned}
		|\mathcal{S}_{32}| &\lesssim\|\nabla h\|_{H^{s+1}} \|\nabla \eta\|_{H^{s}} \|\vv V\|_{H^{s}} + \|\nabla h\|_{H^{s+1}} \|\vv V\|_{H^{s}} \|\Lam^s\nabla\eta\|_{L^2}\\
	 &
	 \lesssim\e \|\eta\|_{H^{s+1}}\|\vv V\|_{H^s}\lesssim \|\eta\|_{X^s_\e}\|\vv V\|_{H^s},
	\end{aligned}\eeno
	which along with \eqref{estimate of S31}  implies
	\beq\label{estimate of S3}
|S_3|\lesssim\|\vv V\|_{H^s}\|\eta\|_{X^s_\e}\lesssim \cE_s(t).
	\eeq

	{\it Step 2.4. Estimate of $\cS$.} Due to \eqref{estimate of S1},\eqref{estimate on S2} and \eqref{estimate of S3}, we obtain
	\begin{align}\label{estimate of S}
		|\mathcal{S}| \le |\mathcal{S}_1| + \e|\mathcal{S}_2| + \e|\mathcal{S}_3| \lesssim \e \cE_s(t). 
	\end{align}

	\begin{remark}
	We remark that each term in $\cS_{32}$ cannot be controlled by $\cE_s(t)$ due to the absence of information regarding $\sqrt\e\|\na\vv V\|_{H^s}$ in $\cE_s(t)$. Nevertheless, there exists a cancellation between the two terms in $\cS_{32}$. This is why  we require $b_3=-c_3$ in the two dimensional case.
	\end{remark}

{\bf Step 3. Estimate of $\mathcal{T}$.}
By the expressions of $\vv F_h(\vv V,\eta)$ and $f_h(\vv V,\eta)$, we rewrite $\mathcal{T}$ as
	\begin{align*}
		\mathcal{T} = & -\f\e2\left\{   \left( \frac{1}{\sqrt{h}} \eta\nabla \eta \,\Big|\, \vv V  \right)_{H^s} + \left( \frac{1}{\sqrt{h}}  \nabla \cdot (\eta \vv V)\,\Big|\, \eta \right)_{H^s}\right\} \\
		& -\f\e2 \left( \frac{1}{\sqrt{h}} (   \f12 \nabla |\vv V|^2 + \vv V \cdot \nabla \vv V + \vv V \nabla \cdot \vv V  )  \,\Big|\, \vv V \right)_{H^s} \\
		& - \f\e2 \left\{  \left( \frac{1}{h^\f32} (  \f12 (\nabla h\cdot \vv V) \vv V - |\vv V|^2 \nabla h  ) \,\Big|\, \vv V  \right)_{H^s} - \left(  \frac{\eta}{2 h^\f32} \nabla h\cdot \vv V \,\Big|\, \eta  \right)_{H^s}   \right\}\\
		\eqdefa&-\f\e2\cT_1-\f\e2\cT_2-\f\e2\cT_3
	\end{align*}

	\textit{Step 3.1. Estimate of $\mathcal{T}_1$.} A direct calculation leads to
	\begin{align*}
		\mathcal{T}_{1} &=  \left( \Lam^s (\frac{\eta}{\sqrt{h}}\nabla \eta )   \,\Big|\, \Lam^s \vv V \right)_{L^2} +\left( \Lam^s  ( \frac{\eta}{\sqrt{h}} \nabla \cdot \vv V )  \,\Big|\, \Lam^s \eta  \right)_{L^2}
		+\left( \Lam^s  ( \frac{\vv V}{\sqrt{h}}\cdot \nabla\eta)  \,\Big|\, \Lam^s \eta  \right)_{L^2}\\
		&= \left\{\left( [\Lam^s,\frac{\eta}{\sqrt{h}}]  \nabla \eta    \,\Big|\, \Lam^s \vv V \right)_{L^2} +\left( [\Lam^s,\frac{\eta}{\sqrt{h}}] \nabla\cdot\vv V    \,\Big|\, \Lam^s \eta \right)_{L^2}
		+\left( [\Lam^s, \frac{\vv V}{\sqrt{h}}]\cdot \nabla\eta \,\Big|\, \Lam^s \eta  \right)_{L^2}\right\}\\
		&\qquad+\left\{-\left( \na(\frac{\eta}{\sqrt{h}})\Lam^s\eta \,\Big|\,\Lam^s\vv V\right)_{L^2} +\left( \frac{\vv V}{\sqrt{h}}\cdot\nabla\Lam^s \eta  \,\Big|\, \Lam^s \eta  \right)_{L^2}\right\}\\
		&\eqdefa \cT_{11}+\cT_{12}.
	\end{align*}
	For $\cT_{11}$, applying \eqref{commutator estimate} and then using \eqref{composite multiplier bounds} and the assumption \eqref{condition for h}, we have
	\beq\label{estimate of T11}
|\cT_{11}|\lesssim\Big\|\na(\f{\eta}{\sqrt h})\Big\|_{H^{s-1}}\|\eta\|_{H^s}\|\vv V\|_{H^s}+\Big\|\na(\f{\vv V}{\sqrt h})\Big\|_{H^{s-1}}\|\eta\|_{H^s}^2
\lesssim \|\vv V\|_{H^s}\|\eta\|_{H^s}^2.
\eeq

While for last term of $\cT_{12}$, integration  by parts yields to
\beno
\left( \frac{\vv V}{\sqrt{h}}\cdot\nabla\Lam^s \eta  \,\Big|\, \Lam^s \eta  \right)_{L^2}=-\f12\left(\na\cdot \frac{\vv V}{\sqrt{h}}\Lam^s \eta  \,\Big|\, \Lam^s \eta  \right)_{L^2}.
\eeno 
Then using the Sobolev's embedding $H^{s-1}(\R^n)\hookrightarrow L^\infty(\R^n)$ for $s>\f{n}{2}+1$, \eqref{composite multiplier bounds} and the assumption \eqref{condition for h}, we obtain
\beno
|\mathcal{T}_{12}|\lesssim\Big\|\na(\f{\eta}{\sqrt h})\Big\|_{H^{s-1}}\|\eta\|_{H^s}\|\vv V\|_{H^s}+\Big\|\na\cdot(\f{\vv V}{\sqrt h})\Big\|_{H^{s-1}}\|\eta\|_{H^s}^2
\lesssim\|\vv V\|_{H^s}\|\eta\|_{H^s}^2,
	\eeno
	which along with \eqref{estimate of T11} leads to
	\beq\label{estimate of T1}
|\mathcal{T}_{1}|\lesssim\|\vv V\|_{H^s}\|\eta\|_{H^s}^2\lesssim \cE_s(t)^{\f32}.
	\eeq

	\textit{Step 3.2. Estimate of $\mathcal{T}_2$.} For the first term of $\cT_2$,  denoting by ${\vv V}=(V_1,V_2)^T$, we have
	\beno\begin{aligned}
&\quad\left( \frac{1}{2\sqrt{h}} \nabla |\vv V|^2  \,\Big|\, \vv V \right)_{H^s}
=\sum_{j,k=1}^2\left(\Lam^s(\frac{1}{\sqrt{h}} V_j\p_kV_j) \,\Big|\,\Lam^sV_k \right)_{L^2}\\
&=\sum_{j,k=1}^2\Bigl\{\left([\Lam^s,\frac{V_j}{\sqrt{h}}]\p_kV_j) \,\Big|\,\Lam^sV_k \right)_{L^2}+\left(\p_k\Lam^sV_j \,\Big|\,\frac{V_j}{\sqrt{h}}\Lam^sV_k \right)_{L^2}\Bigr\}\\
&=\sum_{j,k=1}^2\Bigl\{\left([\Lam^s,\frac{V_j}{\sqrt{h}}]\p_kV_j) \,\Big|\,\Lam^sV_k \right)_{L^2}-\left(\Lam^sV_j \,\Big|\,\p_k\bigl(\frac{V_j}{\sqrt{h}}\bigr)\Lam^sV_k \right)_{L^2}\Bigr\}-\left(\Lam^s\vv V \,\Big|\,\frac{\vv V}{\sqrt{h}}\Lam^s\na\cdot\vv V \right)_{L^2}\\
&\eqdefa \cT_{21}-\left(\Lam^s\vv V \,\Big|\,\frac{\vv V}{\sqrt{h}}\Lam^s\na\cdot\vv V \right)_{L^2}.
	\end{aligned}\eeno
Then we get
\begin{align*}
\mathcal{T}_{2}&=\cT_{21}-\left(\Lam^s\vv V \,\Big|\,\frac{\vv V}{\sqrt{h}}\Lam^s\na\cdot\vv V \right)_{L^2}+\left(\Lam^s\bigl(\f{\vv V}{\sqrt h}\na\cdot\vv V\bigr) \,\Big|\,\Lam^s\vv V\bigr) \right)_{L^2}+\left(\Lam^s\bigl(\f{\vv V}{\sqrt h}\cdot\na\vv V\bigr) \,\Big|\,\Lam^s\vv V\bigr) \right)_{L^2}\\
&=\cT_{21}+\left([\Lam^s,\f{\vv V}{\sqrt h}]\na\cdot\vv V \,\Big|\,\Lam^s\vv V\bigr) \right)_{L^2}+\left([\Lam^s,\f{\vv V}{\sqrt h}]\cdot\na\vv V \,\Big|\,\Lam^s\vv V\bigr) \right)_{L^2}-\f12\left(\na\cdot\bigl(\f{\vv V}{\sqrt h}\bigr)\Lam^s\vv V\,\Big|\,\Lam^s\vv V\bigr) \right)_{L^2}
	\end{align*}

Similar derivation as \eqref{estimate of T1} yields to
\beq\label{estimate on T2}
|\mathcal{T}_{2}|\lesssim\Big\|\na(\f{\vv V}{\sqrt h})\Big\|_{H^{s-1}}\|\vv V\|_{H^s}^2
\lesssim\|\vv V\|_{H^s}^3\lesssim \cE_s(t)^{\f32}.
\eeq
	
	\textit{Step 3.3. Estimate of $\mathcal{T}_3$.} The classical tame estimate and \eqref{composite multiplier bounds} give rise to 
	\begin{align}\label{estimate of T3}
		|\mathcal{T}_{3}| \lesssim \Big\|\f{\na h}{h^{\f32}}\Big\|_{H^s}\|\vv V\|_{H^s}\bigl(\|\vv V\|_{H^s}^2 + \|\eta\|_{H^s}^2 \bigr) \lesssim \cE_s(t)^\f32. 
	\end{align}

{\it Step 3.4. Estimate of $\cT$.} Due to \eqref{estimate of T1}-\eqref{estimate of T3}, we get
	\begin{align}\label{estimate of T}
		|\mathcal{T}| \lesssim \e \cE_s(t)^\f32. 
	\end{align}
	
{\bf Step 4. Estimate of $\mathcal{N}$.} We estimate $\mathcal{N}_j\,(j=1,2)$ one by one.

\textit{Step 4.1. Estimate of $\mathcal{N}_1$.} Recall the definitions of  operators $\mathcal{P}_h^1$ and $\mathcal{P}_h^2$ in \eqref{def of Phi, d>1}, it follows that
	\begin{align*}
		\mathcal{N}_1 &=  a_1 \bigl\{  ( \nabla (h^2 \nabla \cdot \vv V_t) \,|\, \vv V   )_{H^s} - ( \nabla (h^2 \nabla \cdot \vv V) \,|\, \vv V_t   )_{H^s}   \bigr\}\\
		 &\quad + d_1 \bigl\{  (  \nabla \cdot (h^2 \nabla\eta_t) \,|\, \eta )_{H^s} -  (  \nabla \cdot (h^2 \nabla\eta) \,|\, \eta_t )_{H^s}    \bigr\} \\
		 &\quad + d_2 \bigl\{  ( \nabla \cdot (h\nabla h \,\eta_t )\,|\, \eta )_{H^s} - ( \nabla \cdot (h\nabla h \, \eta )\,|\, \eta_t )_{H^s}    \bigr\} \\
		 &\eqdefa a_1 \mathcal{N}_{11} + d_1\mathcal{N}_{12}  + d_2\mathcal{N}_{13} .
	\end{align*}

	For term $\mathcal{N}_{11}$, we have
	\begin{align*}
		\mathcal{N}_{11} &= - \bigl(\Lam^s( h^2 \nabla \cdot \vv V_t)\,|\, \Lam^s\nabla \cdot \vv V )_{L^2}+ \bigl(\Lam^s(h^2 \nabla \cdot \vv V)\,\big|\, \Lam^s\nabla \cdot \vv V_t \bigr)_{L^2}     \\
		& = -([\Lam^s,h^2 ] \nabla \cdot \vv V_t \,|\, \Lam^s \nabla \cdot\vv V  )_{L^2}+ \bigl( \Lam([\Lam^s ,h^2] \nabla \cdot \vv V)\,|\, \Lam^{s-1} \nabla \cdot \vv V_t \bigr)_{L^2} .
	\end{align*}
	which along with \eqref{composite commutator} and \eqref{condition for h}  yields to
	\begin{align}\label{estimate of N11}
		|\mathcal{N}_{11}| &\lesssim \norm{\nabla h}{H^{s+1}}\norm{\nabla \cdot \vv V_t}{H^{s-1}}\norm{\nabla \cdot \vv V}{H^s}
	 \lesssim \e \norm{\nabla \cdot \vv V_t}{H^{s-1}}\norm{\nabla \cdot \vv V}{H^s}  \lesssim \norm{\vv V_t}{{\bf X}_\e^{s-1}}\norm{\vv V}{\bf{X}^s_\e} .
	\end{align}
A similar argument gives rise to
	\begin{equation}\label{estimate of N12}
		|\mathcal{N}_{12}| \lesssim \|\na h\|_{H^{s+1}}\norm{\nabla  \eta_t}{H^{s-1}} \norm{\nabla \eta}{H^s}  \lesssim \e\norm{\nabla  \eta_t}{H^{s-1}}\norm{\nabla \eta}{H^s} \lesssim \norm{\eta_t}{X_\e^{s-1}} \|\eta\|_{X_\e^s} . 
	\end{equation}

	For $\mathcal{N}_{13}$, due to the tame estimate, \eqref{composite multiplier bounds} and the assumption \eqref{condition for h}, we have
    \beno\begin{aligned}
        |\mathcal{N}_{13}|  &\lesssim\norm{h\nabla h\,  \eta_t}{H^s} \norm{\nabla \eta}{H^s} + \norm{\nabla \cdot ( h\nabla h\,\eta )}{H^s} \norm{\eta_t}{H^s}\\
        & \lesssim \|\na h\|_{H^{s+1}}\norm{\eta_t}{H^s}  \norm{\eta}{H^{s+1}}  
        \lesssim \e\norm{\eta_t}{H^s}\norm{\eta}{H^{s+1}}  \lesssim \norm{\eta_t}{X_\e^{s-1}} \|\eta\|_{X_\e^s},
    \end{aligned}\eeno
    which together with  \eqref{estimate of N11} and \eqref{estimate of N12} hints
    \begin{align}\label{estimate of N1a}
    	|\mathcal{N}_1| \lesssim(  \norm{\vv V_t(t,\cdot) }{{\bf X}_\e^{s-1}} + \norm{\eta_t(t,\cdot)}{X_\e^{s-1}} )\cE_s(t)^\f12. 
    \end{align}
    
\textit{Step 4.2. Estimate of $(\vv V_t,\eta_t)$.} We shall use the equations in  \eqref{Boussinesq-like case 1} to derive the bounds of $(\vv V_t,\eta_t)$.
    
    \textit{Step 4.2.1. Estimate of $\norm{ \vv V_t}{{\bf X}_\e^{s-1}}$.} Thanks to \eqref{equivalent functional 1}, we have
    \beno
\norm{\vv V_t(t,\cdot)}{{\bf X}_\e^{s-1}}^2 \sim ( (1-\f\e2 \mathcal{P}_h^1) \vv V_t\,|\, \vv V_t  )_{H^{s-1}},
    \eeno 
from which and the first equation of \eqref{Boussinesq-like case 1}, we deduce that
   \beno \begin{aligned}
    \norm{\vv V_t(t)}{{\bf X}_\e^{s-1}}^2 &\lesssim \bigl(
   \|\sqrt{h}\nabla \eta\|_{H^{s-1}}+ \e\|\vv F_h(\vv V,\eta)\|_{H^{s-1}} 
    +\e\|\sqrt{h} \nabla (h\nabla h\cdot \nabla \eta)\|_{H^{s-1}} +\e\|\nabla h\nabla \cdot (h^\f32 \nabla \eta)\|_{H^{s-1}}\\
    &\qquad+\e\|\sqrt{h} \nabla h(\nabla h\cdot \nabla \eta) \|_{H^{s-1}}\bigr)\cdot\|\vv V_t\|_{H^{s-1}} + \e |\bigl( \sqrt{h} \nabla\nabla \cdot (h^2 \nabla \eta)\,|\, \vv V_t  \bigr)_{H^{s-1}}|\\
    &\eqdefa A_1+A_2. 
    \end{aligned}\eeno
    We remark here that $A_1$ involves the lower order terms, while $A_2$ pertains to the highest order term in the first equation of \eqref{Boussinesq-like case 1}.

{\it (i) For $A_1$,} using the tame estimate \eqref{composite multiplier bounds} and the assumption \eqref{condition for h}, we get
    \beno
A_1\lesssim\bigl(\|\eta\|_{H^s}+\e\|\na\eta\|_{H^s}+ \e\|\vv F_h(\vv V,\eta)\|_{H^{s-1}} \bigr)\cdot\|\vv V_t\|_{H^{s-1}}
\eeno
    Recalling the expression of $\vv F_h(\vv V,\eta)$ in \eqref{def of Fh and fh},  by using \eqref{composite multiplier bounds} and the assumption \eqref{condition for h}, we obtain 
    \begin{align}
    	\norm{\vv F_h(\vv V,\eta)}{H^{s-1}} &\lesssim  \norm{\eta \nabla \eta}{H^{s-1}} + \norm{\nabla |\vv V|^2}{H^{s-1}} + \norm{\vv V\cdot \nabla \vv V}{H^{s-1}} + \norm{\vv V\nabla \cdot \vv V}{H^{s-1}} + \norm{\vv V}{H^{s-1}}^2\notag \\
    	& \lesssim \norm{\eta}{H^s}^2 + \norm{\vv V}{H^s}^2 \lesssim \cE_s(t), \notag
    \end{align}
    which implies 
    \begin{align}\label{estimate of A1}
    	A_1\lesssim\bigl(\|\eta\|_{X^s_\e}+\e\norm{\eta}{H^s}^2+\e\norm{\vv V}{H^s}^2 \bigr) \norm{\vv V_t}{H^{s-1}} \lesssim\bigl(\cE_s(t)^{\f12}+\e \cE_s(t)\bigr) \norm{\vv V_t}{{\bf X}_\e^{s-1}}.  
    \end{align}

  {\it (ii) For $A_2$,} using integration by part, we have 
    \beno
\bigl( \sqrt{h} \nabla\nabla \cdot (h^2 \nabla \eta)\,|\, \vv V_t  \bigr)_{H^{s-1}}&=-\bigl(\sqrt{h}\nabla \cdot (h^2 \nabla \eta)\,|\, \na\cdot\vv V_t  \bigr)_{H^{s-1}}
-\bigl(\na\sqrt{h}\nabla \cdot (h^2 \nabla \eta)\,|\,\vv V_t  \bigr)_{H^{s-1}}.
\eeno
Due to the tame estimate, \eqref{composite multiplier bounds} and the assumption \eqref{condition for h}, we get
\beno\begin{aligned}
A_2&\lesssim\e\|\nabla \cdot (h^2 \nabla \eta)\|_{H^{s-1}}\|\na\cdot\vv V_t\|_{H^{s-1}}
+\e\|\na\sqrt{h}\|_{H^{s-1}}\|\nabla \cdot (h^2 \nabla \eta)\|_{H^{s-1}}\|\vv V_t\|_{H^{s-1}}\\
&\lesssim\e^{\f12}\|\na\eta\|_{H^s}\cdot\e^{\f12}\bigl(\|\na\cdot\vv V_t\|_{H^{s-1}}+\|\vv V_t\|_{H^{s-1}}\bigr)\lesssim\|\eta\|_{X_\e^s}\|\vv V_t\|_{\bf{X}_\e^{s-1}}.
\end{aligned}\eeno
Then  there holds
\beq\label{estimate of A2}
A_2\lesssim \cE_s(t)^{\f12}\|\vv V_t\|_{\bf{X}_\e^{s-1}}.
\eeq

Thanks to \eqref{estimate of A1}  and \eqref{estimate of A2}, we obtain
    \begin{align*}
    	\norm{\vv V_t(t,\cdot)}{{\bf X}_\e^{s-1}}^2\lesssim A_1+A_2\lesssim\bigl(\cE_s(t)^\f12 + \e \cE_s(t)\bigr)\norm{\vv V_t(t,\cdot)}{{\bf X}_\e^{s-1}},
    \end{align*}
    which hints
 \beq\label{estimate of Vt}
\norm{\vv V_t(t,\cdot)}{{\bf X}_\e^{s-1}}\lesssim \cE_s(t)^\f12 + \e \cE_s(t).
 \eeq
    
   \textit{Step 4.2.2. Estimate of $\norm{\eta_t}{X_\e^{s-1}}$.} Due to \eqref{equivalent functional 1}, we have
    \beno
\norm{\eta_t(t,\cdot)}{X_\e^{s-1}}^2 &\sim \bigl( (1-\f\e2 \mathcal{P}_h^2) \eta_t\,| \, \eta_t \bigr)_{H^{s-1}}
    \eeno
    from which and the second equation of \eqref{Boussinesq-like case 1}, we derive that 
\beno\begin{aligned}
\norm{\eta_t(t,\cdot)}{X_\e^{s-1}}^2&\lesssim
\Bigl(\|\nabla \cdot (\sqrt{h} \vv V)\|_{H^{s-1}} +\e\|f_h(\vv V,\eta)\|_{H^{s-1}}\Bigr)\|\eta_t\|_{H^{s-1}}\\
&\qquad+\e\Bigl(\|h^2 \nabla \nabla \cdot (\sqrt{h} \vv V)\|_{H^{s-1}} +\|h \nabla h\nabla \cdot (\sqrt{h}\vv V)\|_{H^{s-1}} \\
&\qquad+\| h\sqrt{h} \nabla (\nabla h\cdot \vv V) \|_{H^{s-1}}+\| \sqrt{h} \nabla h (\nabla h\cdot \vv V) \|_{H^{s-1}} \Bigr)\|\na\eta_t\|_{H^{s-1}}\\
&\eqdefa B_1+B_2.
\end{aligned}\eeno
We remark here that the terms in $B_2$ are obtain by applying  integration by parts to the corresponding terms in $\bigl( (1-\f\e2 \mathcal{P}_h^2) \eta_t\,| \, \eta_t \bigr)_{H^{s-1}}$.

{\it (i)  For $B_1$,} by virtue of the tame estimate,\eqref{composite multiplier bounds} and the assumption \eqref{condition for h}, we get
\beno
B_1\lesssim\bigl(\|\vv V\|_{H^s}+\e\|f_h(\vv V,\eta)\|_{H^{s-1}}\bigr)\|\eta_t\|_{H^{s-1}}.
\eeno
Due to the expression of $f_h(\vv V,\eta)$ in \eqref{def of Fh and fh}, by using \eqref{composite multiplier bounds} and the assumption \eqref{condition for h}, we obtain
\beno
\|f_h(\vv V,\eta)\|_{H^{s-1}}\lesssim \| \nabla \cdot (\eta \vv V)\|_{H^{s-1}} +\|\eta\vv V \|_{H^{s-1}}\lesssim\|\eta\|_{H^s}\|\vv V\|_{H^s},
\eeno 
which implies
\beq\label{estimate of B1}
B_1\lesssim\bigl(\|\vv V\|_{H^s}+\e\|\eta\|_{H^s}\|\vv V\|_{H^s}\bigr)\|\eta_t\|_{H^{s-1}}\lesssim\bigl(\cE_s(t)^{\f12}+\e \cE_s(t)\bigr)\|\eta_t\|_{H^{s-1}}.
\eeq

{\it (ii) For $B_2$,} since
\beno
 \nabla \nabla \cdot (\sqrt{h} \vv V)=\na(\na\sqrt h\cdot\vv V)+\na(\sqrt h\na\cdot\vv V),
\eeno
using \eqref{composite multiplier bounds}, the tame estimate and the assumption \eqref{condition for h},  we deduce that
\beno
\|h^2 \nabla \nabla \cdot (\sqrt{h} \vv V)\|_{H^{s-1}}
\lesssim\|\na\sqrt h\cdot\vv V\|_{H^s}+\|\sqrt h\na\cdot\vv V\|_{H^s}
\lesssim\|\vv V\|_{H^s}+\|\na\cdot\vv V\|_{H^s}.
\eeno
Similarly, there holds
\beno\begin{aligned}
&\|h \nabla h\nabla \cdot (\sqrt{h}\vv V)\|_{H^{s-1}} +\| h\sqrt{h} \nabla (\nabla h\cdot \vv V) \|_{H^{s-1}}+\| \sqrt{h} \nabla h (\nabla h\cdot \vv V) \|_{H^{s-1}}\\
&\quad
\lesssim\|\vv V\|_{H^{s-1}}+\|\na\cdot\vv V\|_{H^{s-1}}\lesssim\|\vv V\|_{H^s}.
\end{aligned}\eeno
Then we obtain
\beq\label{estimate of B2}
B_2\lesssim\e^{\f12}\bigl(\|\vv V\|_{H^s}+\|\na\cdot\vv V\|_{H^s}\bigr)\cdot\e^{\f12}\|\na\eta_t\|_{H^{s-1}}\lesssim\|\vv V\|_{{\bf X}_\e^s}\|\eta_t\|_{X_\e^{s-1}}
\lesssim \cE_s(t)^{\f12}\|\eta_t\|_{X_\e^{s-1}}.
\eeq
   
   Hence, by virtue of \eqref{estimate of B1} and \eqref{estimate of B2}, we derive that 
   \beno
\|\eta_t\|_{X_\e^{s-1}}^2\lesssim B_1+B_2\lesssim\bigl(\cE_s(t)^{\f12}+\e\cE_s(t)\bigr)\|\eta_t\|_{X_\e^{s-1}}
   \eeno
   which gives rise to
    \beq\label{estimate on etat}
    	\|\eta_t\|_{X_\e^{s-1}}\lesssim \cE_s(t)^{\f12}+\e\cE_s(t).
    \eeq
    
   {\it Step 4.2.3. Estimates of $(\vv V_t,\eta_t)$ and $\cN_1$.} Combining \eqref{estimate of Vt} and \eqref{estimate on etat}, we obtain 
   \beq\label{estimate of Vt and etat}
\|\vv V_t\|_{{\bf X}_\e^{s-1}}+\|\eta_t\|_{X_\e^{s-1}}\lesssim \cE_s(t)^{\f12}+\e \cE_s(t).
   \eeq

Using \eqref{estimate of Vt and etat}, we deduce from \eqref{estimate of N1a} that
    \beq\label{estimate on N1}
    |\mathcal{N}_1| \lesssim \cE_s(t) (1+ \e\cE_s(t)^\f12  ). 
    \eeq
    
 \noindent\textit{Step 4.3. Estimate of $\mathcal{N}_2$.} For the first term of $\cN_2$, we get by using integration by parts that
 \beno
\bigl(\sqrt{h} \nabla ( h\nabla h\cdot \nabla \eta) \,|\,\vv V\bigr)_{H^s}
=-\bigl(\sqrt{h} h\nabla h\cdot \nabla \eta \,|\,\nabla\cdot\vv V \bigr)_{H^s}
-\bigl(\na\sqrt{h} (h\nabla h\cdot \nabla \eta) \,|\,\vv V \bigr)_{H^s}
 \eeno
 which along with the tame estimate,\eqref{composite multiplier bounds} and the assumption \eqref{condition for h} implies
 \beno\begin{aligned}
\bigl|\bigl(\sqrt{h} \nabla ( h\nabla h\cdot \nabla \eta) \,|\,\vv V\bigr)_{H^s}\bigr|
&\lesssim\|\na h\|_{H^s}\|\na\eta\|_{H^s}\bigl(\|\na\cdot\vv V\|_{H^s}+\|\vv V\|_{H^s}\bigr)\\
&\lesssim\e^{\f12} \|\na\eta\|_{H^s}\cdot\e^{\f12}\bigl(\|\na\cdot\vv V\|_{H^s}+\|\vv V\|_{H^s}\bigr)\lesssim\|\eta\|_{X_\e^s}\|\vv V\|_{{\bf X}_\e^s}.
 \end{aligned}\eeno
 The same estimate also holds for the remainder terms of $\cN_2$. Then we obtain
\beq\label{estimate on N2}
    |\mathcal{N}_2| \leq \|\eta\|_{X_\e^s}\|\vv V\|_{{\bf X}_\e^s} \lesssim\cE_s(t).
 \eeq

\textit{Step 4.4. Estimate of $\mathcal{N}$.}
 Thanks to \eqref{estimate on N1} and \eqref{estimate on N2}, we get
    \begin{align}\label{estimate on N}
    |\mathcal{N}| \lesssim\e|\mathcal{N}_1| +\e| \mathcal{N}_2 | \lesssim \e\cE_s(t) ( 1+ \e\cE_s(t)^\f12 ). 
    \end{align}

{\bf Step 5. The final energy estimate.} Combining estimates  \eqref{estimate of S}, \eqref{estimate of T} and \eqref{estimate on N}, we deduce from \eqref{calculation of energy} that
    \begin{align*}
    	\frac{d}{dt} E_s(t) \lesssim \e\cE_s(t) ( 1+ \e\cE_s(t)^\f12 ),\quad \forall t\in [0,\f{T_0}{\e}].
    \end{align*}
Then we have
\beno
E_s(t)\leq E_s(0)+C\e t\max_{\tau\in [0,t]}\cE_s(\tau)\bigl( 1+ \e\cE_s(\tau)^\f12 \bigr),\quad \forall t\in [0,\f{T_0}{\e}],
\eeno
which, along with \eqref{equivalent functional 1}, implies that there exist constants $C_1>1$ and $C_2>1$ such that \eqref{priori estimate for case 1} holds.

    For $n=1$, the proof of \eqref{priori estimate for case 1} is much  simpler since there is no distinction between $\nabla \cdot \vv V$ and $\nabla \vv V$. We omit the details of the proof. Thus, the proof of the proposition is concluded.
	\end{proof}
	\begin{remark}
(1). Thanks to Lemma \ref{technical lemma 1}, we have
	\beno
E_s(t)\sim\|\vv V\|_{{\bf X}_\e^s}^2+\|\eta\|_{X_\e^s}^2,
	\eeno 
under the condition that  $a_1>0$ and $d_1>0$ for the parameters appearing in $\cP_h^1$ and $\cP_h^2$. However, there are no restrictions on the parameters $a_2$ and $d_2$ in $\cP_h^1$ and $\cP_h^2$. The restriction $a_2=0$ in Proposition \ref{priori energy estimate} and Theorem \ref{main results 1} for $n=2$ is to cancel the following uncontrolled terms
\beno
a_2\bigl\{ - ( h\na h\cdot \vv V_t \,|\, \na\cdot\vv V   )_{H^s} + ( h\na h\cdot \vv V \,|\,\na\cdot \vv V_t   )_{H^s}  \bigr\}
\eeno
which are parts of $\left( \mathcal{P}_h^1 \vv V_t \,|\, \vv V  \right)_{H^s} - \left(  \mathcal{P}_h^1\vv V \,|\, \vv V_t \right)_{H^s}$.
These terms arise during the procedure of the \textit{a priori} energy estimate of \eqref{Boussinesq-like case 1}. Indeed,  since there lacks  the information of $\|\vv V_t\|_{H^s}$ for $n=2$, the above terms can not be controlled except when $a_2=0$.

(2). To control the term $\mathcal{S}$, which involves the symmetric linear parts of the system \eqref{Boussinesq-like case 1}, we impose the condition $\nabla h =O(\varepsilon)$ to ensure that $\mathcal{S}$ is of order $O(\e)$. For the term $\e\mathcal{N}_2$, which involves the lower order linear parts of the system \eqref{Boussinesq-like case 1}, the condition  $\nabla h =(\sqrt{\e})$ is sufficient to guarantee that $\e\mathcal{N}_2$ is of order $O(\e)$.

	\end{remark}
	
	\subsection{ The proof of Theorem \ref{main results 1}}
With the {\it a priori } energy estimate in Proposition \ref{priori energy estimate}, we can complete the proof of Theorem \ref{main results 1} in this subsection.

	\begin{proof}[Proof of Theorem \ref{main results 1}]

	Assuming that  $c_0\in(0,1]$ is the constant stated in Lemma \ref{technical lemma 1}, there exists $\e_0\in(0,1)$, such that for any $\e\in (0,\e_0)$,
\beno
\|\na h\|_{H^{s+2}}\leq C\e\leq c_0\leq1.
\eeno
Then Proposition \ref{priori energy estimate} for the  {\it a priori} estimates  of \eqref{Boussinesq-like system} holds. By virtue of the {\it a priori } estimate \eqref{priori estimate for case 1} in Proposition \ref{priori energy estimate}, we obtain a unique solution 
$(\vv V,\eta)\in C\bigl([0,T_0/\e];{\bf X}_\e^s\times X_\e^s\bigr)\cap C^1\bigl([0,T_0/\e];{\bf X}_\e^{s-1}\times X_\e^{s-1}\bigr)$ for \eqref{Boussinesq-like system}with the initial data $(\vv V_0,\eta_0)\in {\bf X}_\e^s\times X_\e^s$, by applying the existence theory for the linear hyperbolic systems and constructing the approximated solution via Picard iteration. We omit the details of the proof of {\it existence} and {\it uniqueness}.

In this proof, we only show how to determine the constant $T_0>0$ and establish the uniform energy estimate \eqref{total energy estimate for case 1}. The proof relies on the continuity argument.

Consider  a sufficiently smooth solution $(\vv V,\eta)$  to \eqref{Boussinesq-like system} -\eqref{initial data} over the time interval $[0,\f{T_0}{\e}]$, subjected to the constraint \eqref{case for slow oscillation}. Due to Proposition \ref{priori energy estimate}, for any $\e\in(0,\e_0)$, there exist $C_1>1$ and $C_2>1$ such that
 \beno
\cE_s(t)
\leq C_1\cE_s(0)+C_2\e t\max_{\tau\in[0,t]}\cE_s(\tau)\bigl(1+\e\cE_s(\tau)^{\f12}\bigr),
\quad\forall\,t\in[0,\f{T_0}{\e}],
\eeno
where $\cE_s(t)=\|\vv V\|_{{\bf X}_\e^s}^2+\|\eta\|_{X_\e^s}^2$ and $\cE_s(0)=\|\vv V_0\|_{{\bf X}_\e^s}^2+\|\eta_0\|_{X_\e^s}^2$.
    
Assuming that    
\beq\label{ansatz for case 1}
\cE_s(t)\leq 4C_1 \cE_s(0),\quad\forall\,t\in[0,\f{T_0}{\e}],
\eeq
we have for any $t\in[0,\f{T_0}{\e}]$,
\beno
\max_{\tau\in[0,t]}\cE_s(\tau)\leq C_1\cE_s(0)+C_2\e t\cdot\bigl(1+2\sqrt{C_1}\cE_s(0)^{\f12}\bigr)\cdot\max_{\tau\in[0,t]}\cE_s(\tau).
\eeno

Taking $T_0=\f{1}{2C_2\bigl(1+2\sqrt{C_1}\cE_s(0)^{\f12}\bigr)}$, we get for any $\e\in(0,\e_0)$,
\beno
\max_{\tau\in[0,t]}\cE_s(\tau)\leq 2C_1\cE_s(0),\quad\forall\,t\in[0,\f{T_0}{\e}].
\eeno
This improves the ansatz \eqref{ansatz for case 1}. The continuity argument is closed.

Thus, under the assumption of Theorem \ref{main results 1}, the Cauchy problem \eqref{Boussinesq-like system} -\eqref{initial data} admits a unique solution $(\vv V,\eta)$ over time interval $[0,\f{T_0}{\e}]$ satisfying the uniform energy estimate \eqref{total energy estimate for case 1}. This proves Theorem \ref{main results 1}.
\end{proof}
      
    \section{Proof of Theorem \ref{main results 2}}

    The goal of this section is to prove Theorem \ref{main results 2}, establishing the long time existence of solutions to the Cauchy problem of \eqref{Boussinesq-like system} ($n=1$) with a fast oscillating bottom, under the condition specified in \eqref{case for fast oscillation}. 
The {\it existence} and {\it uniqueness} of solutions follow from the standard mollification method and the Picard (Cauchy-Lipschitz) existence theorem, analogous to the approach in Section 5 of \cite{SWX}. Here, we employ a continuity argument combined with {\it a priori} energy estimates to establish both the existence time interval and the closure of the energy estimates, thereby completing the proof of Theorem \ref{main results 2}.

    \subsection{Reduction of the system}
    
    When $n=1$ and the parameters $\{ a_i,d_i \}_{i=1,2}, \{ b_i,c_i \}_{1\le i \le 4}$ fulfill \eqref{case for fast oscillation}, i.e.,
    \beno
     a_1>0,\,a_2=0,\,d_1=d_2=0,\,b_1=c_1<0,\,b_2+c_2+b_3+c_3=0,\,(b_4,d_4)\in\R^2,
   \eeno
    there hold
    \beno
    \cP_h^1 = a_1 \p_x ( h^2 \p_x \quad ),\quad \mathcal{P}_h^2 =0.
    \eeno

    By rewriting $F_h(V,\eta)$ and $f_h(V,\eta)$ into
    \begin{align*}
    	 F_h(V,\eta) =  \frac{1}{\sqrt h} \eta \p_x \eta + \frac{3}{\sqrt h} V \p_x V - \frac{\p_x h}{2h\sqrt h}|V|^2 , \quad f_h(V,\eta) = \p_x \bigl(  \frac{\eta V}{\sqrt{h}} \bigr),
    \end{align*}
    the system \eqref{Boussinesq-like system} is reduced to 
    \beq\label{Boussinesq-like, d=1}
    	 \left\{ \begin{aligned}
    	 	 & (1-\f\e2 \cP_h^1  ) \p_t V + \sqrt{h} \p_x \eta + \f\e2 B_h(D) \eta + \frac{\e}{2\sqrt h} \eta \p_x \eta +\frac{3\e }{2\sqrt h} V \p_x V =\frac{\e \p_x h }{4 h^{\f32}} |V|^2 ,\\
    	 	 &\p_t \eta + \p_x( \sqrt{h} V ) + \f\e2 C_h(D) V+\f\e2 \p_x \bigl( \frac{\eta V}{\sqrt h} \bigr)=0,
    	 \end{aligned}  \right.
    \eeq
    where the operators $B_h(D),C_h(D)$ are defined by
    \beno\begin{aligned}
    	&B_h(D) f \eqdefa b_1 \sqrt{h} \p_x^2 (h^2 \p_x f) + (b_2+b_3)h^{\f32}\p_x h \p_x^2f + r_1(h)\p_xf,\\
    	&C_h(D) g \eqdefa c_1 \partial_x \big( h^2 \p_x^2 ( \sqrt{h} g) \big) + (c_2+c_3)\p_x^2 \bigl( h^{\f32}\p_xh g\bigr)-\p_x(r_2(h)g),
    	\end{aligned}\eeno
    	and 
 \beq\label{def of Rh12}
\begin{aligned}
  &r_1(h)\eqdefa(b_2 + \f32 b_3 + b_4) \sqrt h (\p_x h)^2  + b_2 h^\f32 \p_x^2 h, \\
  &r_2(h)\eqdefa(c_2 + \f32 c_3 - c_4) \sqrt h (\p_x h)^2 + c_2 h^\f32 \p_x^2 h .
\end{aligned}
\eeq

Notice that $b_1=c_1$ and $b_2+b_3+c_2+c_3=0$. Then we obtain
\beq\label{def of BC}\begin{aligned}
&B_h(D)f= b_1 \sqrt{h} \p_x^2 (h^2 \p_x f) + (b_2+b_3)h^{\f32}\p_x h \p_x^2f+r_1(h)\p_xf,\\
&C_h(D)g=b_1 \partial_x \big( h^2 \p_x^2 ( \sqrt{h}g) \big) -(b_2+b_3)\p_x^2 \bigl( h^{\f32}\p_xh g\bigr)-\p_x(r_2(h)g).
\end{aligned}\eeq
Moreover, there holds
\beq\label{key 1}
\bigl(B_h(D)f\,|\,g\bigr)_{L^2}+\bigl(C_h(D)g\,|\,f\bigr)_{L^2}=\bigl(r(h)\p_x f\,|\,g\bigr)_{L^2},\quad\forall\,f,g\in\cS(\R),
\eeq
where 
\beq\label{def of Rh}
r(h)=r_1(h)+r_2(h)=(\f12b_3 + \f12 c_3 + b_4-c_4) \sqrt h (\p_x h)^2  + (b_2+c_2) h^\f32 \p_x^2 h.
\eeq

\begin{remark}
(1). Formula \eqref{key 1} reveals that the principal parts of the operators  $B_h(D)$ and $C_h(D)$ possess antisymmetric properties, leading to cancellation in \eqref{key 1}. Despite this favorable cancellation, \eqref{key 1} still induces a loss of derivatives in the highest-order energy estimates  (e.g.,estimates for $(\p_t^2\eta,\p_t^2V)_{L^2}$). We will study this issue more closely during the derivation of the highest-order energy estimates.

(2). The terms $\f\e2 B_h(D) \eta$ and $\f\e2 C_h(D) \eta$ in \eqref{Boussinesq-like, d=1} represent dispersive effects. The operators $B_h(D)$ and $C_h(D)$ behave like $-\p_x^3$, and their essential properties are summarized in the following lemma.
\end{remark}

Before presenting the lemma, we state a useful interpolation inequality as follows:
\beq\label{interpolation}
\e^{\f{j}{2}}\|\p_x^jf\|_{L^2}\lesssim\|f\|_{L^2}^{1-\f{j}{k}}\cdot\bigl( \e^{\f k2} \|\p_x^kf\|_{L^2}\bigr)^{\f{j}{k}}\lesssim\|f\|_{X^0_{\e^k}},\quad 0<j<k,\, j,k\in\Z_{\geq 1}.
\eeq

\begin{lemma}\label{lem for BC}
Let $b_1<0$, $h-1\in H^5(\R)$  satisfy
   \beq\label{condition for h 2}
2\geq h(x)\geq h_0>0,\quad\|\p_x h\|_{H^4}\leq 1,
   \eeq
 for some $h_0>0$. There exists $\e_1\in(0,1)$ such that for all $\e\in(0,\e_1)$,  there hold for any $f,g\in\cS(\R)$,
 \begin{align}
&(1).\quad \bigl( (\sqrt h\p_x +\f\e2 B_h(D))f\,|\,\p_xf\bigr)_{L^2}\sim\|f_x\|_{X_\e^0}^2,\label{B 1}\\
&(2). \quad\bigl\|(\sqrt h\p_x +\f\e2 B_h(D))f\|_{L^2}^2\sim\|f_x\|_{X_{\e^2}^0}^2,\label{B 2}\\
&(3).\quad\|f_x\|_{X^0_{\e^2}}^2-\|f\|_{L^2}^2\lesssim \bigl\|\bigl(\p_x(\sqrt h\cdot)+\f\e2 C_h(D)\bigr)f\|_{L^2}^2\lesssim\|f_x\|_{X^0_{\e^2}}^2+\|f\|_{L^2}^2,\label{C 1}\\
&(4).\quad \|f_x\|_{X_\e^0}^2-\|f\|_{L^2}^2\lesssim \bigl((\p_x(\sqrt h\cdot)+\f\e2 C_h(D))f\,|\,\p_xf\bigr)_{L^2}\|f_x\|_{X_\e^0}^2+\|f\|_{L^2}^2.\label{C 2}
 \end{align}
\end{lemma}
\begin{proof} We proceed to verify the equivalences one by one.

{\bf (1). Proof of \eqref{B 1}.} Using expression of $B_h(D)$ in \eqref{def of BC}, we have
\beno\begin{aligned}
\bigl( (\sqrt h\p_x &+\f\e2 B_h(D))f\,|\,\p_xf\bigr)_{L^2}
=\bigl(\sqrt h\p_xf\,|\,\p_xf\bigr)_{L^2}-\f{b_1 \e}{2}\bigl(  \p_x (h^2 \p_x f)\,|\,\p_x(\sqrt{h}\p_xf)\bigr)_{L^2}\\
&\quad+\f{(b_2+b_3)\e}{2}\bigl(h^{\f32}\p_x h \p_x^2f\,|\,\p_xf\bigr)_{L^2}+\f\e2\bigl( r_1(h)\p_xf\,|\,\p_xf\bigr)_{L^2}\\
&=\bigl( \sqrt h\p_xf\,|\,\p_xf\bigr)_{L^2}-\f{b_1 \e}{2}\bigl(h^{\f52}\p_x^2f\,|\,\p_x^2f)\bigr)_{L^2}+\f\e2 I_1,
\end{aligned}\eeno
where 
\beno
I_1\eqdefa(b_2+b_3-\f52b_1)\bigl(h^{\f32}\p_x h \p_x^2f\,|\,\p_xf\bigr)_{L^2}+\bigl( [r_1(h)-b_1h^{\f12}(\p_xh)^2]\p_xf\,|\,\p_xf\bigr)_{L^2}.
\eeno

For $I_1$, we get
\beno
I_1=\bigl(\{-\f12(b_2+b_3-\f52b_1)\p_x(h^{\f32}\p_x h)+r_1(h)-b_1h^{\f12}|\p_xh|^2\}\p_xf\,|\,\p_xf\bigr)_{L^2},
\eeno
which along with the expression of $r_1(h)$ in \eqref{def of Rh12} and the condition \eqref{condition for h 2} yields to
\beno\begin{aligned}
|I_1|&\lesssim\bigl(\|h^{\f32}\p_x^2 h\|_{L^\infty}+\|h^{\f12}(\p_xh)^2\|_{L^\infty}\bigr)\cdot\|f_x\|_{L^2}^2\\
&
\lesssim\bigl(\|\p_x^2h\|_{H^1}+\|\p_xh\|_{H^1}^2\bigr)\cdot\|f_x\|_{L^2}^2
\lesssim\|f_x\|_{L^2}^2.
\end{aligned}\eeno
Using the condition $0<h_0\leq h\leq 2$ and $b_1<0$ lead to 
\beno
\bigl( \sqrt h\p_xf\,|\,\p_xf\bigr)_{L^2}-\f{b_1 \e}{2}\bigl(h^{\f52}\p_x^2f\,|\,\p_x^2f)\bigr)_{L^2}\sim\|f_x\|_{L^2}^2+\e\|\p_xf_x\|_{L^2}^2\sim\|f_x\|_{X^0_\e}^2,
\eeno
Consequently, for sufficiently small $\e\in(0,1)$, there holds
\beno
\bigl( (\sqrt h\p_x+\f\e2 B_h(D))f\,|\,\p_xf\bigr)_{L^2}\sim\|f_x\|_{X_\e^0}^2.
\eeno
This is exactly \eqref{B 1}.

\smallskip

{\bf (2). Proof of \eqref{B 2}.} Thanks to the expression of $B_h(D)$, we have
\beq\label{eq 1 for B 2}
(\sqrt h\p_x +\f\e2 B_h(D))f=\sqrt h\p_xf+\f\e2b_1 h^{\f52}\p_x^3f+L_{h,1}(f),
\eeq
where 
\beno
L_{h,1}(f)= \f\e2(b_2+b_3+4b_1)h^{\f32}\p_x h \p_x^2f+ \f\e2\bigl(2b_1h^{\f12}\p_x(h\p_xh)+r_1(h)\bigr)\p_xf.
\eeno
Then we obtain
\beno
\|(\sqrt h\p_x +\f\e2 B_h(D))f\|_{L^2}^2
=\|\sqrt h\p_xf+\f\e2b_1 h^{\f52}\p_x^3f\|_{L^2}^2+
\|L_{h,1}(f)\|_{L^2}^2
+2\bigl(\sqrt h\p_xf+\f\e2b_1 h^{\f52}\p_x^3f\,\big|\,L_{h,1}(f)\bigr),
\eeno
and 
\beno
2|\bigl(\sqrt h\p_xf+\f\e2b_1 h^{\f52}\p_x^3f\,\big|\,L_{h,1}(f)\bigr)|
\leq\f12\|\sqrt h\p_xf+\f\e2b_1 h^{\f52}\p_x^3f\|_{L^2}^2
+2\|L_{h,1}(f)\|_{L^2}^2,
\eeno
which imply that
\beq\label{eq 2 for B 2}\begin{aligned}
\f12\|\sqrt h\p_xf+\f\e2b_1 h^{\f52}\p_x^3f\|_{L^2}^2-
\|L_{h,1}(f)\|_{L^2}^2&\leq \|(\sqrt h\p_x +\f\e2 B_h(D))f\|_{L^2}^2\\
&
\leq\f32\|\sqrt h\p_xf+\f\e2b_1 h^{\f52}\p_x^3f\|_{L^2}^2+3
\|L_{h,1}(f)\|_{L^2}^2.
\end{aligned}\eeq

{\it i). Estimate of $\|L_{h,1}(f)\|_{L^2}$.}
By virtue of $0<h_0\leq h\leq 2$ and the expression of $r_1(h)$, we get
\beno
\begin{aligned}
\|L_{h,1}(f)\|_{L^2}&\lesssim\e\|h^{\f32}\p_x h\|_{L^\infty}\|\p_x^2f\|_{L^2}+\e\bigl(\|h^{\f32}\p_x^2 h\|_{L^\infty}+\|h^{\f12}(\p_x h)^2\|_{L^\infty}\bigr)\|\p_xf\|_{L^2}\\
&\lesssim\e\|\p_xh\|_{H^1}\|\p_x^2f\|_{L^2}+\e\bigl(\|\p_x^2h\|_{H^1}+\|\p_xh\|_{H^1}^2\bigr)\|\p_xf\|_{L^2},
\end{aligned}
\eeno
which together with the condition $\|\p_xh\|_{H^3}\leq 1$ yields to
\beno
\|L_{h,1}(f)\|_{L^2}\lesssim\e\|f_x\|_{L^2}+\e^{\f12}\cdot\e^{\f12 }\|\p_xf_x\|_{L^2}.
\eeno
Using interpolation inequality \eqref{interpolation}, we obtain 
\beq\label{eq 4 for B 2}
\|L_{h,1}(f)\|_{L^2}\lesssim\e^{\f12}\cdot\|f_x\|_{X^0_{\e^2}}.
\eeq

{\it ii). Estimate of $\|\sqrt h\p_xf+\f\e2b_1 h^{\f52}\p_x^3f\|_{L^2}$.}
It is easy to check that
\beno\begin{aligned}
\|\sqrt h\p_xf+\f\e2b_1 h^{\f52}\p_x^3f\|_{L^2}^2
&=\|\sqrt h\p_xf\|_{L^2}^2
+\f{\e^2}{4}|b_1|^2\| h^{\f52}\p_x^3f\|_{L^2}^2
+\e b_1\bigl(h^3\p_xf\,\big|\,\p_x^3f\bigr)_{L^2},
\end{aligned}\eeno
and 
\beno
\bigl(h^3\p_xf\,\big|\,\p_x^3f\bigr)_{L^2}
=-\bigl(h^3\p_x^2f\,\big|\,\p_x^2f\bigr)_{L^2}
-3\bigl(h^2\p_xh\p_xf\,\big|\,\p_x^2f\bigr)_{L^2},
\eeno
which imply that
\beno\begin{aligned}
\|\sqrt h\p_xf+\f\e2b_1 h^{\f52}\p_x^3f\|_{L^2}^2
&=\|\sqrt h\p_xf\|_{L^2}^2
+\f{\e^2}{4}|b_1|^2\| h^{\f52}\p_x^3f\|_{L^2}^2
-\e b_1\| h^{\f32}\p_x^2f\|_{L^2}^2\\
&\quad -3\e b_1\bigl(h^2\p_xh\p_xf\,\big|\,\p_x^2f\bigr)_{L^2}.
\end{aligned}\eeno

Using the condition \eqref{condition for h 2} leads to
\beno
\e\bigl|\bigl(h^2\p_xh\p_xf\,\big|\,\p_x^2f\bigr)_{L^2}\bigr|
\lesssim\e\|\p_xh\|_{H^1}\|\p_xf\|_{L^2}\|\p_x^2f\|_{L^2}
\lesssim\e\|f_x\|_{L^2}\|\p_xf_x\|_{L^2}\lesssim\e^{\f12}\|f_x\|_{X^0_\e}^2.
\eeno
Since $b_1<0$ and $0<h_0\leq h\leq 2$, there holds 
\beno\begin{aligned}
&\|\sqrt h\p_xf\|_{L^2}^2
+\f{\e^2}{4}|b_1|^2\| h^{\f52}\p_x^3f\|_{L^2}^2
-\e b_1\| h^{\f32}\p_x^2f\|_{L^2}^2\\
\sim &\|f_x\|_{L^2}^2
+\e^2\| \p_x^2f_x\|_{L^2}^2+\e\|\p_xf_x\|_{L^2}^2
\sim \|f_x\|_{X^0_{\e^2}}^2,
\end{aligned}\eeno
where we used \eqref{interpolation}. Consequently, for sufficiently small $\e\in(0,1)$, we get
\beq\label{eq 5 for B 2}
\|\sqrt h\p_xf+\f\e2b_1 h^{\f52}\p_x^3f\|_{L^2}^2\sim \|f_x\|_{X^0_{\e^2}}^2.
\eeq

Thanks to \eqref{eq 4 for B 2} and \eqref{eq 5 for B 2}, we deduce from \eqref{eq 2 for B 2} that 
\beno
\|(\sqrt h\p_x +\f\e2 B_h(D))f\|_{L^2}^2\sim\|f_x\|_{X^0_{\e^2}}^2,
\eeno
provided that $\e\in(0,1)$ is sufficiently small. It verifies \eqref{B 2}.

\smallskip

{\bf (3). Proof of \eqref{C 1}.}  Due to the expression of $C_h(D)$ in \eqref{def of BC}, we first get 
\beq\label{eq 1 for C1}
\bigl(\p_x(\sqrt h\cdot)+\f\e2 C_h(D)\bigr)f
=\sqrt h\p_xf+\f\e2 b_1h^{\f52}\p_x^3f+ L_{h,2}(f),
\eeq
where
\beno\begin{aligned}
L_{h,2}(f)&\eqdefa \f{\p_xh}{2\sqrt h}f+\e b_1h\p_xh\p_x^2(\sqrt h f)+\f{b_1\e}{2}h^2[\p_x^3,\sqrt h]f\\
&\quad
-\f{(b_2+b_3)\e}{2}\p_x^2(h^{\f32}\p_xh f)-\f\e2\p_x\bigl(r_2(h)f\bigr).
\end{aligned}\eeno
Similar argument as \eqref{eq 2 for B 2} leads to
\beq\label{eq 2 for C1}
\begin{aligned}
\f12\|\sqrt h\p_xf+\f\e2b_1 h^{\f52}\p_x^3f\|_{L^2}^2-
\|L_{h,2}(f)\|_{L^2}^2&\leq \|(\sqrt h\p_x +\f\e2 C_h(D))f\|_{L^2}^2\\
&
\leq\f32\|\sqrt h\p_xf+\f\e2b_1 h^{\f52}\p_x^3f\|_{L^2}^2+3
\|L_{h,2}(f)\|_{L^2}^2.
\end{aligned}
\eeq

For the estimate of $\|L_{h,2}(f)\|_{L^2}$, using the condition $0<h_0\leq h\leq 2$ and the expression of $r_2(h)$, we have
\beno\begin{aligned}
\|L_{h,2}(f)\|_{L^2}&\lesssim\|\p_xh\|_{H^1}\|f\|_{L^2}+\e\|\p_xh\|_{H^1}\|\sqrt h f\|_{H^2}
+\e\|\p_x\sqrt h\|_{H^2}\|f\|_{H^2}\\
&\quad+\e\|h^{\f32}\p_xh\|_{H^2}\|f\|_{H^2}
+\e\bigl(\|\sqrt h(\p_xh)^2\|_{H^1}+\|h^{\f32}\p_x^2h\|_{H^1}\bigr)\|f\|_{H^1},
\end{aligned}\eeno
which along with \eqref{multiplier bounds} and the assumption $\|\p_xh\|_{H^3}\leq 1$ implies
\beno
\|L_{h,2}(f)\|_{L^2}\lesssim\|f\|_{L^2}+\e\|f\|_{H^2}\lesssim\|f\|_{L^2}+\e\|\p_xf_x\|_{L^2}.
\eeno
Then by virtue of \eqref{interpolation}, we get
\beq\label{eq 3 for C1}
\|L_{h,2}(f)\|_{L^2}^2\lesssim\|f\|_{L^2}^2+\e\|f_x\|_{X^0_{\e^2}}^2.
\eeq

Thanks to \eqref{eq 5 for B 2} and \eqref{eq 3 for C1}, we deduce \eqref{C 1} from \eqref{eq 2 for C1} supposed that $\e\in(0,1)$ is sufficiently small.

\smallskip

{\bf (4). Proof of \eqref{C 2}.} Using the expression of $C_h(D)$ in \eqref{def of BC}, we have
\beno\begin{aligned}
&\bigl((\p_x(\sqrt h\cdot)+\f\e2 C_h(D))f\,|\,\p_xf\bigr)_{L^2}\\ 
=&\bigl( \sqrt h\p_xf\,|\,\p_xf\bigr)_{L^2}-\f{b_1 \e}{2}\bigl(h^{\f52}\p_x^2f\,|\,\p_x^2f)\bigr)_{L^2}+\bigl(\p_x(\sqrt h)f\,|\,\p_xf \bigr)_{L^2}+\f\e2 I_2,
\end{aligned}\eeno
where 
\beno\begin{aligned}
&I_2\eqdefa-b_1\bigl( [\p_x^2,\sqrt h]f\,|\,h^2\p_x^2f\bigr)_{L^2}+(b_2+b_3)\bigl(\p_x(h^{\f32}\p_x hf)\,|\,\p_x^2f\bigr)_{L^2}-\bigl(\p_x(r_2(h)f)\,|\,\p_xf\bigr)_{L^2}.
\end{aligned}\eeno

Thanks to the expressions of $r_2(h)$ and \eqref{condition for h 2}, we get
\beno
\e|I_2|\lesssim\e\|f\|_{H^1}\|f_x\|_{H^1}\lesssim\|f\|_{L^2}^2+\e\|f_x\|_{L^2}^2+\e\|f_x\|_{L^2}\|\p_xf_x\|_{L^2},
\eeno
which along with the equivalence $\|g\|_{X^0_\e}^2\sim\|g\|_{L^2}^2+\e\|\p_xg\|_{L^2}^2$ implies
\beno
\e|I_2|\lesssim\|f\|_{L^2}^2+\e^{\f12}\|f_x\|_{X^0_\e}^2.
\eeno

While using \eqref{condition for h 2}, we get
\beno
\bigl|\bigl(\p_x(\sqrt h)f\,|\,\p_xf \bigr)_{L^2}\bigr|=\f12\bigl|\bigl(\p_x^2(\sqrt h)f\,|\,f \bigr)_{L^2}\bigr|
\lesssim\|f\|_{L^2}^2.
\eeno

The condition $0<h_0\leq h\leq 2$ and $b_1<0$ yield that
\beno
\bigl( \sqrt h\p_xf\,|\,\p_xf\bigr)_{L^2}-\f{b_1 \e}{2}\bigl(h^{\f52}\p_x^2f\,|\,\p_x^2f)\bigr)_{L^2}\sim\|f_x\|_{L^2}^2+\e\|\p_xf_x\|_{L^2}^2\sim\|f_x\|_{X^0_\e}^2.
\eeno

Consequently, for sufficiently small $\e\in(0,1)$, we could obtain \eqref{C 2}.

\smallskip

Based on the above proof, there exists $\e_1\in(0,1)$ such that for any $\e\in(0,\e_1)$, \eqref{B 1}, \eqref{B 2}, \eqref{C 1} and \eqref{C 2} all hold. The lemma is proved.
\end{proof}

\subsection{ {\it a priori} estimates on $(\p_t^kV,\p_t^k\eta)$ for $k=0,1,2$.} 
Since the coefficients of the linear parts of system \eqref{Boussinesq-like, d=1} depend on $h$ and $\p_xh=O(1)$, proving the long time existence result requires  to derive  the higher-order regularity in time variable $t$ rather than  the space variable $x$. The higher-order regularity in $x$ can then be recovered from the higher order regularity in $t$ through the equations in \eqref{Boussinesq-like, d=1}. 

Before presenting the key {\it a priori} energy estimates of $(\p_t^kV,\p_t^k\eta)$, we introduce the total energy functional associated with \eqref{Boussinesq-like, d=1}, along with the energy functional corresponding to $(\p_t^kV,\p_t^k\eta)$, as follows:
\beq\label{totoal functional}
\cE(t)\eqdefa\|V\|_{X^2_{\e^3}}^2
+\|V_t\|_{X^1_{\e^2}}^2+\|V_{tt}\|_{X^0_{\e}}^2+\|\eta\|_{X^2_{\e^2}}^2+\|\eta_t\|_{X^1_{\e}}^2+\|\eta_{tt}\|_{L^2}^2,
\eeq
\beq\label{functional for V eta t}
E(t)\eqdefa\sum_{k=0}^2 \left\{  \bigl((1-\f\e2 \cP_h^1 )\p_t^kV\,\bigr|\,\p_t^kV\bigr)_{L^2}+\bigl(\p_t^k\eta\,\bigr|\,\p_t^k\eta\bigr)_{L^2}\right\}.
\eeq

The main result of this subsection is stated in the following proposition.
   
   \begin{prop}\label{L2 space time energy estimate} Let $\e\in(0,1)$, $T_0>0$, $h-1\in H^5(\R)$ satisfy the condition \eqref{condition for h 2}. Then for any sufficiently smooth solution $(V,\eta)$ to \eqref{Boussinesq-like, d=1} over the time interval $[0,\f{T_0}{\e}]$, there holds
 \beq\label{space time energy estimate}
   		\f{d}{dt}\bigl(E(t)+\f\e2\wt{E}_2(t)\bigr)
   		\lesssim\e\cdot\cE(t)+\e\bigl(\cE(t)\bigr)^{\f32},  
\eeq
where the associated functional $\wt{E}_2(t)$ satisfies
\beq\label{estimate of wt E2}
|\wt E_2(t)|\lesssim \|V_t\|_{X_\e^1}^2.
\eeq
   \end{prop}
\begin{proof} We divide the proof into several steps.

{\bf Step 1. Energy functionals.} For $k=0,1,2$, applying $\p_t^k$ to both sides of equations in \eqref{Boussinesq-like, d=1} leads to
  \beq\label{Boussinesq-like k}
   \left\{ \begin{aligned}
    & (1-\f\e2 \cP_h^1 ) \p_t(\p_t^kV)+ \sqrt{h} \p_x(\p_t^k\eta) + \f\e2 B_h(D)\p_t^k\eta\\ 
    &\qquad\qquad+\frac{\e}{2\sqrt h}\p_t^k(\eta \p_x \eta)+\frac{3\e}{2\sqrt h}\p_t^k( V \p_xV)=\f{\e\p_xh}{4h^{\f32}}\p_t^k(|V|^2),\\
     &\p_t(\p_t^k\eta) + \p_x( \sqrt{h}\p_t^kV) + \f\e2 C_h(D)\p_t^kV+\f\e2\p_t^k\p_x(\f{\eta V}{\sqrt h})=0.
   \end{aligned}  \right.
    \eeq
 We define the energy functional associated with \eqref{Boussinesq-like k} by   
\beno
E_k(t)\eqdefa\bigl((1-\f\e2 \cP_h^1 )\p_t^kV\,\bigr|\,\p_t^kV\bigr)_{L^2}+\bigl(\p_t^k\eta\,\bigr|\,\p_t^k\eta\bigr)_{L^2}.
\eeno
Then using $\cP_h^1 = a_1 \p_x ( h^2 \p_x \quad )$ and the assumptions $a_1>0$ and $2\geq h(x)\geq h_0>0$, we get
\beq\label{equivalent 1}\begin{aligned}
E_k(t)&=\|\p_t^kV\|_{L^2}^2+\f\e2 a_1\|h\p_x\p_t^kV\|_{L^2}^2+\|\p_t^k\eta\|_{L^2}^2\\
&
\sim\|\p_t^kV\|_{L^2}^2+\e\|\p_x\p_t^kV\|_{L^2}^2+\|\p_t^k\eta\|_{L^2}^2.
\end{aligned}\eeq

{\bf Step 2. Estimates of $(\p_t^kV,\p_t^k\eta)$. }
Due to \eqref{Boussinesq-like k} and  $\cP_h^1 = a_1 \p_x ( h^2 \p_x \quad )$, we have
\beno\begin{aligned}
\f12\f{d}{dt} E_k(t)&=\bigl((1-\f\e2 \cP_h^1 ) \p_t(\p_t^kV)\,\big|\,\p_t^kV\bigr)_{L^2}+\bigl(\p_t(\p_t^k\eta)\,\big|\,\p_t^k\eta\bigr)_{L^2}\\
& 
=-\f\e2\cL_k-\f\e2\cS_k-\f{3\e}{2}\cT_k+\f\e4\cN_k,
\end{aligned}\eeno
where 
\beq\label{def of LSTN}\begin{aligned}
\cL_k&\eqdefa\bigl(B_h(D)\p_t^k\eta\,\big|\,\p_t^kV\bigr)_{L^2}+\bigl(C_h(D)\p_t^kV\,\big|\,\p_t^k\eta\bigr)_{L^2},\\
\cS_k&\eqdefa\bigl(\frac{1}{\sqrt h}\p_t^k(\eta \p_x \eta)\,\big|\,\p_t^kV\bigr)_{L^2}+\bigl(\p_t^k\p_x(\f{\eta V}{\sqrt h})\,\big|\,\p_t^k\eta\bigr)_{L^2},\\
\cT_k&\eqdefa\bigl(\frac{1}{\sqrt h}\p_t^k( V \p_xV)\,\big|\,\p_t^kV\bigr)_{L^2},\quad
\cN_k\eqdefa\bigl(\f{\p_xh}{h^{\f32}}\p_t^k(|V|^2)\,\big|\,\p_t^kV\bigr)_{L^2}.
\end{aligned}\eeq

Applying  \eqref{key 1} to $\cL_k$ implies that
\beq\label{eq 1 case 2}
\f12\f{d}{dt} E_k(t)=-\f\e2\bigl(r(h)\p_x\p_t^k\eta\,|\,\p_t^kV\bigr)_{L^2}-\f\e2\cS_k-\f{3\e}{2}\cT_k+\f\e4\cN_k,
\eeq
where $r(h)$ was defined in \eqref{def of Rh}.

{\bf Step 3. Estimates of $E_0(t)$ and $E_1(t)$.} For $k=0,1$, using the expression of $r(h)$ in \eqref{def of Rh} and the condition \eqref{condition for h 2}, we have
\beno\begin{aligned}
&\bigl|\bigl(r(h)\p_x\eta\,|\,V\bigr)_{L^2}\bigr|\leq\|r(h)\|_{L^\infty}\|\p_x\eta\|_{L^2}\|V\|_{L^2}
\lesssim\bigl(\|\p_xh\|_{L^\infty}^2+\|\p_x^2h\|_{L^\infty}\bigr)\|\p_x\eta\|_{L^2}\|V\|_{L^2},\\
&\bigl|\bigl(r(h)\p_x\eta_t\,|\,V_t\bigr)_{L^2}\bigr|\leq\|r(h)\|_{L^\infty}\|\p_x\eta_t\|_{L^2}\|V_t\|_{L^2}
\lesssim\bigl(\|\p_xh\|_{L^\infty}^2+\|\p_x^2h\|_{L^\infty}\bigr)\|\p_x\eta_t\|_{L^2}\|V_t\|_{L^2},
\end{aligned}\eeno
which along with \eqref{condition for h 2} and the inequality $\|f\|_{L^\infty(\R)}\lesssim\|f\|_{H^1(\R)}$ implies
\beno
\bigl|\bigl(r(h)\p_x\eta\,|\,V\bigr)_{L^2}\bigr|+\bigl|\bigl(r(h)\p_x\eta_t\,|\,V_t\bigr)_{L^2}\bigr|\lesssim\|\p_xh\|_{H^2}\bigl(\|\eta\|_{H^1}\|V\|_{L^2}+\|\eta_t\|_{H^1}\|V_t\|_{L^2}\bigr).
\eeno
Then using $\|\p_xh\|_{H^4}\leq 1$, we deduce from \eqref{eq 1 case 2} that
\beq\label{estimate of E01 a}
\f{d}{dt}\bigl(E_0(t)+E_1(t)\bigr)\lesssim\e\bigl(\|\eta\|_{H^1}\|V\|_{L^2}+\|\eta_t\|_{H^1}\|V_t\|_{L^2}\bigr)+\e\sum_{k=0}^1\bigl(|\cS_k|+|\cT_k|+|\cN_k|\bigr).
\eeq

The expressions of  $\cS_0$ and $\cS_1$ in \eqref{def of LSTN}, along with \eqref{condition for h 2} and the inequality $\|f\|_{L^\infty(\R)}\lesssim\|f\|_{H^1(\R)}$, lead to
\beno\begin{aligned}
|\cS_0|&\lesssim\|\frac{1}{\sqrt h}(\eta \p_x \eta)\|_{L^2}\|V\|_{L^2}
+\|\p_x(\f{\eta V}{\sqrt h})\|_{L^2}\cdot\|\eta\|_{L^2}\\
&\lesssim\|\eta\|_{H^1}^2\|V\|_{L^2}+\|V\|_{H^1}\|\eta\|_{H^1}\|\eta\|_{L^2}
\lesssim\|\eta\|_{H^1}^2\|V\|_{H^1},\\
|\cS_1|&\lesssim\|\frac{1}{\sqrt h}\p_t(\eta \p_x \eta)\|_{L^2}\|\p_tV\|_{L^2}
+\|\p_t\p_x(\f{\eta V}{\sqrt h})\|_{L^2}\cdot\|\p_t\eta\|_{L^2}\\
&\lesssim\|\eta\|_{H^1}\|\eta_t\|_{H^1}\|V_t\|_{L^2}+\bigl(\|V_t\|_{H^1}\|\eta\|_{H^1}+\|V\|_{H^1}\|\eta_t\|_{H^1}\bigr)\|\eta_t\|_{L^2}\\
&\lesssim\bigl(\|\eta\|_{H^1}+\|V\|_{H^1}\bigr)\cdot\bigl(\|\eta_t\|_{H^1}^2
+\|V_t\|_{H^1}^2\bigr),\\
\end{aligned}\eeno
which imply that
\beno
|\cS_0|+|\cS_1|\lesssim\bigl(\|\eta\|_{H^1}+\|V\|_{H^1}\bigr)\cdot\bigl(\|\eta\|_{H^1}^2+\|\eta_t\|_{H^1}^2+\|V_t\|_{H^1}^2\bigr).
\eeno
Similar derivation yields to
\beno
|\cT_0|+|\cT_1|+|\cN_0|+|\cN_1|\lesssim\|V\|_{H^1}\bigl(\|V\|_{H^1}^2
+\|V_t\|_{H^1}^2\bigr).
\eeno

Consequently, we derive from \eqref{estimate of E01 a} that 
\beno\begin{aligned}
\f{d}{dt}\bigl(E_0(t)+E_1(t)\bigr)&\lesssim\e\bigl(\|\eta\|_{H^1}\|V\|_{L^2}+\|\eta_t\|_{H^1}\|V_t\|_{L^2}\bigr)\\
&\qquad+\e\bigl(\|\eta\|_{H^1}^2+\|V\|_{H^1}^2+\|\eta_t\|_{H^1}^2+\|V_t\|_{H^1}^2\bigr)^{\f32},
\end{aligned}\eeno
which along with \eqref{totoal functional} implies
\beq\label{estimate of E01}
\f{d}{dt}\bigl(E_0(t)+E_1(t)\bigr)\lesssim\e\cdot\cE(t)+\e\bigl(\cE(t)\bigr)^{\f32}.
\eeq

{\bf Step 4. Estimate of  $E_2(t)$.} Since $(V_{tt},\eta_{tt})$ is expected to belong to space $X_\e^0\times L^2$,  the right-hand side term $\f\e2\bigl(r(h)\p_x\p_t^2\eta\,|\,\p_t^2V\bigr)_{L^2}$ in \eqref{eq 1 case 2} can not be directly controlled without losing  the factor $\sqrt\e$. To overcome this difficulty, we use the equations in \eqref{Boussinesq-like, d=1} to achieve the following {\bf claim}:
 \beq\label{estimate for BC}
\bigl(r(h)\p_x\p_t^2\eta\,|\,\p_t^2V\bigr)_{L^2} =  \frac{d}{dt}\wt{E}_2(t) + R_2,
\eeq
 and 
 \beq\label{bound of e2}
 |\wt E_2(t)|\lesssim \|V_t\|_{X_\e^1}^2,\quad |R_2|\lesssim\cE(t)
+\bigl(\cE(t)\bigr)^{\f32}.
 \eeq

 {\it Step 4.1. The proof of \eqref{estimate for BC} and \eqref{bound of e2}.}
Indeed, by using the evolution equation of $\eta$ in \eqref{Boussinesq-like, d=1}, we have
\beno
 \p_t^2 \eta=\p_t\eta_t = - \p_x ( \sqrt{h} \p_tV ) - \f\e2 C_h(D) \p_t V-\f\e2\p_t\p_x(\f{\eta V}{\sqrt h}),
\eeno
which yields to
\beq\label{BC 1}\begin{aligned}
\bigl(r(h)\p_x\p_t^2\eta\,|\,\p_t^2V\bigr)_{L^2}&=-\bigl(r(h)\p_x^2( \sqrt{h} \p_t V )\,|\,\p_t^2V\bigr)_{L^2}-\f\e2\bigl(r(h)\p_xC_h(D) \p_t V\,|\,\p_t^2V\bigr)_{L^2}\\
&\qquad- \f\e2\bigl(r(h)\p_t\p_x^2(\f{\eta V}{\sqrt h})\,|\,\p_t^2V\bigr)_{L^2}\\
&\eqdefa R_{21}+ R_{22}+ R_{23}.
\end{aligned}\eeq

{\it (i). For $R_{21}$}, direct calculations give rise to
\beno\begin{aligned}
R_{21}&=\f12\f{d}{dt}\bigl(r(h)\sqrt{h}\p_x\p_t V\,|\,\p_x\p_tV\bigr)_{L^2}+\bigl(\p_xr(h)\p_x\sqrt{h}\p_t V\,|\,\p_t^2V\bigr)_{L^2}\\
&\qquad+\bigl(\p_xr(h)\sqrt{h}\p_x\p_t V\,|\,\p_t^2V\bigr)_{L^2}-\bigl(\p_x\bigl(r(h)\p_x\sqrt{h}\p_t V\bigr)\,|\,\p_t^2V\bigr)_{L^2}\\
&\eqdefa \f12\f{d}{dt}\wt E_{21}(t)+\wt{R}_{21}.
\end{aligned}\eeno
where
\beno
\wt E_{21}(t)=\bigl(r(h)\sqrt{h}\p_xV_t\,|\,\p_xV_t\bigr)_{L^2}.
\eeno

Thanks to \eqref{def of Rh} and \eqref{condition for h 2}, we get
\beq\label{eq 2 case 2}
\wt E_{21}(t)|\lesssim\|\p_xh\|_{H^2}\|\p_xV_t\|_{L^2}^2\quad
\text{and}\quad
|\wt{R}_{21}|\lesssim\|\p_xh\|_{H^3}\|V_t\|_{H^1}\|V_{tt}\|_{L^2}.
\eeq

{\it (ii). For $R_{22}$}, the expression of $C_h(D)$ in \eqref{def of BC}, along with a derivation similar to that of $R_{21}$, suggests that
\beno\begin{aligned}
R_{22}
&=-\f\e2 b_1\bigl(h^2\p_x^2(\sqrt h \p_t V)\,|\,\p_x^2\bigl(r(h)\p_t^2V\bigr)\bigr)_{L^2}-\f\e2(b_2+b_3)\bigl(\p_x^2(h^{\f32}\p_xh \p_t V)\,|\,\p_x\bigl(r(h)\p_t^2V\bigr)\bigr)_{L^2}\\
&\qquad+\f\e2 \bigl(\p_x(r_2(h) \p_tV)\,|\,\p_x\bigl(r(h)\p_t^2V\bigr)\bigr)_{L^2}\\
&=\f12\f{d}{dt}\wt E_{22}(t)+\wt{R}_{22},
\end{aligned}\eeno
where 
\beno
\wt E_{22}(t)\eqdefa-\f\e2 b_1\bigl(r(h)h^{\f52} \p_x^2V_t)\,|\,\p_x^2V_t\bigr)_{L^2}
\eeno
and $\wt{R}_{22}$ contains the remainder terms of $R_{22}$. Similar derivation as \eqref{eq 2 case 2} leads to
\beq\label{eq 3 case 2}\begin{aligned}
&|\wt E_{22}(t)|\lesssim\|\p_xh\|_{H^2}\cdot\e\|\p_x^2V_t\|_{L^2}^2
\lesssim\|\p_xh\|_{H^2}\|V_t\|_{X^1_\e}^2,\\
\text{and}\quad&
|\wt{R}_{22}|\lesssim\|\p_xh\|_{H^3}\cdot\e\|V_t\|_{H^2}\|V_{tt}\|_{H^1}
\lesssim\|\p_xh\|_{H^3}\|V_t\|_{X^1_\e}\|V_{tt}\|_{X^0_\e}.
\end{aligned}\eeq

{\it (iii). For $R_{23}$}, using the condition \eqref{condition for h 2}, we have
\beno\begin{aligned}
|R_{23}|&\lesssim\e\|r(h)\|_{L^\infty}\bigl\|\p_x^2\p_t\bigl( \frac{\eta V}{\sqrt h} \bigr)\bigr\|_{L^2}\|\p_t^2V\|_{L^2}\\
&\lesssim \e\bigl(\|\eta_t\|_{H^2}\|V\|_{H^2}+\|\eta\|_{H^2}\|V_t\|_{H^2}\bigr)\|V_{tt}\|_{L^2}\\
&\lesssim \bigl(\|\eta_t\|_{X^1_\e}\|V\|_{H^2}+\|\eta\|_{H^2}\|V_t\|_{X^1_\e}\bigr)\|V_{tt}\|_{L^2}
\end{aligned}\eeno
which implies
\beq\label{eq 4 case 2}
|R_{23}|\lesssim\bigl(\|\eta\|_{H^2}+\|V\|_{H^2}\bigr)\cdot\bigl(\|\eta_t\|_{X^1_\e}^2+\|V\|_{H^2}^2+\|V_t\|_{X^1_\e}^2+\|V_{tt}\|_{L^2}^2\bigr).
\eeq

Consequently,  by defining
\beno\begin{aligned}
&\wt E_2(t)\eqdefa\wt E_{21}(t)+\wt E_{22}(t),\\
&R_2\eqdefa \wt{R}_{21}+\wt{R}_{22}+R_{23},
\end{aligned}\eeno 
 we derive \eqref{estimate for BC} from \eqref{BC 1}. Moreover, by virtue of \eqref{eq 2 case 2}, \eqref{eq 3 case 2} and \eqref{eq 4 case 2}, we have 
 \beno
 |\wt E_2(t)|\lesssim\|\p_x h\|_{H^2} \|V_t\|_{X_\e^1}^2\lesssim \|V_t\|_{X_\e^1}^2,
 \eeno
 \beno
|R_2|\lesssim\|V_t\|_{X^1_\e}\|V_{tt}\|_{X^0_\e}
+\bigl(\|\eta\|_{H^2}^2+\|\eta_t\|_{X^1_\e}^2+\|V\|_{H^2}^2+\|V_t\|_{X^1_\e}^2+\|V_{tt}\|_{L^2}^2\bigr)^{\f32},
 \eeno
 which along with \eqref{totoal functional} give rise to \eqref{bound of e2}.

\smallskip

{\it Step 4.2. Estimate of $E_2(t)$.} Thanks to \eqref{eq 1 case 2} and \eqref{estimate for BC}, we have
\beq\label{eq 5 case 2}
\f12\f{d}{dt}\bigl(E_2(t)+\f\e2\wt E_2(t)\bigr)=-\f\e2 R_2-\f\e2\cS_2-\f{3\e}{2}\cT_2+\f\e4\cN_2.
\eeq

(i). For $\cS_2$, the expression in \eqref{def of LSTN} leads to
\beno\begin{aligned}
\cS_2&=\bigl\{\bigl(\frac{\eta}{\sqrt h}\p_x \eta_{tt}\,\big|\,V_{tt}\bigr)_{L^2}-\bigl(\f{\eta }{\sqrt h}V_{tt}\,\big|\,\p_x\eta_{tt}\bigr)_{L^2}\bigr\}
-\bigl(\f{V }{\sqrt h}\eta_{tt}\,\big|\,\p_x\eta_{tt}\bigr)_{L^2}\\
&\qquad
+\bigl(\frac{1}{\sqrt h}[\p_t^2,\eta]\p_x \eta)\,\big|\,V_{tt}\bigr)_{L^2}+
2\bigl(\p_x(\f{\eta_tV_t}{\sqrt h})\,\big|\,\eta_{tt}\bigr)_{L^2}\\
&=\f12\bigl(\p_x(\f{V }{\sqrt h})\eta_{tt}\,\big|\,\eta_{tt}\bigr)_{L^2}+\bigl(\eta_{tt}\p_x \eta+2\eta_t\p_x \eta_t\,\big|\,\frac{V_{tt}}{\sqrt h}\bigr)_{L^2}+
2\bigl(\p_x(\f{\eta_tV_t}{\sqrt h})\,\big|\,\eta_{tt}\bigr)_{L^2},
\end{aligned}\eeno
which along with \eqref{condition for h 2} implies
\beq\label{estimate of S2}\begin{aligned}
|\cS_2|&\lesssim\|V\|_{H^2}\|\eta_{tt}\|_{L^2}^2+\|\eta\|_{H^2}\|\eta_{tt}\|_{L^2}\|V_{tt}\|_{L^2}+\|\eta_t\|_{H^1}^2\|V_{tt}\|_{L^2}+\|\eta_t\|_{H^1}\|V_t\|_{H^1}\|\eta_{tt}\|_{L^2}\\
&\lesssim\bigl(\|V\|_{H^2}^2+\|V_t\|_{H^1}^2+\|V_{tt}\|_{L^2}^2+\|\eta\|_{H^2}^2+\|\eta_t\|_{H^1}^2+\|\eta_{tt}\|_{L^2}^2\bigr)^{\f32}.
\end{aligned}\eeq

(ii). For $\cT_2$, the expression in \eqref{def of LSTN} implies
\beno\begin{aligned}
\cT_2&=\bigl(\frac{1}{\sqrt h}V \p_xV_{tt})\,\big|\,V_{tt}\bigr)_{L^2}
+\bigl(\frac{1}{\sqrt h}(V_{tt}\p_xV+2V_t\p_xV_t)\,\big|\,V_{tt}\bigr)_{L^2}\\
&=-\f12\bigl(\p_x(\frac{V}{\sqrt h})V_{tt}\,\big|\,V_{tt}\bigr)_{L^2}
+\bigl(\frac{1}{\sqrt h}(V_{tt}\p_xV+2V_t\p_xV_t)\,\big|\,V_{tt}\bigr)_{L^2}
\end{aligned}\eeno
which along with \eqref{condition for h 2} gives rise to
\beq\label{estimate of T2}
|\cT_2|\lesssim\|V\|_{H^2}\|V_{tt}\|_{L^2}^2+\|V_t\|_{H^1}^2\|V_{tt}\|_{L^2}
\lesssim\bigl(\|V\|_{H^2}^2+\|V_t\|_{H^1}^2+\|V_{tt}\|_{L^2}^2\bigr)^{\f32}.
\eeq

(iii). For $\cN_2$, the expression in \eqref{def of LSTN} and \eqref{condition for h 2} hints that
\beno
|\cN_2|\lesssim\|\f{\p_xh}{h^{\f32}}\p_t^2(|V|^2)\|_{L^2}\|\p_t^2V\|_{L^2}
\lesssim\|V\|_{H^1}\|V_{tt}\|_{L^2}^2+\|V_t\|_{H^1}^2\|V_{tt}\|_{L^2},
\eeno
which implies
\beq\label{estimate of N2}
|\cN_2|\lesssim\bigl(\|V\|_{H^2}^2+\|V_t\|_{H^1}^2+\|V_{tt}\|_{L^2}^2\bigr)^{\f32}.
\eeq

Combining \eqref{bound of e2}, \eqref{estimate of S2}, \eqref{estimate of T2}
and \eqref{estimate of N2}, using \eqref{totoal functional}, we deduce from \eqref{eq 5 case 2} that
\beq\label{estimate of E2}
\f12\f{d}{dt}\bigl(E_2(t)+\f\e2\wt E_2(t)\bigr)
\lesssim\e\cdot\cE(t)
+\e\bigl(\cE(t)\bigr)^{\f32}.
\eeq

{\bf Step 5. The {\it a priori} energy estimates.} Thanks to \eqref{estimate of E01} and \eqref{estimate of E2}, we get
\beno
\f12\f{d}{dt}\bigl(\sum_{k=0}^2E_k(t)+\f\e2\wt E_2(t)\bigr)
\lesssim\e\cdot\cE(t)+\e\bigl(\cE(t)\bigr)^{\f32},
\eeno
which along with \eqref{functional for V eta t} implies \eqref{space time energy estimate}. Thus, the proof of proposition is completed.
\end{proof}

\subsection{Proof of Theorem \ref{main results 2}.} 
\begin{proof}[Proof of Theorem \ref{main results 2}] In the proof, we focus on the determination of  the existence time interval of solutions and on the closure of the energy estimates. The argument relies on a continuity argument combined with the {\it a priori} energy estimate for system \eqref{Boussinesq-like, d=1}. 
Thanks to Proposition \ref{L2 space time energy estimate}, we first have
\beq\label{eq 1 for thm 2}
   		\f{d}{dt}\bigl(E(t)+\f\e2\wt{E}_2(t)\bigr)
   		\lesssim\e\cdot\cE(t)+\e\bigl(\cE(t)\bigr)^{\f32},  
\eeq
and 
\beq\label{eq 2 for thm 2}
|\wt E_2(t)|\lesssim \|V_t\|_{X_\e^1}^2.
\eeq
We shall use the {\it a priori} estimate \eqref{eq 1 for thm 2} to close the energy estimate. 

Throughout the proof, we always assume that $\e\in(0,\min\{\e_1,C_0^{-1}\})$ and $h-1\in H^5(\R)$ satisfies
\beq\label{ansatz for h}
0<h_0\leq h\leq 2,\quad\|\p_xh\|_{H^4}\leq 1,
\eeq
for some constant $h_0$. Here $\e_1>0$ is a constant given in Lemma \ref{lem for BC} and $C_0>1$ is the constant in \eqref{ansatz for energy}. The proof is divided  into several steps.

{\bf Step 1. Ansatz for the continuity argument.} Let $(V_0,\eta_0)$ be the initial data for $(V,\eta)$ and
\beno
\cE_0\eqdefa\|V_0\|_{X^2_{\e^3}}^2+ \|\eta_0\|_{X^2_{\e^2}}^2.
\eeno
Without loss of generality, we assume that $\cE_0=1$.
Our ansatz for the energy functional is given by
\beq\label{ansatz for energy}
\max_{0\leq t\leq T_0/\e}\cE(t)\leq 8C_0\cE_0=8C_0,
\eeq
where $T_0>0$ and $C_0>1$ are constants, independent of $\e$, to be determined later.

To close the continuity argument, we need to show that there exists $\e_0>0$ such that for any $\e\in(0,\e_0)$, \eqref{ansatz for energy} is improved to 
\beq\label{improvement 1}
\max_{0\leq t\leq T_0/\e}\cE(t)\leq 4C_0\cE_0=4C_0.
\eeq

{\bf Step 2. Equivalent energy functionals.} To close the {\it a priori} energy estimates for \eqref{Boussinesq-like, d=1}, we must verify the following equivalence
\beq\label{equiv functional}
\cE(t)\sim E(t).
\eeq

Recalling the functionals $\cE(t)$ and $E(t)$ were defined in \eqref{totoal functional} and \eqref{functional for V eta t} as follows:
\beno
\cE(t)=\|V\|_{X^2_{\e^3}}^2
+\|V_t\|_{X^1_{\e^2}}^2+\|V_{tt}\|_{X^0_{\e}}^2+\|\eta\|_{X^2_{\e^2}}^2+\|\eta_t\|_{X^1_{\e}}^2+\|\eta_{tt}\|_{L^2}^2,
\eeno
\beno
E(t)=\sum_{k=0}^2 \left\{  \bigl((1-\f\e2 \cP_h^1 )\p_t^kV\,\bigr|\,\p_t^kV\bigr)_{L^2}+\bigl(\p_t^k\eta\,\bigr|\,\p_t^k\eta\bigr)_{L^2}\right\}.
\eeno

By virtue of \eqref{equivalent 1}, we have
\beq\label{equiv for Et}
E(t)\sim\|V\|_{X^0_\e}^2+\|V_t\|_{X^0_\e}^2+\|V_{tt}\|_{X^0_\e}^2
+\|\eta\|_{L^2}^2+\|\eta_t\|_{L^2}^2+\|\eta_{tt}\|_{L^2}^2,
\eeq
which implies that 
\beno
\cE(t)\sim E(t)+\|V_x\|_{X^1_{\e^3}}^2+\|V_{tx}\|_{X^0_{\e^2}}^2
+\|\eta_x\|_{X^1_{\e^2}}^2+\|\eta_{tx}\|_{X^0_{\e}}^2.
\eeno
Thereby, to proof \eqref{equiv functional}, it suffices to show that
\beq\label{eq 3 for thm 2}
\|V_x\|_{X^1_{\e^3}}^2
+\|V_{tx}\|_{X^0_{\e^2}}^2+\|\eta_x\|_{X^1_{\e^2}}^2+\|\eta_{tx}\|_{X^0_\e}^2
\lesssim E(t).
\eeq
That is to say, we shall recover spatial regularity of the solutions from their temporal regularity through the equations in \eqref{Boussinesq-like, d=1}.

\smallskip

 \textbf{ Step 2.1. Recovery of spatial regularity for $\eta_t$ and $V_t$.} 

{\it i). Estimate of $\|\eta_{tx}\|_{X^0_\e}^2$.} Using \eqref{Boussinesq-like, d=1}, we have
\beno
 ( \sqrt h\p_x + \f\e2 B_h(D) ) \p_t \eta = -(1-\f\e2 \cP_h^1) \p_t^2 V - \frac{\e}{4\sqrt h} \p_x \p_t ( \eta^2 + 3 V^2  ) +\frac{\e \p_x h}{4 h^\f32} \p_t (V^2),
\eeno
which implies
\beq\label{eq 4 for thm 2}\begin{aligned}
&\quad\bigl( (\sqrt h\p_x + \f\e2 B_h(D) ) \eta_t\,\big|\,\p_x\eta_t\bigr)_{L^2}\\
&=-\bigl( (1-\f\e2 \cP_h^1) V_{tt}\,\big|\,\p_x\eta_t\bigr)_{L^2}
-\underbrace{\f\e4\bigl(\frac{1}{\sqrt h} \p_x \p_t ( \eta^2 + 3 V^2 )-\frac{\p_x h}{h^\f32} \p_t (V^2)\,\big|\,\p_x\eta_t\bigr)_{L^2}}_{A_1}
\end{aligned}\eeq

Since $\cP_h^1=a_1\p_x(h^2\p_x\quad)$, we have
\beno
\bigl( (1-\f\e2 \cP_h^1)V_{tt}\,\big|\,\p_x\eta_t\bigr)_{L^2}
=\bigl(V_{tt}\,\big|\,\p_x\eta_t\bigr)_{L^2}+\f{a_1\e}{2}\bigl(h^2\p_xV_{tt}\,\big|\,\p_x^2\eta_t\bigr)_{L^2}
\eeno
which along with condition $0<h_0\leq h\leq 2$ hints
\beq\label{eq 5 for thm 2}
\bigl|\bigl( (1-\f\e2 \cP_h^1)V_{tt}\,\big|\,\p_x\eta_t\bigr)_{L^2}\bigr|
\lesssim\|V_{tt}\|_{L^2}\|\eta_{tx}\|_{L^2}+\e^{\f12}\|\p_xV_{tt}\|_{L^2}\cdot\e^{\f12}\|\p_x\eta_{tx}\|_{L^2}\lesssim\|V_{tt}\|_{X^0_\e}\|\eta_{tx}\|_{X^0_\e}.
\eeq

For $A_1$, integration by parts leads to
\beno
A_1=-\f\e2\bigl(\eta\eta_t + 3 VV_t\,\big|\,\p_x(\frac{1}{\sqrt h}\eta_{tx})\bigr)_{L^2}-\f\e4\bigl(\frac{\p_x h}{h^\f32} \p_t (V^2)\,\big|\,\eta_{tx}\bigr)_{L^2},
\eeno
which along with  $\|f\|_{L^\infty}\lesssim\|f\|_{H^1}$ and the condition \eqref{ansatz for h} implies
\beq\label{eq 6 for thm 2}\begin{aligned}
|A_1|&\lesssim\e\bigl(\|\eta\|_{H^1}\|\eta_t\|_{L^2}+\|V\|_{H^1}\|V_t\|_{L^2}\bigr)\cdot\|\eta_{tx}\|_{H^1}\\
&
\lesssim\e^{\f12}\bigl(\|\eta\|_{H^1}\|\eta_t\|_{L^2}+\|V\|_{H^1}\|V_t\|_{L^2}\bigr)\cdot\|\eta_{tx}\|_{X^0_\e}
\end{aligned}\eeq

By virtue of \eqref{B 1} in Lemma \ref{lem for BC}, we have
\beq\label{eq 7 for thm 2}
\bigl( ( h\p_x + \f\e2 B_h(D) ) \eta_t\,\big|\,\p_x\eta_t\bigr)_{L^2}\sim 
\|\eta_{tx}\|_{X^0_\e}^2.
\eeq

Thanks to \eqref{eq 5 for thm 2}, \eqref{eq 6 for thm 2} and \eqref{eq 7 for thm 2}, we deduce from \eqref{eq 4 for thm 2} that
\beno
\|\eta_{tx}\|_{X^0_\e}^2\lesssim\bigl(\|V_{tt}\|_{X^0_\e}+\e^{\f12}\|\eta\|_{H^1}\|\eta_t\|_{L^2}+\e^{\f12}\|V\|_{H^1}\|V_t\|_{L^2}\bigr)\|\eta_{tx}\|_{X^0_\e},
\eeno
 which along with \eqref{ansatz for energy}  and $\e\in (0,C_0^{-1})$ implies 
\beq\label{recovery for eta t}
\|\eta_{tx}\|_{X^0_\e}^2\lesssim\|V_{tt}\|_{X^0_\e}^2+\|\eta_t\|_{L^2}^2+\|V_t\|_{L^2}^2\lesssim E(t).
\eeq

{\it ii). Estimate of $\|V_{tx}\|_{X^0_{\e^2}}^2$.} Due to \eqref{Boussinesq-like, d=1}, we have
\beno
\bigl(\p_x(\sqrt{h}\cdot)+ \f\e2 C_h(D)\bigr) V_t=-\p_t^2 \eta -\f\e2 \p_x\p_t \bigl( \frac{\eta V}{\sqrt h} \bigr)
\eeno
which along with condition \eqref{ansatz for h} and \eqref{ansatz for energy} gives rise to
\beno\begin{aligned}
\|\bigl( \p_x(\sqrt{h}\cdot )+ \f\e2 C_h(D)\bigr) V_t\|_{L^2}&
\lesssim\|\eta_{tt}\|_{L^2}
+\e\|\eta_t\|_{H^1}\|V\|_{H^1}+\e\|\eta\|_{H^1}\|V_t\|_{H^1}\\
&\lesssim\|\eta_{tt}\|_{L^2}
+\e\|V\|_{H^1}(\|\eta_t\|_{L^2}+\|\eta_{tx}\|_{L^2})+\e^{\f12}\|\eta\|_{H^1}\|V_t\|_{X^0_\e}\\
&\lesssim\|\eta_{tt}\|_{L^2}+\|\eta_t\|_{L^2}+\|\eta_{tx}\|_{L^2}+\|V_t\|_{X^0_\e}\lesssim E(t)^{\f12},
\end{aligned}\eeno
where we used \eqref{equiv for Et} and \eqref{recovery for eta t} in the last inequality. Consequently, by virtue of \eqref{C 1} and \eqref{equiv for Et}, we get
\beq\label{recovery for vt}
\|V_{tx}\|_{X^0_{\e^2}}^2\lesssim \|\bigl( \sqrt{h}\p_x+ \f\e2 C_h(D)\bigr) V_t\|_{L^2}^2+\|V_t\|_{L^2}^2\lesssim E(t).
\eeq

\smallskip

\textbf{ Step 2.2. Recovery of spatial regularity for $\eta$ and $V$.} We divide the the spatial regularity of $\eta$ and $V$ into two parts as follows:
\beq\label{equiv for eta x and V x} 
\|\eta_x\|_{X^1_{\e^2}}^2\sim\|\eta_x\|_{X^0_{\e}}^2+\|\eta_{xx}\|_{X^0_{\e^2}}^2,\quad \|V_x\|_{X^1_{\e^3}}^2\sim \|V_x\|_{X^0_\e}^2+\|V_{xx}\|_{X^0_{\e^3}}^2.
\eeq

{\it i). Estimate of $\|\eta_x\|_{X^0_{\e}}$.} Firstly, using the first equation of \eqref{Boussinesq-like, d=1}, we  have
\beno\begin{aligned}
(\sqrt h\p_x + \f\e2 B_h(D) )\eta &= -(1-\f\e2 \cP_h^1) \p_t V - \frac{\e}{4\sqrt h} \p_x ( \eta^2 + 3 V^2  ) +\frac{\e \p_x h}{4 h^\f32}V^2,\\
\end{aligned}\eeno
which gives rise to
\beq\label{eq 8 for thm 2}\begin{aligned}
&\quad\bigl( (\sqrt h\p_x + \f\e2 B_h(D) ) \eta\,\big|\,\p_x\eta\bigr)_{L^2}\\
&=-\bigl( (1-\f\e2 \cP_h^1) V_t\,\big|\,\p_x\eta\bigr)_{L^2}
-\f\e4\bigl(\frac{1}{\sqrt h} \p_x ( \eta^2 + 3 V^2 )-\frac{\p_x h}{h^\f32}(V^2)\,\big|\,\p_x\eta\bigr)_{L^2}.
\end{aligned}\eeq

By virtue of \eqref{B 1} in Lemma \ref{lem for BC}, we get
\beno\begin{aligned}
\|\eta_x\|_{X^0_\e}^2&\lesssim\bigl( (\sqrt h\p_x + \f\e2 B_h(D) ) \eta\,\big|\,\p_x\eta\bigr)_{L^2}\\
&\lesssim\bigl(\|(1-\f\e2 \cP_h^1) V_t\|_{L^2}+\e\Bigl\|\frac{1}{\sqrt h} \p_x ( \eta^2 + 3 V^2 )\Bigr\|_{L^2}+\e\Bigl\|\frac{\p_x h}{h^\f32}(V^2)\Bigr\|_{L^2}\bigr)\cdot\|\eta_x\|_{L^2},
\end{aligned}\eeno
which along with the expression $\cP_h^1=a_1\p_x(h^2\p_x\quad)$, ansatz \eqref{ansatz for h}, \eqref{ansatz for energy}  and $\e\in (0,C_0^{-1})$ shows that
\beno\begin{aligned}
\|\eta_x\|_{X^0_\e}&\lesssim\|V_t\|_{L^2}+\e\|V_{tx}\|_{H^1}+\e\|\eta\|_{H^1}\|\eta_x\|_{L^2}
+\e\|V\|_{H^1}(\|V\|_{L^2}+\|V_x\|_{L^2})\\
&\lesssim\|V_t\|_{L^2}+\e^{\f12}\|V_{tx}\|_{X^0_{\e^2}}+\e\|\eta_x\|_{X^0_\e}+\e^{\f12}\|V\|_{X^0_\e}.
\end{aligned}\eeno
Then for sufficiently small  $\e\in(0, \min\{ \e_1,C_0^{-1} \} )$, there holds
\beno
\|\eta_x\|_{X^0_\e}\lesssim\|V_t\|_{L^2}+\|V_{tx}\|_{X^0_{\e^2}}+\|V\|_{X^0_\e}
\eeno
which together with \eqref{equiv for Et} and \eqref{recovery for vt}  yields to
\beq\label{recovery for eta x}
\|\eta_x\|_{X^0_\e}^2\lesssim E(t).
\eeq

{\it ii). Estimate of $\|V_x\|_{X^0_\e}$.}  Due to the second equation \eqref{Boussinesq-like, d=1}, we have
\beno
\p_x(\sqrt{h}V)+\f\e2 C_h(D) V= -\eta_t -\f\e2 \p_x \bigl( \frac{\eta V}{\sqrt h} \bigr),
\eeno
which along with \eqref{ansatz for h} and \eqref{ansatz for energy} shows that
\beno\begin{aligned}
\bigl(\p_x(&\sqrt{h}V)+\f\e2 C_h(D) V\,|\,V_x\bigr)_{L^2}
=\bigl(\p_x\eta_t\,|\,V\bigr)_{L^2}-\f\e2\bigl(\p_x \bigl( \frac{\eta V}{\sqrt h} \bigr)\,|\,V_x\bigr)_{L^2}\\
&\lesssim\|\eta_{tx}\|_{L^2}\|V\|_{L^2}+\e\|\eta\|_{H^1}\|V\|_{H^1}\|V_x\|_{L^2}\\ 
&\lesssim\|\eta_{tx}\|_{L^2}\|V\|_{L^2}  +\e\|V\|_{L^2}^2+\e\|V_x\|_{L^2}^2.
\end{aligned}\eeno
Then by virtue of \eqref{C 2}, we get 
\beno
\|V_x\|_{X^0_\e}^2\lesssim\bigl(\p_x(\sqrt{h}V)+\f\e2 C_h(D) V\,|\,V_x\bigr)_{L^2}+\|V\|_{L^2}^2\lesssim \|V\|_{L^2}^2+\|\eta_{tx}\|_{L^2}\|V\|_{L^2}   +\e\|V_x\|_{L^2}^2,
\eeno
which along with \eqref{equiv for Et} and \eqref{recovery for eta t} implies
\beno 
\|V_x\|_{X^0_\e}^2\lesssim E(t)+\e\|V_x\|_{X^0_\e}^2.
\eeno
Thus,  for sufficiently small $\e\in(0,\min\{\e_1,C_0^{-1}\})$, there holds
\beq\label{recovery for V x}
\|V_x\|_{X^0_\e}^2\lesssim E(t).
\eeq

{\it iii). Estimate of $\|\eta_{xx}\|_{X^0_{\e^2}}$.}
Using the first equation of \eqref{Boussinesq-like, d=1}, we have
\beno
(\sqrt h\p_x + \f\e2 B_h(D) )\eta_x =-\f{\p_xh}{2\sqrt h}\p_x\eta-\f\e2[\p_x, B_h(D)]\eta + A_2
\eeno
where 
\beno
A_2\eqdefa-\p_x\bigl\{(1-\f\e2 \cP_h^1) \p_t V +\frac{\e}{4\sqrt h} \p_x ( \eta^2 + 3 V^2  ) -\frac{\e \p_x h}{4 h^\f32}V^2\bigr\}.
\eeno
By virtue of \eqref{ansatz for h}, we get
\beq\label{eq 9 for thm 2}\begin{aligned}
\|( \sqrt h\p_x + \f\e2 B_h(D) )\eta_x\|_{L^2}
&\lesssim\|\eta_x\|_{L^2}+\e\|[\p_x, B_h(D)]\eta\|_{L^2}
+\|A_2\|_{L^2}
\end{aligned}\eeq

Since $\cP_h^1=a_1\p_x(h^2\p_x\quad)$, using \eqref{ansatz for h}, we have
\beno\begin{aligned}
&\|\p_x\bigl((1-\f\e2 \cP_h^1) \p_t V\bigr)\|_{L^2}
\lesssim\|V_{tx}\|_{L^2}+\e\|\p_x^2V_{tx}\|_{L^2}+\e\|\p_xh\|_{H^1}\|\p_xV_{tx}\|_{L^2}\lesssim\|V_{tx}\|_{X^0_{\e^2}},\\
&\e\|\p_x\bigl(\frac{1}{\sqrt h} \p_x ( \eta^2 + 3 V^2  )\bigr)\|_{L^2}
+\e\|\p_x\bigl(\frac{\e \p_x h}{4 h^\f32}V^2\bigr)\|_{L^2}\\
&\quad
\lesssim\e\|\eta\|_{H^1}\|\eta_x\|_{H^1}
+\e\|V\|_{H^1}\|V_x\|_{H^1}+\e\|V\|_{H^1}\|V\|_{L^2}\\ 
&\quad\lesssim\e^{\f12}\|\eta\|_{H^1}\|\eta_x\|_{X^0_\e}+\e\|V\|_{H^1}\bigl(\|V_x\|_{L^2}+\|V_{xx}\|_{L^2}\bigr)+\e\|V\|_{H^1}\|V\|_{L^2}
\end{aligned}\eeno
which along with \eqref{ansatz for energy}  and $\e\in (0,C_0^{-1})$ implies
\beno
\|A_2\|_{L^2}\lesssim\|V_{tx}\|_{X^0_{\e^2}}+\|\eta_x\|_{X^0_\e}
+\|V_x\|_{L^2}+\|V\|_{L^2}+\e\|V_{xx}\|_{L^2}.
\eeno
Thanks to \eqref{equiv for Et}, \eqref{recovery for vt}, \eqref{recovery for eta x} and \eqref{recovery for V x}, we get
\beq\label{eq 10 for thm 2}
\|A_2\|_{L^2}\lesssim E(t)^{\f12}.
\eeq

For term $\e\|[\p_x, B_h(D)]\eta\|_{L^2}$, using the expression of $B_h(D)$ in \eqref{def of BC} and ansatz \eqref{ansatz for h}, we obtain
\beno\begin{aligned}
\e\|[\p_x, B_h(D)]\eta\|_{L^2}&\lesssim\e\|\p_x(h^{\f52})\p_x^3\eta\|_{L^2}
+\e\|\p_x(h^{\f32}\p_xh)\p_x^2\eta\|_{L^2}\\
&\quad+\e\|\p_x\bigl(h^{\f12}(\p_xh)^2+h^{\f32}\p_x^2h\bigr)\p_x\eta\|_{L^2}\\ 
&
\lesssim\e\|\eta_x\|_{H^2}\lesssim\e\|\eta_x\|_{L^2}+\e^{\f12}\cdot\e^{\f12}\|\p_x\eta_{xx}\|_{L^2},
\end{aligned}\eeno
which along with  \eqref{interpolation} and \eqref{recovery for eta x} implies
\beq\label{eq 11 for thm 2}
\e\|[\p_x, B_h(D)]\eta\|_{L^2}\lesssim\e\|\eta_x\|_{L^2}+\e^{\f12}\|\eta_{xx}\|_{X^0_{\e^2}}\lesssim E(t)^{\f12}+\e^{\f12}\|\eta_{xx}\|_{X^0_{\e^2}}.
\eeq

By virtue of \eqref{recovery for eta x}, \eqref{eq 10 for thm 2} and \eqref{eq 11 for thm 2}, we deduce from \eqref{eq 9 for thm 2} that
\beq\label{eq 12 for thm 2}
\|( \sqrt h\p_x + \f\e2 B_h(D) )\eta_x\|_{L^2}^2
\lesssim E(t)+\e\|\eta_{xx}\|_{X^0_{\e^2}}^2.
\eeq

Due to \eqref{B 2}, for any sufficiently small $\e\in(0,\e_1)$, there holds
\beno
\|( \sqrt h\p_x + \f\e2 B_h(D) )\eta_x\|_{L^2}^2\sim \|\eta_{xx}\|_{X^0_{\e^2}}^2
\eeno
from which and \eqref{eq 12 for thm 2}, we deduce
\beq\label{recovery for eta xx}
\|\eta_{xx}\|_{X^0_{\e^2}}^2\lesssim E(t),
\eeq
 for any sufficiently small $\e\in(0,\min\{\e_1,C_0^{-1}\})$.

{\it iv). Estimate of $\|V_{xx}\|_{X^0_{\e^3}}$. } By virtue of the second equation of \eqref{Boussinesq-like, d=1}, we have
\beno\begin{aligned}
&\p_x( \sqrt{h} V_x ) + \f\e2( C_h(D) V_x)=-\p_x\bigl(\f{\p_xh}{2\sqrt h}V\bigr)-\f\e2[\p_x,C_h(D)]V-\p_x \eta_t-\f\e2 \p_x^2\bigl( \frac{\eta V}{\sqrt h} \bigr),\\ 
&\p_x( \sqrt{h} V_{xx} ) + \f\e2( C_h(D) V_{xx})=-\p_x\bigl([\p_x^2,\sqrt h] V\bigr)-\f\e2[\p_x^2,C_h(D)]V-\p_x^2 \eta_t-\f\e2 \p_x^3 \bigl( \frac{\eta V}{\sqrt h} \bigr),
\end{aligned}\eeno
which along with tame estimates, \eqref{ansatz for energy} and $\e\in (0,C_0^{-1})$  imply
\beno\begin{aligned}
\bigl\|\bigl(\p_x( \sqrt{h} \cdot) + \f\e2 C_h(D)\bigr) V_x\bigr\|_{L^2}&\lesssim\|V\|_{H^1}
+\|\eta_{tx}\|_{L^2}+\e\|[\p_x,C_h(D)]V\|_{L^2}+ \e \|\eta\|_{H^2} \|V\|_{H^2}\\
& \lesssim \|V\|_{H^1}
+\|\eta_{tx}\|_{L^2}+\e\|[\p_x,C_h(D)]V\|_{L^2}+ \e^\f12 \|V\|_{H^2}, \\
\e^{\f12 }\bigl\|\bigl(\p_x( \sqrt{h} \cdot) + \f\e2 C_h(D)\bigr) V_{xx}\bigr\|_{L^2}
&\lesssim\e^{\f12 }\|V\|_{H^2}
+\e^{\f12 }\|\p_x \eta_{tx}\|_{L^2}+\e^{\f32 }\|[\p_x^2,C_h(D)]V\|_{L^2} \\
&\quad\quad\quad \quad\quad\quad + \e^\f32 \|\eta\|_{H^3} \|V\|_{H^2}+ \e^\f32 \|\eta\|_{H^2} \|V\|_{H^3} \\
&\lesssim  \e^{\f12 }\|V\|_{H^2}
+\e^{\f12 }\|\p_x \eta_{tx}\|_{L^2}+\e^{\f32 }\|[\p_x^2,C_h(D)]V\|_{L^2} + \e \|V\|_{H^3} .
\end{aligned}\eeno

The expression of $C_h(D)$ in \eqref{def of BC} and similar argument as \eqref{eq 11 for thm 2} lead to
\beno\begin{aligned}
\e\|[\p_x,C_h(D)]V\|_{L^2}&\lesssim\e\|\p_xh\|_{H^3}\|V\|_{H^3}\lesssim\e\|V\|_{H^3}, \\
\e^{\f32}\|[\p_x^2,C_h(D)]V\|_{L^2}&\lesssim\e^{\f32}\|\p_xh\|_{H^4}\|V\|_{H^4}
\lesssim \e^{\f32}\|V\|_{H^4}.
\end{aligned}\eeno
Using the interpolation inequality \eqref{interpolation}, we have
\beno\begin{aligned}
\e\|V\|_{H^3}&\lesssim\e\|V\|_{L^2}+\e^{\f12}\cdot\e^{\f12}\|\p_xV_{xx}\|_{L^2}
\lesssim \|V\|_{L^2}+\e^{\f12}\|V_{xx}\|_{X^0_{\e^3}},\\
\e^{\f32}\|V\|_{H^4}&\lesssim \e^{\f32}\|V\|_{L^2}+\e^{\f12}\cdot\e\|\p_x^2V_{xx}\|_{L^2}
\lesssim \|V\|_{L^2}+\e^{\f12}\|V_{xx}\|_{X^0_{\e^3}},
\end{aligned}\eeno
which shows that
\beno 
\e\|[\p_x,C_h(D)]V\|_{L^2}+\e^{\f32}\|[\p_x^2,C_h(D)]V\|_{L^2}
\lesssim \|V\|_{L^2}+\e^{\f12}\|V_{xx}\|_{X^0_{\e^3}}.
\eeno 

Consequently, we get 
\beq\label{eq 13 for thm 2}\begin{aligned}
&\bigl\|\bigl(\p_x( \sqrt{h} \cdot) + \f\e2 C_h(D)\bigr) V_x\bigr\|_{L^2}
+\e^{\f12 }\bigl\|\bigl(\p_x( \sqrt{h} \cdot) + \f\e2 C_h(D)\bigr) V_{xx}\bigr\|_{L^2}\\ 
&\quad 
\lesssim \|V\|_{L^2}+\|V_x\|_{X^0_\e}+\|\eta_{tx}\|_{X^0_\e}+\e^{\f12}\|V_{xx}\|_{X^0_{\e^3}} .
\end{aligned}\eeq

On the other hand, by virtue of \eqref{C 1}, we have
\beno\begin{aligned}
\|V_{xx}\|_{X^0_{\e^2}}^2+\e\|V_{xx}\|_{X^0_{\e^2}}^2
&\lesssim\bigl\|\bigl(\p_x( \sqrt{h} \cdot) + \f\e2 C_h(D)\bigr) V_x\bigr\|_{L^2}^2+\e\bigl\|\bigl(\p_x( \sqrt{h} \cdot) + \f\e2 C_h(D)\bigr) V_{xx}\bigr\|_{L^2}^2\\ 
&\quad 
+\|V_x\|_{L^2}^2+\e\|\p_xV_x\|_{L^2}^2
\end{aligned}\eeno
from which and \eqref{eq 13 for thm 2}, we deduce that
\beno 
\|V_{xx}\|_{X^0_{\e^3}}^2\lesssim\|V_{xx}\|_{X^0_{\e^2}}^2+\e\|V_{xx}\|_{X^0_{\e^2}}^2
\lesssim\|V\|_{L^2}^2+\|V_x\|_{X^0_\e}^2+\|\eta_{tx}\|_{X^0_\e}^2+\e\|V_{xx}\|_{X^0_{\e^3}}^2.
\eeno

Thanks to \eqref{equiv for Et}, \eqref{recovery for V x} and \eqref{recovery for eta t}, we get
\beno
\|V_{xx}\|_{X^0_{\e^3}}^2
\lesssim E(t)+\e\|V_{xx}\|_{X^0_{\e^3}}^2,
\eeno
which hints
\beq\label{recovery for V xx}
\|V_{xx}\|_{X^0_{\e^3}}^2\lesssim E(t),
\eeq
for any sufficiently small $\e\in(0,\e_1)$.

\smallskip

Combining \eqref{recovery for eta x},  \eqref{recovery for V x}, \eqref{recovery for eta xx} and \eqref{recovery for V xx}, we deduce from \eqref{equiv for eta x and V x}  that
\beq\label{recovery the spatial derivative} 
\|\eta_x\|_{X^1_{\e^2}}^2+\|V_x\|_{X^1_{\e^3}}^2\lesssim E(t).
\eeq

\smallskip

\textbf{Step 2.3. Equivalence of the energy functionals.} Thanks to \eqref{recovery for eta t}, \eqref{recovery for vt} and \eqref{recovery the spatial derivative}, we obtain \eqref{eq 3 for thm 2}. Then there holds \eqref{equiv functional}. By virtue of \eqref{equiv functional}, \eqref{eq 2 for thm 2} and \eqref{ansatz for h},  for any sufficiently small $\e\in(0,\min\{\e_1,C_0^{-1}\})$, there holds
\beq\label{equiv functional 2}
E(t)+\f\e2\wt{E}_2(t)\sim \cE(t)\sim E(t),\quad\forall\, t\in[0,T_0/\e],
\eeq
 in which all constants are independent of $T_0,C_0,\e$.
\smallskip 

{\bf Step 3. Initial energy functional.} To obtain the final uniform energy estimate from \eqref{eq 1 for thm 2}, it remains to construct the initial data for $\eta_t,\,V_t,\,\eta_{tt}$ and $V_{tt}$ through system \eqref{Boussinesq-like, d=1} and the initial conditions $\eta|_{t=0}=\eta_0$ and $V|_{t=0}=V_0$.  

{\it Step 3.1. Construction for $V_t|_{t=0}\in X^1_{\e^2}$.} The first equation of \eqref{Boussinesq-like, d=1} shows that
\beno 
(1-\f\e2 \cP_h^1  )V_t|_{t=0} =- \sqrt{h} \p_x \eta_0 -\f\e2 B_h(D) \eta_0 - \frac{\e}{2\sqrt h} \eta_0 \p_x \eta_0 -\frac{3\e }{2\sqrt h} V_0 \p_x V_0 +\frac{\e \p_x h }{4 h^{\f32}} |V_0|^2,
\eeno
which along with the expression of $B_h(D)$, \eqref{ansatz for h} and the assumption $\cE_0=\|V_0\|_{X^2_{\e^3}}^2+\|\eta_0\|_{X^2_{\e^2}}^2=1$ implies
\beq\label{eq 14 for thm 2}
\|(1-\f\e2 \cP_h^1  )V_t|_{t=0}\|_{H^1}
\lesssim\|\p_x\eta_0\|_{H^1}+\e\|\p_x\eta_0\|_{H^3}+\e\|\eta_0\|_{H^2}^2+\e\|V_0\|_{H^2}^2\lesssim\cE_0^{\f12},
\eeq
where we used the interpolation inequality \eqref{interpolation} in the last inequality.

Since $\cP_h^1=a_1\p_x(h^2\p_x\quad)$ and $a_1>0$, we have
\beq\label{eq 15 for thm 2}
\|(1-\f\e2 \cP_h^1  )f\|_{L^2}^2=\|f\|_{L^2}^2+\f{\e^2 a_1^2}{4}\|\p_x(h^2\p_xf)\|_{L^2}^2+\e a_1\|h\p_xf\|_{L^2}^2,
\eeq
and
\beno\begin{aligned}
\|\p_x(h^2\p_xf)\|_{L^2}^2
&=\|h^2\p_x^2f\|_{L^2}^2+4\|h\p_xh\p_xf\|_{L^2}^2+4\bigl(h^3\p_xh\p_x^2f\,\big|\,\p_xf\bigr)_{L^2}\\ 
&=\|h^2\p_x^2f\|_{L^2}^2+4\|h\p_xh\p_xf\|_{L^2}^2-2\bigl(\p_x(h^3\p_xh)\p_xf\,\big|\,\p_xf\bigr)_{L^2}.
\end{aligned}\eeno
Due to \eqref{ansatz for h}, there holds
\beno 
|\bigl(\p_x(h^3\p_xh)\p_xf\,\big|\,\p_xf\bigr)|\lesssim\|\p_x(h^3\p_xh)\|_{H^1}\|\p_xf\|_{L^2}^2\lesssim \|\p_xf\|_{L^2}^2
\eeno 
which  gives rise to
\beno 
\e^2\|\p_x^2f\|_{L^2}^2-\e^2\|\p_xf\|_{L^2}^2\lesssim\e^2\|\p_x(h^2\p_xf)\|_{L^2}^2
\lesssim\e^2\|\p_x^2f\|_{L^2}^2+\e^2\|\p_xf\|_{L^2}^2.
\eeno

Using the condition $0<h_0\leq h\leq 2$, we deduce from \eqref{eq 15 for thm 2} that
\beq\label{eq 16 for thm 2} 
\|(1-\f\e2 \cP_h^1  )f\|_{L^2}^2\sim\|f\|_{L^2}^2+\e^2\|\p_x^2f\|_{L^2}^2=\|f\|_{X^0_{\e^2}}^2,
\eeq
provided that $\e\in(0,\e_1)$ is sufficiently small. Then we have
\beq\label{eq 17 for thm 2}\begin{aligned}
&\|\p_x[(1-\f\e2 \cP_h^1)f]\|_{L^2}^2=\|(1-\f\e2 \cP_h^1)\p_xf-\f{a_1\e}{2}\p_x(\p_xh\p_xf)\|_{L^2}^2\\ 
&=\|(1-\f\e2 \cP_h^1)f_x\|_{L^2}^2+\f{a_1^2\e^2}{4}\|\p_x(\p_xh\p_xf)\|_{L^2}^2
-a_1\e\bigl((1-\f\e2 \cP_h^1)\p_xf\,\big|\,\p_x(\p_xh\p_xf)\bigr)_{L^2}.
\end{aligned}\eeq
Since
\beno\begin{aligned}
&\bigl((1-\f\e2 \cP_h^1)\p_xf\,\big|\,\p_x(\p_xh\p_xf)\bigr)_{L^2} 
=-\bigl(\p_x^2f\,\big|\,\p_xh\p_xf\bigr)_{L^2}
+\f{a_1\e}{2}\bigl(h^2\p_x^2f\,\big|\,\p_x^2(\p_xh\p_xf)\bigr)_{L^2}\\ 
=&\f{1}{2}\bigl(\p_xf\,\big|\,\p_x^2h\p_xf\bigr)_{L^2} + \frac{a_1 \e}{2} \bigl( h^2 \p_x^2 f \,\big|\, \p_x^3 h \p_x f \bigr)_{L^2}
+\f{a_1\e}{4}\bigl((3h^2\p_x^2h-h(\p_xh)^2)\p_x^2f\,\big|\,\p_x^2f\bigr)_{L^2},
\end{aligned}\eeno
using condition \eqref{ansatz for h}, we have
\beq\label{eq 18 for thm 2}
\e\bigl|\bigl((1-\f\e2 \cP_h^1)\p_xf\,\big|\,\p_x(\p_xh\p_xf)\bigr)_{L^2}\bigr|
\lesssim\e\|f_x\|_{L^2}^2+\e^2\|\p_xf_x\|_{L^2}^2
\lesssim\e\|f_x\|_{X^0_{\e^2}}^2,
\eeq
where we used \eqref{interpolation} in the last inequality. While using \eqref{ansatz for h} and \eqref{interpolation} yields to 
\beq\label{eq 19 for thm 2} 
\e^2\|\p_x(\p_xh\p_xf)\|_{L^2}^2\lesssim\e^2\|f_x\|_{H^1}^2
\lesssim\e\|f_x\|_{X^0_{\e^2}}^2.
\eeq

Due to \eqref{eq 16 for thm 2}, there holds
\beno 
\|(1-\f\e2 \cP_h^1)f_x\|_{L^2}^2\sim\|f_x\|_{X^0_{\e^2}}^2.
\eeno
from which, \eqref{eq 17 for thm 2}, \eqref{eq 18 for thm 2} and \eqref{eq 19 for thm 2}, we deduce that
\beq\label{eq 20 for thm 2}
\|\p_x[(1-\f\e2 \cP_h^1)f]\|_{L^2}^2\sim\|f_x\|_{X^0_{\e^2}}^2,
\eeq
provided that $\e\in(0,\e_1)$ is sufficiently small. 

Combining \eqref{eq 16 for thm 2} and \eqref{eq 20 for thm 2}, we obtain
\beq\label{eq 21 for thm 2}
\|(1-\f\e2 \cP_h^1)f\|_{H^1}^2\sim\|f\|_{X^0_{\e^2}}^2+\|f_x\|_{X^0_{\e^2}}^2
\sim\|f\|_{X^1_{\e^2}}^2.
\eeq

Thanks to \eqref{eq 14 for thm 2} and \eqref{eq 21 for thm 2}, we obtain
\beq\label{initial for Vt}
\|V_t|_{t=0}\|_{X^1_{\e^2}}^2\lesssim\cE_0.
\eeq

{\it Step 3.2. Construction for $\eta_t|_{t=0}\in X^1_{\e}$.} The second equation of \eqref{Boussinesq-like, d=1} hints
\beno 
\eta_t|_{t=0}=-\p_x( \sqrt{h} V_0 ) -\f\e2 C_h(D) V_0-\f\e2 \p_x \bigl( \frac{\eta_0 V_0}{\sqrt h} \bigr),
\eeno
which along with the expression of $C_h(D)$, \eqref{ansatz for h} and interpolation inequality \eqref{interpolation} implies
\beno\begin{aligned}
\|\eta_t|_{t=0}\|_{X^1_{\e}}&\lesssim\|\eta_t|_{t=0}\|_{H^1}+\e^{\f12}\|\p_x^2\eta_t|_{t=0}\|_{L^2}\\
&\lesssim\|V_0\|_{X^2_{\e^3}}+\e\|\eta_0\|_{X^2_\e}\|V_0\|_{X^2_\e}^2\lesssim
\|V_0\|_{X^2_{\e^3}}.
\end{aligned}\eeno
Then we obtain 
\beq\label{initial for etat}
\|\eta_t|_{t=0}\|_{X^1_{\e}}^2\lesssim\cE_0.
\eeq

{\it Step 3.3. Construction for $V_{tt}|_{t=0}\in X^0_{\e}$.} Using the first equation of \eqref{Boussinesq-like, d=1} leads to
\beno\begin{aligned}
(1-\f\e2 \cP_h^1  )V_{tt}|_{t=0}& =\Bigl(- \sqrt{h} \p_x \eta_t -\f\e2 B_h(D) \eta_t - \frac{\e}{2\sqrt h}\p_x(\eta\eta_t) \\
&\quad-\frac{3\e }{2\sqrt h}\p_x(VV_t) +\frac{\e \p_x h }{2 h^{\f32}}VV_t\Bigr)\Big|_{t=0},
\end{aligned}\eeno
which along with  the condition \eqref{ansatz for h} yields to
\beq\label{eq 22 for thm 2}\begin{aligned}
\bigl((1-\f\e2 \cP_h^1  )V_{tt}\bigl|_{t=0}\,\big|\,V_{tt}\bigl|_{t=0}\bigr)_{L^2}
&\lesssim\Bigl(\|\eta_t|_{t=0}\|_{H^1}+\e\|\eta_0\|_{H^1}\|\eta_t|_{t=0}\|_{H^1}+\e\|V_0\|_{H^1}\|V_t|_{t=0}\|_{H^1}\Bigr)_{}\\ 
&\qquad\times\|V_{tt}|_{t=0}\|_{L^2}+\e\bigl|\bigl(B_h(D)\eta_t\,\big|\,V_{tt}\bigr)_{L^2}\big|_{t=0}\bigr|.
\end{aligned}\eeq

Since $\cP_h^1=a_1\p_x(h^2\p_x\quad)$, $a_1>0$ and $0<h_0\leq h\leq 2$, there holds
\beno
\bigl((1-\f\e2 \cP_h^1  )f\,\big|\,f\bigr)_{L^2}
=\|f\|_{L^2}^2+\f{a_1\e}{2}\|h\p_xf\|_{L^2}^2\sim\|f\|_{X^0_\e}^2.
\eeno
Thanks to the ansatz $\cE_0=1$, \eqref{initial for Vt} and \eqref{initial for etat}, we deduce from \eqref{eq 22 for thm 2} that
\beq\label{eq 23 for thm 2}
\|V_{tt}|_{t=0}\|_{X^0_\e}^2\lesssim\cE_0^{\f12}\cdot\|V_{tt}|_{t=0}\|_{L^2}
+\e\bigl|\bigl(B_h(D)\eta_t\,\big|\,V_{tt}\bigr)_{L^2}\big|_{t=0}\bigr|.
\eeq

By virtue of the expression of $B_h(D)$ in \eqref{def of BC}, we have
\beno 
\bigl(B_h(D)\eta_t\,\big|\,V_{tt}\bigr)_{L^2}
=-b_1\bigl(\p_x(h^2\p_x\eta_t)\,\big|\,\p_x(\sqrt h V_{tt})\bigr)_{L^2}
+\bigl((b_2+b_3)h^{\f32}\p_xh\p_x^2\eta_t+r_1(h)\p_x\eta_t\,\big|\,V_{tt}\bigr)_{L^2},
\eeno
which along with the expression of $r_1(h)$ and \eqref{ansatz for h} yields to
\beno
\e\big| \bigl(B_h(D)\eta_t\,\big|\,V_{tt}\bigr)_{L^2}\big|_{t=0} \big|
\lesssim\e\|\eta_t|_{t=0}\|_{H^2}\|V_{tt}|_{t=0}\|_{H^1}\lesssim\|\eta_t|_{t=0}\|_{X^1_{\e}}\|V_{tt}|_{t=0}\|_{X^0_\e}.
\eeno 
Consequently, using \eqref{initial for etat}, we deduce from \eqref{eq 23 for thm 2} that
\beno 
\|V_{tt}|_{t=0}\|_{X^0_\e}^2\lesssim\cE_0^{\f12}\cdot\|V_{tt}|_{t=0}\|_{X^0_\e},
\eeno
which gives rise to
\beq\label{initial for Vtt}
\|V_{tt}|_{t=0}\|_{X^0_\e}^2\lesssim\cE_0.
\eeq

{\it Step 3.4. Construction for $\eta_{tt}|_{t=0}\in L^2(\R)$.} Due to the second equation of \eqref{Boussinesq-like, d=1}, we have
\beno 
\eta_{tt}|_{t=0}=\Bigl(-\p_x( \sqrt{h} V_t ) -\f\e2 C_h(D) V_t-\f\e2 \p_x \bigl( \frac{\eta_t V+\eta V_t}{\sqrt h} \bigr)\Bigr)\Big|_{t=0},
\eeno
which together with the expression of $C_h(D)$ in \eqref{def of BC}, ansatz \eqref{ansatz for h} and $\cE_0=1$ shows that
\beno\begin{aligned}
\|\eta_{tt}|_{t=0}\|_{L^2}&\lesssim\|V_t|_{t=0}\|_{H^1}+\e\|V_t|_{t=0}\|_{H^3}
+\e\|\eta_t|_{t=0}\|_{H^1}\|V_0\|_{H^1}+\e\|\eta_0\|_{H^1}\|V_t|_{t=0}\|_{H^1}\\ 
&\lesssim\|V_t|_{t=0}\|_{X^1_{\e^2}}+\e\|\eta_t|_{t=0}\|_{H^1}.
\end{aligned}\eeno
Using \eqref{initial for Vt} and \eqref{initial for etat}, we get
\beq\label{initial for etatt}
\|\eta_{tt}|_{t=0}\|_{L^2}^2\lesssim\cE_0.
\eeq

{\it Step 3.5. Initial energy functional.} Combining \eqref{initial for Vt}, \eqref{initial for etat}, \eqref{initial for Vtt} and \eqref{initial for etatt}, we obtain 
\beq\label{initial Et}
\cE(0)\lesssim \cE_0.
\eeq

{\bf Step 4. Final energy estimates and the closure of the continuity argument.} Going back to the {\it a priori} energy estimate \eqref{eq 1 for thm 2}, using \eqref{equiv functional 2} (i.e., $E(t)+\f\e2\wt{E}_2(t)\sim\cE(t)\sim E(t)$), there exists sufficiently small $\e_0\in(0,\min\{\e_1,C_0^{-1} \}]$ such that for any $\e\in(0,\e_0)$,
\beno
 \f{d}{dt}\bigl(E(t)+\f\e2\wt{E}_2(t)\bigr)
 \lesssim\e\cdot E(t)+\e\bigl(E(t)\bigr)^{\f32},\quad\forall\, 0<t\leq T_0/\e. 
\eeno
Using \eqref{equiv functional 2} again, we get
\beno 
E(t)\lesssim E(0)+\e t\cdot\max_{0\leq \tau\leq t} E(\tau)
+\e t\cdot\bigl(\max_{0\leq \tau\leq t} E(\tau)\big)^{\f32},\quad\forall\, 0<t\leq T_0/\e.
\eeno
Consequently, by virtue of \eqref{equiv functional}, \eqref{initial Et} and $\cE(0)\sim E(0)$, there exist constants $C_1,C_2>1$  independent of $T_0,C_0,\e$, such that
\beq\label{eq 24 for thm 2}
\cE(0)\leq C_1\cE_0,
\eeq
and 
\beq\label{eq 25 for thm 2}
\max_{0\leq\tau\leq t} \cE(\tau)\leq 2C_1\cE_0+C_2\e t\cdot\bigl(1+\bigl(\max_{0\leq \tau\leq t} \cE(\tau)\big)^{\f12}\bigr)\cdot\max_{0\leq \tau\leq t} \cE(\tau),\quad \forall\, 0<t\leq T_0/\e.
\eeq
Here \eqref{eq 24 for thm 2} guarantees that  \eqref{Boussinesq-like, d=1}-\eqref{initial data} admits solutions on a short time interval $[0,T]$, during which the ansatz \eqref{ansatz for energy} still holds. 

Taking $C_0=C_1$, using the ansatz \eqref{ansatz for energy}, we derive from \eqref{eq 25 for thm 2} that
\beno
\max_{0\leq\tau\leq t} \cE(\tau)\leq 2C_1\cE_0+C_2\e t\cdot\bigl(1+\sqrt{8C_1}\bigr)\cdot\max_{0\leq \tau\leq t} \cE(\tau),\quad \forall\, 0<t\leq T_0/\e.
\eeno
Taking $T_0=\f{1}{2C_2(1+\sqrt{8C_1})}$, we obtain
\beq\label{finial energy estimate} 
\max_{0\leq\tau\leq T_0/\e} \cE(\tau)\leq 4C_1\cE_0=4C_0\cE_0.
\eeq
This is the desired energy estimate \eqref{total energy estimate for case 2}.
It also improves the ansatz \eqref{ansatz for energy} so that the continuity argument is closed. This concludes   the proof of Theorem \ref{main results 2}.
\end{proof}

 \vspace{0.5cm}
			
\noindent {\bf Acknowledgments.}  The work of the second  author was partially  supported by the ANR project ISAAC (ANPG2023).
The work of the first and the third authors were partially supported by NSF of China under grants 12171019.

\end{document}